\documentclass[11pt]{article}

\usepackage{amsbsy,amsfonts,amsmath,amssymb,mathrsfs}
\usepackage[T1]{fontenc}
\usepackage{tikz}
\usetikzlibrary{cd,arrows,matrix,positioning}
\usepackage{hyperref}
\usepackage{cleveref}
\usepackage{comment}

\def\itemn#1{\item[\hspace{0.6mm} {\rm (#1)}]}
\def\itemm#1{\item[\indent {\rm (#1)}]}

\renewcommand{\geq}{\geqslant}
\renewcommand{\leq}{\leqslant}
\renewcommand{\ge}{\geqslant}

\newcommand{\sHom}{\sH\!om}
\def\red{{\rm red}}
\def\toto{\rightrightarrows}
\def\too{\longrightarrow}
\def\tooo{\relbar\joinrel\longrightarrow}
\def\into{\hookrightarrow}
\newcommand*{\intoo}{\ensuremath{\lhook\joinrel\relbar\joinrel\rightarrow}}
\def\onto{\twoheadrightarrow}
\def\isomto{\xrightarrow{\,\smash{\raisebox{-0.5ex}{\ensuremath{\scriptstyle\sim}}}\,}}
\def\To{\Rightarrow}
\renewcommand{\hat}{\widehat}
\newcommand{\eps}{\epsilon}

\makeatletter
\newcommand*{\defeq}{\mathrel{\rlap{%
                     \raisebox{0.3ex}{$\m@th\cdot$}}%
                     \raisebox{-0.3ex}{$\m@th\cdot$}}%
                     =}
\makeatother

\newtheorem{counter}[subsubsection]{$\!\!$}
\newtheorem{subcounter}[subsection]{$\!\!$}
\newenvironment{defi}{\begin{counter} \rm {\bf Definition.}}{\end{counter}}

\newenvironment{prop}{\begin{counter} {\bf Proposition.}}{\end{counter}}
\newenvironment{lemma}{\begin{counter} {\bf Lemma.}}{\end{counter}}
\newenvironment{coro}{\begin{counter} {\bf Corollary.}}{\end{counter}}
\newenvironment{theo}{\begin{counter} {\bf Theorem.}}{\end{counter}}

\newenvironment{rema}{\begin{counter} \rm {\bf Remark.}}{\end{counter}}
\newenvironment{remas}{\begin{counter} \rm {\bf Remarks.}}{\end{counter}}
\newenvironment{exam}{\begin{counter} \rm {\bf Example.}}{\end{counter}}

\newenvironment{noth}[1]{\begin{counter} {\bf #1.}\rm}{\end{counter}}
\newenvironment{noth*}[1]{\begin{subcounter} {\bf #1.}\rm}{\end{subcounter}}
\newenvironment{proof}{{\flushleft \bf Proof~:}}{\hfill $\square$ \vspace{5mm}}

\DeclareMathOperator{\sep}{sep}
\DeclareMathOperator{\pf}{pf}
\DeclareMathOperator{\copf}{copf}
\DeclareMathOperator{\Nil}{Nil}
\DeclareMathOperator{\Spec}{Spec}
\DeclareMathOperator{\Frob}{F}

\DeclareMathOperator{\et}{\acute{e}t}

\DeclareMathOperator{\colim}{colim}
\DeclareMathOperator{\Hom}{Hom}
\DeclareMathOperator{\Aut}{Aut}
\DeclareMathOperator{\Isom}{Isom}

\DeclareMathOperator{\id}{id}

\DeclareMathOperator{\pr}{pr}
\DeclareMathOperator{\swap}{sw}
\DeclareMathOperator{\gpd}{gpd}
\DeclareMathOperator{\coeq}{coeq}

\DeclareMathOperator{\aff}{aff}
\DeclareMathOperator{\surj}{surj}
\DeclareMathOperator{\dom}{dom}
\DeclareMathOperator{\Ass}{Ass}
\DeclareMathOperator{\Emb}{Emb}

\DeclareMathOperator{\Fdiv}{Fdiv}
\DeclareMathOperator{\GL}{GL}
\DeclareMathOperator{\Pro}{Pro}

\DeclareMathOperator{\eq}{eq}

\DeclareMathOperator{\Cat}{\bf Cat}

\DeclareMathOperator{\Sp}{\bf Sp}

\DeclareMathOperator{\Vect}{\bf Vect}
\DeclareMathOperator{\Perf}{\bf Perf}
\DeclareMathOperator{\Groupoid}{\bf Groupoid}
\DeclareMathOperator{\Pregroupoid}{\bf Pregroupoid}

\DeclareMathOperator{\FEt}{\bf F\acute{E}t}

\def\cO{{\cal O}}  
  
   \def\cY{{\cal Y}}
\def\cZ{{\cal Z}}

\newcommand\FF{\mathbb{F}}
\newcommand\GG{\mathbb{G}}

\newcommand\NN{\mathbb{N}}

\newcommand\QQ{\mathbb{Q}}

\def\sC{\mathscr{C}} \def\sD{\mathscr{D}} \def\sE{\mathscr{E}}
  \def\sH{\mathscr{H}}
\def\sI{\mathscr{I}} \def\sM{\mathscr{M}} \def\sN{\mathscr{N}}
 \def\sP{\mathscr{P}} \def\sQ{\mathscr{Q}}
\def\sR{\mathscr{R}} \def\sS{\mathscr{S}} \def\sU{\mathscr{U}}
\def\sX{\mathscr{X}} \def\sY{\mathscr{Y}} \def\sZ{\mathscr{Z}}

\def\sfE{{\sf E}}

\begin{document}

\begin{center}
{\bf \Large Unramified $\Frob$-divided objects
and the \'etale fundamental

\medskip

pro-groupoid in positive characteristic}

\bigskip
\bigskip

Yuliang Huang,
Giulio Orecchia,
Matthieu Romagny

\bigskip
\bigskip
\today
\end{center}

\bigskip

\begin{center}
\begin{minipage}{17.4cm}
{\small {\bf Abstract.} Fix a scheme $S$ of characteristic $p$. Let $\sM$
be an $S$-algebraic stack and let $\Fdiv(\sM)$ be the stack of
$\Frob$-divided objects, that is sequences of objects $x_i\in\sM$
with isomorphisms $\sigma_i:x_i\to \Frob^*x_{i+1}$.
Let $\sX$ be a flat, finitely presented $S$-algebraic stack and
$\sX\to \Pi_1(\sX/S)$ the \'etale fundamental pro-groupoid,
constructed in the present text. We prove that if $\sM$ is a
quasi-separated Deligne-Mumford stack and $\sX\to S$ has
geometrically reduced fibres, there is a bifunctorial isomorphism
of stacks $\sHom(\Pi_1(\sX/S),\sM) \simeq \sHom(\sX,\Fdiv(\sM))$.
In particular, the system of relative Frobenius morphisms
$\sX\to \sX^{p/S}\to \sX^{p^2/S}\to\dots$ allows to recover
the space of connected components $\pi_0(\sX/S)$ and
the relative \'etale
fundamental gerbe. In order to obtain these results, we study
the existence and properties of relative perfection for algebras
in characteristic $p$.}
\end{minipage}
\end{center}



\tableofcontents

\section{Introduction}

\begin{noth*}{\bf Motivation}
\addcontentsline{toc}{subsection}{\thesubsection. Motivation}
Using Cartier's Theorem on the descent of vector bundles under
Frobenius, Gieseker and Katz were able to give another viewpoint
on stratified vector bundles on a smooth variety of characteristic $p$.
Namely, they showed that these objects are equivalent to
{\em $\Frob$-divided vector bundles}, that is, sequences $\{E_i\}_{i\ge 0}$
of vector bundles and isomorphisms $E_i\simeq\Frob^*E_{i+1}$
where~$\Frob$ is the Frobenius endomorphism of the variety,
see~\cite{Gi75}. Since then, these have occupied an important place
in the research on vector bundles in characteristic~$p$. Looking
only at the recent literature, we can mention the works of
dos Santos~\cite{DS07}, \cite{DS11}, Esnault and Mehta \cite{EM10},
Berthelot~\cite{Be12}, Tonini and Zhang~\cite{TZ17}.

More generally, one can expect that in the study of curves,
or morphisms, or torsors (etc) in characteristic~$p$, the
$\Frob$-divided curves, morphisms, torsors (etc) are natural objects
which are likely to play an important role. In the present article,
for any algebraic stack $\sM$ we introduce the stack $\Fdiv(\sM)$ of
$\Frob$-divided objects of $\sM$ and we seek to understand it
(see Remark~\ref{rema:warning on notation} for a warning on notation).
Note that $\Frob$-divided vector bundles correspond to the case where
$\sM$ is the classifying stack $B\GL_n$, a typical example of Artin
stack with affine positive-dimensional inertia. In this article,
we study the somehow opposite case where~$\sM$ is a Deligne-Mumford
stack. In this case we call the objects of $\Fdiv(\sM)$
{\em unramified $\Frob$-divided objects}. Roughly speaking, our main
result says that unramified $\Frob$-divided objects defined over
geometrically reduced bases are quasi-isotrivial. In order to achieve
this, we establish various results on the perfection of algebras,
and on the coperfection of algebraic spaces and stacks, which have
independent interest. Let us now give more precise statements.
\end{noth*}

\begin{noth*}{\bf Quasi-isotriviality of unramified $\Frob$-divided objects}
\addcontentsline{toc}{subsection}{\thesubsection. Quasi-isotriviality
of unramified $\Frob$-divided objects}
Let $S$ be an algebraic space and $\sX\to S$ an algebraic stack.
Throughout the paper, we say that $\sX\to S$ is {\em separable}
if it has geometrically reduced fibres. When $\sX\to S$ is
flat, finitely presented and separable, we construct its
{\em \'etale fundamental pro-groupoid} $\sX\to\Pi_1(\sX/S)$. This is
a 2-pro-object of the 2-category of \'etale algebraic stacks, with
coarse moduli space the space of connected components $\pi_0(\sX/S)$,
see~\cite{Rom11}, seen as a constant 2-pro-object. When~$S$ is the
spectrum of a field~$k$ and~$\sX$ is geometrically connected,
the \'etale fundamental pro-groupoid $\Pi_1(\sX/S)$ is representable
in the 2-category of stacks by the \'etale fundamental gerbe
$\Pi^{\et}_{\sX/k}$ of Borne and Vistoli \cite[\S~8]{BV15}.
In characteristic~$p$, \'etale morphisms are perfect and it
follows that the natural map $\Fdiv(\Pi_1(\sX/S))\to \Pi_1(\sX/S)$
is an isomorphism.

\bigskip

\noindent {\bf Theorem A.}
{\em Let $S$ be a noetherian algebraic space of characteristic~$p$.
Let $\sX\to S$ be a flat, finitely presented, separable algebraic
stack. Let $\sM\to S$ be a quasi-separated Deligne-Mumford stack.
Then by applying $\Fdiv$ and precomposing with
$\sX\to\Pi_1(\sX/S)$, we obtain an isomorphism
\[
\sHom(\Pi_1(\sX/S),\sM) \isomto \sHom(\sX,\Fdiv(\sM))
\]
between the stacks of morphisms of pro-Deligne-Mumford stacks
(with $\sM$ seen as a constant 2-pro-object) on the source,
and morphisms of stacks on the target. This isomorphism is
functorial in $\sX$ and $\sM$.}

\bigskip

See~\ref{theo:coperfection_stacks}. Intuitively, this means that any
$\Frob$-divided object of $\sM$ over the base $\sX$ becomes
constant after \'etale surjective base changes on $S$ and on $\sX$,
i.e. is quasi-isotrivial in a suitable sense. Here is a simple
illustration. Let us assume that $X$ is a connected, simply
connected variety over a separably closed field $k$.
Then Theorem~A implies that all $\Frob$-divided families $C\to X$
of stable $n$-pointed
curves of genus $g$ with $2g-2+n>0$ are constant. The same assertion
with vector bundles replacing curves is the (almost exact) analogue
of Gieseker's conjecture, proved by Esnault and Mehta~\cite{EM10}.
However, Esnault and Mehta's situation and
ours are different in nature. In fact, in~{\em loc. cit.} as well as
in our work, the approach has two comparable steps. First one uses
the fact that objects are described by a morphism from a suitable
fundamental group(oid) scheme $\Pi$ (the \'etale fundamental
pro-groupoid for us, and the stratified fundamental group
scheme in~\cite{EM10}). Second one proves that under
the given assumptions, the group scheme $\Pi$ vanishes. The crucial
difference is that in our setting, the first step is the difficult part
of the argument and the second step is almost trivial, while for
Esnault and Mehta the first step is easy and the second step is where
all the effort lies.

If contemplated with a focus on $\sX$, Theorem~A gives
information on its coperfection. The viewpoint being substantially
different, it is worth giving the corresponding version of the
statement. For this we denote by $\sX^{p^i/S}$ the $i$-th
Frobenius twist of $\sX/S$ and
\[
\Frob_i:\sX^{p^i/S}\too \sX^{p^{i+1}/S}
\]
the relative Frobenius morphism.

\bigskip

\noindent {\bf Theorem A'.}
{\em Let $S$ be a noetherian algebraic space of characteristic~$p$.
\begin{trivlist}
\itemn{1} Let $X\to S$ be a flat, finitely presented,
separable morphism of algebraic spaces.
The inductive system of relative Frobenii
\[
\begin{tikzcd}[column sep=8mm]
X \ar[r, "\Frob_0"] & X^{p/S} \ar[r, "\Frob_1"] &
X^{p^2/S} \ar[r] & \dots
\end{tikzcd}
\]
admits a colimit in the category of algebraic spaces over $S$.
This colimit is the algebraic space of connected components
$\pi_0(X/S)$; it is a coperfection of $X\to S$.
\itemn{2} Let $\mathscr{X}\to S$ be a flat, finitely presented,
separable algebraic stack. The inductive system of relative
Frobenii
\[
\begin{tikzcd}[column sep=8mm]
\sX \ar[r, "\Frob_0"] & \sX^{p/S} \ar[r, "\Frob_1"] &
\sX^{p^2/S} \ar[r] & \dots
\end{tikzcd}
\]
admits a 2-colimit in the 2-category of pro-quasi-separated
Deligne-Mumford stacks over $S$. This 2-colimit is the
pro-\'etale stack $\Pi_1(\mathscr{X}/S)$;
it is a 2-coperfection of $\sX/S$ in the 2-category
of pro-Deligne-Mumford stacks.
\end{trivlist}}

\bigskip

See Remarks~\ref{rem:on coperfection 1} and
\ref{rem:on coperfection 2}. Statement~(2) is equivalent
to Theorem~A as explained in Remark~\ref{rema:coperf determines perf}.
Note that~(2) includes~(1) as a special case,
because $\Pi_1(\mathscr{X}/S)$ has coarse moduli space $\pi_0(\sX/S)$.
We include~(1) for emphasis and also because the proof actually
proceeds by deducing~(2) from~(1).

Theorem A' seems to suggest that taking coperfection in the higher
category of pro-Deligne-Mumford $n$-stacks
would eventually recover the whole relative \'etale homotopy type
of $X\to S$. We plan to investigate this eventuality in a future
article.
\end{noth*}

\begin{noth*}{Perfection of algebras; largest \'etale subalgebras}
\addcontentsline{toc}{subsection}{\thesubsection. Perfection
of algebras; largest \'etale subalgebras}
Within the category of algebras, the situation is somehow
more subtle. Given a characteristic~$p$ ring $R$ and an algebra
$R\to A$, denote by
\[
\Frob_i:A^{p^{i+1}/R}\to A^{p^i/R}
\]
the relative Frobenius of $A^{p^i/R}$, the $i$-th
Frobenius twist of $A$. Define the {\em preperfection}
of $A/R$:
\[
A^{p^{\infty}/R}=
\lim\,\big(\!\!
\begin{tikzcd}[column sep=8mm]
\cdots \ A^{p^2/R} \ar[r, "\Frob_1"] &
A^{p/R} \ar[r, "\Frob_0"] & A
\end{tikzcd}
\!\!\big).
\]
The name is explained by a surprising fact: the algebra
$A^{p^{\infty}/R}$ is not perfect in general, even if $R\to A$
is flat, finitely presented and separable. We give an example
of this with $R$ equal to the local ring of a nodal curve
singularity (see~\ref{example_non_perfect}). In our example
the double preperfection is perfect but we do not know if
iterated preperfections should converge to a perfect
algebra in general. In the affine case $S=\Spec(R)$ and
$X=\Spec(A)$, we write $\pi_0(A/R)$ instead of $\pi_0(X/S)$.
What Theorem~A' implies in this case is that there is an
isomorphism of $R$-algebras:
\[
\cO(\pi_0(A/R))\isomto A^{p^{\infty}/R}.
\]
Here $\cO(-)$ is the functor of global functions. Given the bad
properties of the rings under consideration, this could not really
be anticipated: indeed, in general $\cO(\pi_0(A/R))$ is not
\'etale and $A^{p^{\infty}/R}$ is not perfect. Although we present
the above isomorphism of $R$-algebras as a corollary to Theorem~A',
the structure of the proof is actually to first establish
this isomorphism of algebras (see~\ref{theo:preperfection_algebra})
and then deduce the geometric statement for spaces and stacks
(Theorem~A').

This begs for a further study of perfection of algebras.
Our general expectation is that for algebras of finite type,
there should exist a largest \'etale subalgebra and this should
be (at least close to) the perfection of $R\to A$. In striving
to materialize this picture, we study \'etale hulls in more detail.
We take up recent work of Ferrand~\cite{Fe19} and prove the
following result which is not special to characteristic $p$.

\bigskip

\noindent {\bf Theorem B.}
{\em Let $S$ be a noetherian, geometrically unibranch algebraic
space without embedded points. Let $f:X\to S$ be a faithfully flat,
finitely presented morphism of algebraic spaces.
\begin{trivlist}
\itemn{1} The category of factorizations $X\to E\to S$ such that
$X\to E$ is a schematically dominant morphism of algebraic spaces
and $E\to S$ is \'etale and affine is a lattice, that is, any two
objects have
a supremum and an infimum (for the obvious relation of domination).
Moreover it has a largest element $\pi^a(X/S)$.
\itemn{2} The functor $X\mapsto \pi^a(X/S)$ is left adjoint to the inclusion
of the category of \'etale, affine $S$-schemes into the category
of faithfully flat, finitely presented $S$-algebraic spaces.
\end{trivlist}}

\bigskip

See Theorem~\ref{theo_affine_etale_hull}
and Corollary~\ref{lemma:fonctoriality_for_pi^a}.
The largest element $\pi^a(X/S)$ is the relative spectrum of
a sheaf of $\cO_S$-algebras which is the largest \'etale
subalgebra of $f_*\cO_X$. It is called the {\em \'etale affine
hull} of $X\to S$. When $S$ is artinian or $X\to S$ is
separable, the functor $\pi_0(X/S)$ is an \'etale algebraic
space and we have morphisms:
\[
X\too \pi_0(X/S)\too \pi^a(X/S).
\]
We can take advantage of this to analyze perfection of algebras
in characteristic $p$. When $S=\Spec(R)$ and $X=\Spec(A)$,
the largest \'etale subalgebra is written $A^{\et/R}\subset A$,
that is $\pi^a(A/R)=\Spec(A^{\et/R})$. We then obtain the
following positive results.

\bigskip

\noindent {\bf Theorem C.}
{\em Let $R\to A$ be a flat, finite type morphism of noetherian
rings of characteristic $p$.
Assume that one of the following holds:
\begin{trivlist}
\itemn{1} $R$ is artinian,
\itemn{2} $R$ is geometrically $\QQ$-factorial (e.g. regular)
and $R\to A$ is separable,
\itemn{2} $R$ is one-dimensional, reduced, geometrically unibranch,
and $R\to A$ is separable.
\end{trivlist}
Then the natural maps give rise to isomorphisms:
\[
A^{\et/R}\isomto \cO(\pi(A/R)) \isomto A^{p^\infty/R}.
\]}
\end{noth*}

See Theorem~\ref{theorem:artinian base case},
Proposition~\ref{prop:set_theoretically_factorial_etale} and
Corollary~\ref{coro:ideal_situation_regular_case}.

\begin{noth*}{Overview of the paper and notations}
\addcontentsline{toc}{subsection}{\thesubsection. Overview of the paper}
Each section starts with a small description of contents,
where the reader will find more detail.
In Section~\ref{section:coperfection} we give definitions and
basic facts on perfect stacks, perfection and coperfection.
In Section~\ref{section:complements_pi0} which makes no assumption
on the characteristic, we give complements on the functor $\pi_0$.
We study factorizations through an \'etale affine scheme, and we
prove Theorem~B. Finally we prove two pushout results that
allow to view $\pi_0(X/S)$ as glued from simpler pieces (the simpler
pieces being either~$\pi_0$ of an atlas or a completion from
a closed fibre), to be used in the last two sections. In
Section~\ref{section:perf_of_algebras}
where we study the commutative algebra of perfection, proving
the results summarized in Theorem~C. Finally in
Section~\ref{section:coperfection_spaces_stacks} we prove
Theorems~A and A', first for algebraic spaces and then for
algebraic stacks.

All sheaves and stacks are considered for the fppf topology
unless explicitly stated otherwise. We denote sets, sheaves
and stacks of homomorphisms by the symbols $\Hom$,
$\underline\Hom$ and $\sHom$ respectively.
\end{noth*}


\begin{noth*}{Acknowledgements}
\addcontentsline{toc}{subsection}{\thesubsection. Acknowledgements}
We express warm thanks to Daniel Ferrand for his detailed reading
of Section~\ref{section:complements_pi0} and his successful efforts
to dissuade us from trying to prove the existence of the \'etale
affine hull in exceeding generality.
We are grateful to Fabio Tonini and Lei Zhang for enlightening
discussions on the\'etale fundamental pro-groupoid. We also thank
Niels Borne, Johann Haas and Angelo Vistoli for kind answers to our
questions. The three authors are supported by the Centre Henri
Lebesgue, program ANR-11-LABX-0020-01 and would like to thank the
executive and administrative staff of IRMAR and of the Centre Henri Lebesgue
for creating an attractive mathematical environment.
\end{noth*}

\section{Perfection and coperfection}
\label{section:coperfection}

Throughout this section, we let $S$ be an algebraic space of
characteristic~$p$.
Our purpose is to make some preliminary
remarks on perfection and coperfection: definitions and formal
properties (\ref{ss:cat definitions} and \ref{ss:base restriction}),
description in the 2-category of stacks (\ref{ss:case of stacks}),
and structure of perfect algebraic stacks (\ref{ss:case of alg stacks}).

There is unfortunately no uniform use of the word ``perfection'' in
the literature. Our convention is to call {\em perfection} resp.
{\em coperfection} the right adjoint, resp. the left adjoint, to
the inclusion of the full subcategory of perfect objects in the
ambient category. This choice is prompted by the fact that in most
cases of existence, the construction of perfections uses limits
while the construction of coperfections uses colimits. For example,
this is the way one can form the perfection $A^{\pf}$ and the
coperfection $A^{\copf}$ of an $\FF_p$-algebra $A$ with absolute
Frobenius $\Frob_{\!A}$:
\[
A^{\pf}=
\lim\big(\!\!
\begin{tikzcd}[column sep=7mm]
\cdots \ A \ar[r, "\Frob_{\!A}"] & A \ar[r, "\Frob_{\!A}"] & A
\end{tikzcd}
\!\!\big);\qquad
A^{\copf}=
\colim\big(\!\!
\begin{tikzcd}[column sep=7mm]
A \ar[r, "\Frob_{\!A}"] & A \ar[r, "\Frob_{\!A}"] & A \ \cdots
\end{tikzcd}
\!\!\big).
\]
We emphasize that our interest is in perfection of algebras, and
coperfection of algebraic spaces and stacks. This means
that our setting is {\em relative} (over a possibly imperfect base)
and {\em geometric} (with schemes, spaces and stacks). Both
features introduce difficulties; we do not know if
perfection of algebras and coperfection of algebraic spaces and
stacks exist in full generality.

\subsection{Categorical definitions}
\label{ss:cat definitions}

\begin{noth}{Frobenius, perfect objects}
Let $f:X\to S$ be a fibred category over $S$ and let
$X^{p/S}=X\times_{S,\Frob_S} S$ be its Frobenius twist.
The {\em absolute Frobenius} is the functor $\Frob_X:X\to X$
defined by $\Frob_X(x)=\Frob_T^*x$, for all $T/S$ and $x\in X(T)$.
The {\em relative Frobenius} is the functor
$\Frob_{X/S}\defeq (\Frob_X,f):X\to X^{p/S}$. Note
that $\Frob_X$ is not a morphism of fibred categories over $S$
while $\Frob_{X/S}$ is. We say that $X\to S$ is {\em perfect}
if $\Frob_{X/S}$ is an isomorphism of fibred categories.
\end{noth}

\begin{noth}{Perfection and coperfection}
Let $C$ be a fibred 2-category over $S$ whose objects are fibred
categories of the type just discussed. We write $\sHom_C(X,Y)$
the categories of morphisms in $C$, and $\Hom_C(X,Y)$ the object
sets of the latter. The objects $X\in C$
which are perfect form a full 2-subcategory $\Perf(C)$ whose
inclusion we denote $i:\Perf(C)\to C$. Now let $X\in C$ be any
object. If the functor $\Perf(C)\to\Cat$, $P\mapsto\sHom_C(iP,X)$
is 2-representable then we call the representing object the
{\em 2-perfection of $X$} and denote it $X^{\pf}$. If the functor
$\Perf(C)^\circ\to\Cat$, $P\mapsto\sHom_C(X,iP)$ is 2-representable
then we call the representing object the {\em 2-coperfection of $X$}
and denote it $X^{\copf}$. We often simply that {\em perfection}
and {\em coperfection} for simplicity. Hence, if all objects have perfections
(resp. coperfections) then the functor $X\to X^{\pf}$
(resp. the functor $X\to X^{\copf}$) is right (resp. left) adjoint
to the inclusion $i$. Note that if a given $X$ of interest may be
seen as an object of different fibred 2-categories $C$ and $C'$,
then its hypothetical perfections in $C$ and $C'$ differ in
general, and similarly for its hypothetical coperfections.
\end{noth}

\begin{noth}{Cofibred setting}
While algebraic spaces and stacks and the 2-categories that contain
them fall under the scope of the ``fibred'' categorical setting, algebras
and the categories that contain them live in the ``cofibred'' categorical
setting. The cofibred analogues of the notions just presented exist
with the obvious modifications; notably, for a cofibred category
$A\to S$, the relative Frobenius is a functor
$\Frob_{A/S}:A^{p/S} \to A$. In this setting, perfection (resp.
coperfection) is again defined as the right (resp. left) adjoint
of the inclusion of perfect objects.
\end{noth}

\begin{noth}{Formal properties of Frobenius}
\label{Formal properties}
If $f:X\to Y$ is a morphism of fibred categories over~$S$, we can
define $X^{p/Y}\defeq X\times_{Y,\Frob_Y} Y$ and relative Frobenius
$\Frob_{X/Y}\defeq (\Frob_X,f):X\to X^{p/Y}$. We say that $f$ is
{\em (relatively) perfect} if $\Frob_{X/Y}$ is an isomorphism.
If $g:Y\to Z$ is another morphism of fibred categories over $S$,
we have $\Frob_{X/Z}=(f^{p/Z})^*\Frob_{Y/Z}\circ\Frob_{X/Y}$
as one can see from the diagram with cartesian squares:
\[
\begin{tikzcd}[column sep=20mm]
X \arrow[r, "\Frob_{X/Y}"]
\arrow[rr, bend left, "\Frob_{X/Z}",out=60,in=120,distance=1cm] &
X^{p/Y} \arrow[r, "(f^{p/Z})^*\Frob_{Y/Z}"]
\arrow[d] \arrow[rd, "\square",phantom] &
X^{p/Z} \arrow[r] \arrow[d,"f^{p/Z}"]
\arrow[rd, "\square",phantom] & X \arrow[d, "f"] \\
& Y \arrow[r, "\Frob_{Y/Z}"'] &
Y^{p/Z} \arrow[r] \arrow[d] \arrow[rd, "\square",phantom] & Y \arrow[d] \\
& & Z \arrow[r, "\Frob_Z"'] & Z
\end{tikzcd}
\]
Using these remarks, one checks the following facts:
\begin{itemize}
\item[(i)] perfect morphisms are stable by base change;
\item[(ii)] perfect morphisms are stable by composition;
\item[(iii)] morphisms between fibred categories perfect over $Z$
are perfect over $Z$;
\item[(iv)] if $X\to Y$, $X\to Z$ are perfect and $f^{p/Z}$
descends isomorphisms (e.g. $f$ is a perfect and faithfully flat
quasi-compact morphism of algebraic stacks), then $Y\to Z$ is perfect.
\end{itemize}
\end{noth}

\subsection{Base restriction}
\label{ss:base restriction}

For the sake of simplicity, let us come back to algebraic spaces.
Let $f:S'\to S$ be a morphism of algebraic spaces. The
{\em base restriction along $f$} is the functor that sends an
$S'$-algebraic space $X'$ to the $S$-algebraic space $X'\to S'\to S$.
We denote by $f_!X'$ the base restriction. The functor $f_!$ is
left adjoint to the pullback $f^*$. It should not be confused with
the Weil restriction functor $f_*$ which is right adjoint to $f^*$.
We will need to use the fact that coperfection commutes with base
restriction. This is a consequence of the simple categorical fact
that if two functors commute and have left adjoints, then the left
adjoints commute. Here is a precise statement in our context.

\begin{lemma} \label{lemma:coperf and base restr}
Let $X,T,S$ be $\FF_p$-algebraic spaces. Let $f:T\to S$ be a
morphism which is relatively perfect, and $X\to T$ a morphism
which admits a coperfection $X^{\copf}$. Then $f_!(X^{\copf})$
is a coperfection for $f_!X$. In a formula, we obtain an isomorphism:
\[
f_!(X^{\copf}) \isomto (f_!X)^{\copf}.
\]
\end{lemma}

\begin{proof}
Let $\Sp_S$ be the category of $S$-algebraic spaces, and
$i_S:\Perf_S\to\Sp_S$ the inclusion of perfect objects.
Since $f:T\to S$ is relatively perfect and relatively perfect morphisms
are stable by composition, the functor $f_!$ maps $\Perf_T$
into $\Perf_S$, that is, it commutes with $i_S$ and $i_T$.
Similarly $f^*$ maps $\Perf_S$ into $\Perf_T$.
For each $Y\in\Perf_S$ we have canonical bijections:
\begin{align*}
\Hom_{\Sp_S}(f_!X,i_SY)
& =\Hom_{\Sp_T}(X,f^*i_SY) \\
& =\Hom_{\Sp_T}(X,i_Tf^*Y) \\
& =\Hom_{\Perf_T}(X^{\copf},f^*Y) \\
& =\Hom_{\Perf_S}(f_!X^{\copf},Y).
\end{align*}
This shows that $f_!X^{\copf}$ is the coperfection of $f_!X$.
\end{proof}

The same result holds, with the same proof, for pairs of commuting
adjoints in similar situations.

\subsection{The case of stacks; $\Frob$-divided objects}
\label{ss:case of stacks}

In this section we describe concretely the perfection and
coperfection of fppf stacks over $S$, and highlight some properties.
As we said in the introduction, all sheaves and stacks are
considered for the fppf topology so most of the time we omit
the adjective.

\begin{noth}{Coperfection of stacks}
\label{noth:coperfection}
Let $\sX$ be a stack over $S$. We let
\[
\sX^{\copf/S}=\colim
\big(\!\!
\begin{tikzcd}[column sep=8mm]
\sX \ar[r, "\Frob_0"] & \sX^{p/S} \ar[r, "\Frob_1"] &
\sX^{p^2/S}\dots
\end{tikzcd}
\!\!\big)
\]
be the colimit in the 2-category of stacks. The inductive system
being filtered, the prestack colimit satisfies the stack property
for coverings of affine schemes $\Spec(A')\to\Spec(A)$, and its
Zariski stackification is an fppf stack, hence is the fppf
stackification. One checks the following facts:
\begin{itemize}
\item[(i)] $\sX^{\copf/S}$ is perfect and is a coperfection of
$\sX$ in the 2-category of $S$-stacks;
\item[(ii)] the formation of $\sX^{\copf/S}$ commutes with all
base changes $S'\to S$ and is functorial in $\sX$;
\item[(iii)] $\sX^{\copf/S}$ is locally of finite presentation
(that is, limit-preserving) if $\sX$ is;
\item[(iv)] if $\sX$ is an algebraic stack, then $\sX^{\copf}$
is far from algebraic in general. For example if~$\sX$ is the
affine line over $\FF_p$ then for an $\FF_p$-algebra $A$, the set
$\sX^{\copf}(A)$ is equal to $A^{\copf/\FF_p}$, the absolute
coperfection of $A$. In particular, for $A=\FF_p[[t]]$ the set
$\sX^{\copf}(A)=\FF_p[[t^{p^{-\infty}}]]$ is much bigger than
$\lim \sX^{\copf}(A/t^n)=\FF_p$.
\end{itemize}
\end{noth}

\begin{noth}{Perfection of stacks; $\Frob$-divided objects}\label{noth:perfection}
Let $\sM$ be a  stack over $S$.
For each $i\ge 0$ let $\Frob_{S,*}^i$ be the Weil restriction
along the $i$-th absolute Frobenius of $S$, and
\[
G_i:\Frob_{S,*}^{i+1}\sM\to \Frob_{S,*}^i\sM
\]
the morphism which maps a $T$-valued object
$x\in \sM(T^{p^{i+1}/S})=(\Frob_{S,*}^{i+1}\sM)(T)$
to the pullback
\[
G_i(x)\defeq \Frob_{T^{p^i/S}/S}^*x
\]
under the Frobenius
$\Frob=\Frob_{T^{p^i/S}/S}:T^{p^i/S}\to T^{p^{i+1}/S}$.
Then we define:
\[
\sM^{\pf/S}=\lim
\big(\!\!
\begin{tikzcd}[column sep=8mm]
\dots \ar[r] & \Frob_{S,*}^2\sM \ar[r, "G_1"] &
\Frob_{S,*}\sM \ar[r, "G_0"] & \sM
\end{tikzcd}
\!\!\big),
\]
the limit being taken in the 2-category of stacks.
One has the following facts:
\begin{itemize}
\item[(i)] $\sM^{\pf/S}$ is perfect and is a perfection of
$\sX$ in the 2-category of $S$-stacks;
\item[(ii)] the formation of $\sM^{\pf/S}$ commutes with all base
changes $S'\to S$ and is functorial in $\sM$;
\item[(iii)] $\sM^{\pf/S}$ is not locally of finite presentation
in general, even if $\sM$ is;
\item[(iv)] assume that $\Frob_S:S\to S$ is finite locally free.
If $\sM$ is a Deligne-Mumford stack, then $\sM^{\pf/S}$ also.
For schemes, this is proven in Kato~\cite{Ka86}, Prop.~1.4.
In general, one uses the fact that the diagonal being unramified,
its relative Frobenius is a monomorphism, hence the transitions
$G_i:\Frob_{S,*}^{i+1}\sM\to \Frob_{S,*}^i\sM$ are representable
affine morphisms. For Artin stacks, the same argument proves that
the diagonal of $\sM^{\pf/S}$ is representable by algebraic spaces,
but in general it is not locally of finite type and $\sM^{\pf/S}$
is not algebraic. For instance, in the case of $\sM=B\GG_m$ over
$S=\Spec(\FF_p)$ the diagonal is a torsor under
$\mu_{p^\infty}=\lim \mu_{p^i}$. Finally if $\Frob_S$ is not finite
locally free, then already the diagonal may fail to be representable;
\item[(v)] If $\sM'\to \sM$ is perfect, the natural morphism
$\sM'^{\pf/S}\to \sM^{\pf/S}\times_\sM\sM'$ is an isomorphism
of stacks.
\end{itemize}

\begin{rema} \label{rema:coperf determines perf}
For arbitrary $S$-stacks $\sX$ and $\sM$, we have canonical
isomorphisms:
\[
\sHom(\sX,\sM^{\pf})=\sHom(\sX^{\copf},\sM^{\pf})=
\sHom(\sX^{\copf},\sM).
\]
This equality is what explains the dual interpretation of our
result embodied by Theorems~A and A' in the introduction. Indeed,
assume we have a satisfactory understanding of the above object as
a bifunctor in~$\sX$ and $\sM$. Then letting $\sX$ vary we obtain a
description of the perfection of $\sM$, while letting $\sM$ vary we
obtain a description of the coperfection of $\sX$.
Going still further, since $T^{\copf}=\colim T^{p^i/S}$ we have
\begin{align*}
\sM^{\pf}(T) & = \sHom(\colim T^{p^i/S},\sM)
=\lim \sHom(T^{p^i/S},\sM) \\
& =\lim \sHom(T,F^i_{S,*}\sM)
=\sHom(T,\lim F_{S,*}^i\sM)=(\lim F_{S,*}^i\sM)(T).
\end{align*}
This shows that once we know coperfection in the 2-category
of stacks, the construction of the perfection is forced upon us.
\end{rema}

The points of the stack $\sM^{\pf/S}$ are exactly the
$\Frob$-divided objects of $\sM$. We want to give the latter an
existence of their own, independent of the adjointness property.

\begin{defi}
We denote by $\Fdiv_S(\sM)$ the stack described as follows.
\begin{trivlist}
\itemn{1} An {\em $\Frob$-divided object of $\sM$} over an
$S$-scheme $T$ is a collection of pairs
$(x_i,\sigma_i)_{i\ge 0}$ where $x_i\in\sM(T^{p^i/S})$ and
$\sigma_i:x_i\to \Frob^*x_{i+1}$ is an isomorphism; here
$\Frob=\Frob_{T^{p^i/S}/S}:T^{p^i/S}\to T^{p^{i+1}/S}$
is Frobenius.
\itemn{2} A morphism between $(x_i,\sigma_i)_{i\ge 0}$ and
$(y_i,\tau_i)_{i\ge 0}$ is a collection of morphisms $u_i:x_i\to y_i$
such that $\tau_i\circ u_i=\Frob^*u_{i+1}\circ\sigma_i$ for all
$i\ge 0$.
\end{trivlist}
\end{defi}

To make things clear: $\Fdiv_S(\sM)$ and $\sM^{\pf/S}$ are really
two names for the same object.

\begin{rema} \label{rema:warning on notation}
In most of the existing literature, e.g. \cite{DS07}, \cite{TZ17},
the notation $\Fdiv(\cZ)$ is used for the category of $\Frob$-divided
vector bundles on $\cZ$. In Tonini and Zhang \cite{TZ17}, Def.~6.20,
the notation is
extended to the effect that $\Fdiv(\cZ,\cY)$ denotes the category
of $\Frob$-divided objects of a stack $\cY$ over the base $\cZ$. In the
present paper, our emphasis is on the stack where divided objects
take their values rather than the base that supports them. We are
therefore led to drop~$\cZ$ from
the notation, so that our $\Fdiv(\sM)$ is Tonini and Zhang's
$\Fdiv(-,\sM)$. We warn the reader that as a result, the notation
$\Fdiv(\sM)$ does not have the same meaning in both works.
Writing $\Vect$ for the stack of vector bundles, the
following table gives a summary of the correspondence of notations.
\smallskip
\begin{center}
\setlength\arraycolsep{.5mm}
\renewcommand{\arraystretch}{1.3}
\begin{tabular}{|c|c|}
\hline
Our notation & Notation in \cite{TZ17} \\
\hline
$\Fdiv(\sM)$ & $\Fdiv(-,\sM)$ \\
$\sM^{\pf/S}$ & $\Fdiv(-,\sM)$ \\
$\Fdiv(\Vect)(T)$ & $\Fdiv(T)$ \\
$\sX^{\copf/S}$ & $\sX^{(\infty,S)}$ \\
\hline
\end{tabular}
\end{center}
\end{rema}
\end{noth}

\bigskip

We end this subsection with a lemma which is a consequence
of the fact that the diagonal of the perfection is the
perfection of the diagonal. This will be useful
later in Section~\ref{section:coperfection_spaces_stacks}.

\begin{lemma}\label{lemma:isoms-Fdiv}
Let $S$ be an algebraic space of characteristic $p$ and
$\sY\to S$ a perfect stack. Let $\sM$ be a stack
and $f\colon \Fdiv_S(\sM)\to \sM$ the perfection morphism. Let
$x,y\colon \sY \to \Fdiv_S(\sM)$ be two morphisms, and write
$x_0,y_0\colon \sY\to \sM$ for the compositions $fx$, $fy$.
Then there is an isomorphism of sheaves on $S$
\[
\underline\Hom\left(x,y\right)\isomto
\Fdiv_S(\underline\Hom\left(x_0,y_0\right))
\]
identifying the morphism
\[
\underline\Hom(x,y)\too \underline\Hom(x_0,y_0)
\]
with the $S$-perfection morphism.
\end{lemma}

\begin{proof}
As $\sM^{\pf/S}=\Fdiv_S(\sM)$ is defined as a limit, the formation
of $\Fdiv_S$ commutes with pro\-ducts, and the natural equivalence
$\Fdiv_S(\sM)\times_S\Fdiv_S(\sM)\isomto \Fdiv_S(\sM\times_S\sM)$
identifies the diagonal $\Delta_{\Fdiv_S(\sM)}$ with
$\Fdiv_S(\Delta_{\sM})$.
We have a 2-cartesian diagram  of stacks on $S$:
\[
\begin{tikzcd}
\underline\Hom\left(x,y\right) \ar[r]\ar[d] &
\sY \ar[d, "{(x,y)}"] \\
\Fdiv_S(\sM) \ar[r,"\Delta"] & \Fdiv_S(\sM)\times_S\Fdiv_S(\sM).
\end{tikzcd}
\]
Because $\sY\to S$ is perfect, the morphism $\Fdiv_S(\sY)\to \sY$
is an isomorphism of stacks. Applying $\Fdiv_S$ to the
2-cartesian diagram
\[
\begin{tikzcd}
\underline\Hom\left(x_0,y_0\right) \ar[r]\ar[d] &
\sY \ar[d, "{(x_0,y_0)}"] \\
\sM \ar[r,"\Delta"] & \sM\times_S\sM
\end{tikzcd}
\]
we obtain the desired isomorphism
$\underline\Hom\left(x,y\right)\isomto
\Fdiv_S(\underline\Hom\left(x_0,y_0\right))$.
\end{proof}

\subsection{Perfect algebraic stacks}
\label{ss:case of alg stacks}

Perfect algebraic stacks have a very simple structure.

\begin{lemma}\label{lemma:perfect_stack}
Let $\sX$ be an algebraic stack over $S$. Consider the following
conditions:
\begin{trivlist}
\itemn{1} $\sX$ is a perfect $S$-stack.
\itemn{2} There exists an \'etale, surjective morphism
$U\to \sX$ from a perfect $S$-algebraic space.
\itemn{3} $\sX$ is an \'etale gerbe over a perfect
$S$-algebraic space.
\end{trivlist}
Then we have the implications
{\rm (1)} $\iff$ {\rm (2)} $\Longleftarrow$ {\rm (3)},
and if the diagonal of $\sX\to S$ is locally of finite
presentation then all three conditions are equivalent.
In particular, all perfect algebraic stacks are Deligne-Mumford.
\end{lemma}

To obtain an example of a perfect algebraic stack that does
not satisfy (3), take a positive-dimensional scheme $X$ over
a perfect field $k$ with a non-free action of a finite group $G$,
and let $\sX=[X^{\pf/k}/G]$.

\begin{proof}
We use the facts collected in~\ref{Formal properties} without explicit
mention.

\smallskip

\noindent (1) $\Rightarrow$ (2) If $\sX\to S$ is perfect,
then so is
$\sX\times_S\sX\to S$ and hence also the diagonal
$\Delta:\sX\to\sX\times_S\sX$. In particular $\Delta$ is
formally unramified. Being locally of finite type
(\cite{SP19}, Tag~\href{http://stacks.math.columbia.edu/tag/04XS}{04XS}), it is unramified in the sense of \cite{Ra70}
and \cite{SP19}. It follows that $\sX$ is Deligne-Mumford
(\cite{SP19}, Tag~\href{http://stacks.math.columbia.edu/tag/06N3}{06N3}). Let $U\to \sX$ be an \'etale surjective morphism
from an algebraic space; then $U\to \sX$ is perfect and it
follows that $U\to S$ is perfect.

\smallskip

\noindent (2) $\Rightarrow$ (1) By~\ref{Formal properties},
if $U$ is perfect and $U\to \sX$ is \'etale surjective
then $\sX$ is perfect.

\smallskip

\noindent (3) $\Rightarrow$ (1) This is clear because an \'etale
gerbe is perfect.

\smallskip

\noindent (1) $\Rightarrow$ (3) If $\Delta$ is locally
of finite presentation, it is
formally \'etale hence \'etale. It follows that the inertia
stack $I_{\sX}\to\sX$ is \'etale and therefore there is an
algebraic space $X$ and an \'etale gerbe morphism
$\sX\to X$, see
\cite{SP19}, Tag~\href{http://stacks.math.columbia.edu/tag/06QJ}{06QJ}.
\end{proof}

\section{\'Etale hulls and connected components}
\label{section:complements_pi0}

In this section, we provide some complements on the functor
$\pi_0$ introduced in~\cite{Rom11}. Although these results hold
for algebraic stacks, we restrict most of the time to algebraic
spaces because this simplifies the treatment a little and is
enough for our needs. There are two viewpoints
on the functor $\pi_0$, and we consider both.

Firstly $\pi_0$ is a left adjoint to the inclusion of the
category of \'etale quasi-compact spaces in the category
of flat, finitely presented, separable spaces. In the study of
such ``\'etalification'' functors, Ferrand~\cite{Fe19} recently
highlighted the
importance of the category of factorizations $X\to E\to S$ where
the second arrow is \'etale. He proved that when the base $S$
has finitely many irreducible components, there is a left
adjoint~$\pi^s$ to the inclusion
of \'etale, separated spaces into all flat, finitely presented
spaces. In \S~\ref{subsection:etale affine hulls}
we prove that the category of factorizations as well as
some interesting subcategories satisfy topological invariance
(in the sense of \cite{SGA4.2}, Exp.~VIII, Th.~1.1). Then we
prove that when~$S$ is noetherian, geometrically unibranch and
without embedded points, there is a left adjoint~$\pi^a$ to the
inclusion of \'etale, affine spaces into all flat,
finitely presented spaces. In \S~\ref{subsection_global_section}
we compare~$\pi^a$ with the affine hull of $\pi_0$.

Secondly $\pi_0$ is the functor of connected components of a
relative space.
In \S~\ref{subsection:computing pi0} we describe ways to compute
$\pi_0(X/S)$ by using an atlas of $X$, or completing along a
closed fibre of $X\to S$.

We sometimes impose some
finiteness or regularity assumptions on the base $S$, but
nothing on the characteristics; it is only in later sections
that we specialize to characteristic $p$.

\subsection{\'Etale affine hulls and largest \'etale subalgebras}
\label{subsection:etale affine hulls}

Let us briefly recall what is know on \'etale hulls, also
called \'etalification functors. Consider the following diagram
of fully faithful subcategories of the category of $S$-spaces
(``fp'' stands for finitely presented):
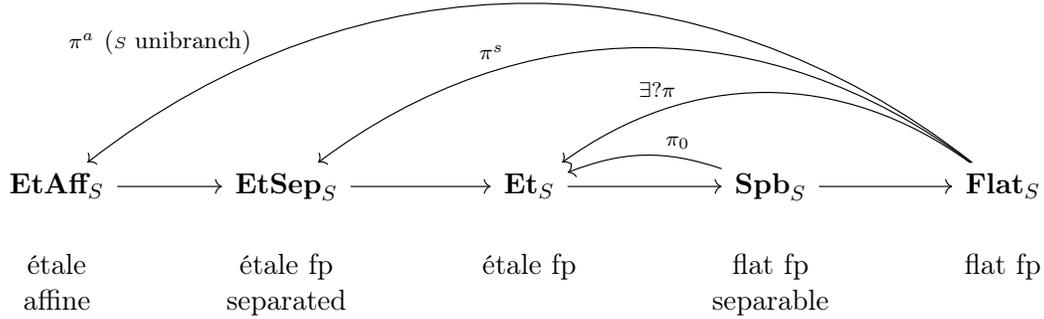
\[
\begin{tikzcd}[column sep=30, row sep=10]
\textbf{EtAff}_S \ar[r,] &
\textbf{EtSep}_S \ar[r] &
\textbf{Et}_S \ar[r] &
\textbf{Spb}_S \ar[r] \ar[l,bend right=20,"\pi_0"',pos=0.4] &
\textbf{Flat}_S \ar[lll,bend right=37,"\pi^s"',pos=0.7]
\ar[ll,bend right=36,"\exists?\pi"',pos=0.7]
\ar[llll,bend right=38,"{\pi^a \mbox{ \footnotesize ($S$ unibranch)}}"',pos=0.8] \\
\mbox{\begin{tabular}{c} \'etale \\ affine \end{tabular}} &
\mbox{\begin{tabular}{c} \'etale fp \\ separated \end{tabular}} &
\mbox{\begin{tabular}{c} \'etale fp \\ \quad \end{tabular}} &
\mbox{\begin{tabular}{c} flat fp \\ separable \end{tabular}} &
\mbox{\begin{tabular}{c} flat fp \\ \quad \end{tabular}}
\end{tikzcd}
\]
Here are some positive facts on the existence of these adjoints:
\begin{trivlist}
\itemn{i} $\pi_0$ is constructed in~\cite{Rom11}. It has
a moduli description in terms of connected components. When $X\to S$
is flat, finitely presented, the functor $\pi_0(X/S)$ is
representable by an algebraic space when either $X$ is separable,
or $S$ is zero-dimensional, see \cite{Rom11}, 2.1.3.
Its main properties (representability, adjointness,
commutation with base change) hold with no assumption on~$S$.
The morphism $X\to\pi_0(X/S)$ is surjective with
connected geometric fibres.
\itemn{ii} $\pi^s$ is constructed in~\cite{Fe19} when $S$
has finitely many irreducible components, and is not known to
exist otherwise. It has no known moduli
description. It has functoriality and base change properties
available only in restricted cases. The morphism $X\to\pi^s(X/S)$
is surjective but its geometric fibres are usually not connected.
\itemn{iii} $\pi^a$ is constructed is the present subsection
when $S$ is noetherian, geometrically unibranch, without embedded
points. It shares the same features as those just listed for $\pi^s$,
except that $X\to\pi^a(X/S)$ is schematically dominant but maybe not
surjective.
\end{trivlist}
Here are some negative facts:
\begin{trivlist}
\itemn{iv} $\pi$ is not known to exist unless $S$ is zero-dimensional
(in which case $\pi=\pi_0$).
\itemn{v} $\pi_0$ {\em does} extend naturally to a functor
$\textbf{Flat}_S\to \textbf{Et}_S$ but this is not a left adjoint
to the inclusion $i:\textbf{Et}_S\to \textbf{Flat}_S$. Indeed
\cite{Rom11}, 2.1.3 implies that for all flat,
finitely presented $X\to S$ the functor $\pi_0(X/S)$ defined as an
\'etale sheaf is constructible, hence an \'etale quasi-compact algebraic
space. Moreover, for each \'etale $E\to S$ there is a map
$\Hom(X,E)\to \Hom(\pi_0(X/S),E)$. However, in general there is no
map in the other direction; in particular there is no morphism
$X\to\pi_0(X/S)$ and this prevents $\pi_0$ from being an adjoint
of $i$. For instance, let $S$ be the spectrum of a discrete
valuation ring $R$ with fraction field $K$ and let
$X=\Spec(R[x]/(x^2-\pi x))$. Then $\pi_0(X/S)\simeq \Spec(K)\sqcup\Spec(K)$
and the map $\pi_0(X/S)\to S$ is not even surjective.
\end{trivlist}


\bigskip

We now start our investigations on $\pi^a$. To start with,
we recall the definition of the category of factorizations from
\cite{Fe19}. In order to make Theorem~\ref{theo_affine_etale_hull}
possible, we modify the definition slightly by
relaxing the assumption of surjectivity.

\begin{defi} \label{defi:cats_of_factorizations}
Let $X\to S$ be a morphism of algebraic spaces. The
{\em category of factorizations} is the category $\sfE(X/S)$ whose
objects are the factorizations $X\to E\to S$ such that $E\to S$ is
\'etale, and whose morphisms are the commutative diagrams:
\[
\begin{tikzcd}[row sep=0mm]
& E_1 \ar[rd] \ar[dd] & \\
X \ar[ru] \ar[rd] & & S . \\
& E_2 \ar[ru] &
\end{tikzcd}
\]
The category $\sfE^{\surj}(X/S)$, resp. $\sfE^{\dom}(X/S)$ is
the full subcategory of factorizations such that $X\to E$ is
surjective, resp. schematically dominant. The category
$\sfE^{\sep}(X/S)$, resp. $\sfE^{\aff}(X/S)$ is the full
subcategory of factorizations such that $E\to S$ is separated,
resp. affine. We write
$\sfE^{\aff,\dom}(X/S)=\sfE^{\aff}(X/S)\cap \sfE^{\dom}(X/S)$
and similarly for other intersections.
\end{defi}

We will often denote a factorization $X\to E\to S$ simply by
using the letter $E$. We draw the attention of the reader to
the fact that for the subcategories $\sfE^\sharp(X/S)$ defined
above, the property~``$\sharp$''
applies either to $E\to S$ or to $X\to S$, depending on the case.

\begin{lemma} \label{lemma:universal_homeo_descent}
Let $X\to S$ be a morphism of algebraic spaces $X\to S$.
Let $f:S'\to S$ be a morphism of spaces which
is integral, radicial and surjective. Let $X'=X\times_S S'$.
\begin{trivlist}
\itemn{1} The pullback functor $f^*:\sfE(X/S)\to \sfE(X'/S')$
is an equivalence which preserves the subcategories
$\sfE^{\sep}$, $\sfE^{\aff}$ and $\sfE^{\surj}$.
\itemn{2} If moreover $S,S'$ are locally noetherian,
$f$ induces a bijection $\Emb(S')\to\Emb(S)$ of embedded
points, and $X\to S$ is faithfully flat, then $f^*$ preserves
also the subcategory $\sfE^{\dom}$.
\end{trivlist}
\end{lemma}

\begin{proof}
First, we recall basic facts on the topological invariance of
the \'etale site. Let $f:S'\to S$ be a morphism of algebraic spaces
which is integral, radicial and surjective. Then the pullback functor
$f^*$ induces an equivalence between the category of \'etale
$S$-spaces and the category of \'etale $S'$-spaces: see \cite{SGA4.2},
Exp.~VIII for schemes and
\cite{SP19}, Tag~\href{http://stacks.math.columbia.edu/tag/05ZG}{05ZG}
for spaces.
This equivalence preserves affine objects, see \cite{SP19},
Tag~\href{http://stacks.math.columbia.edu/tag/07VW}{07VW}.

\smallskip

\noindent (1) We prove that $f^*$ is essentially surjective.
Let $X'\to E'\to S'$
be a factorization. By topological invariance of the \'etale site,
there exists an essentially unique $E\to S$ such that
$E'\simeq E\times_SS'$. In order to descend $u':X'\to E'$ to a
morphism $u:X\to E$, by descent of morphisms to an \'etale
scheme along universal submersions (\cite{SGA1}, Exp.~IX, prop.~3.2)
it is enough to prove that $\pr_1^*u'=\pr_2^*u'$ where
$\pr_1,\pr_2:S'\times_SS'\to S'$ are the projections.
By \cite{SGA1}, Exp.~IX, prop.~3.1 it is enough to find a surjective
morphism $g:S'''\to S'\times_SS'$ such that the two maps agree
after base change along $g$. We can take $S'''=S'$ and $g$
the diagonal map. This proves essential surjectivity; we leave
full faithfulness to the reader.
We now prove that $f^*$ preserves the indicated subcategories.
Since the diagonal of $E\to S$ is a closed immersion if and only
if the diagonal of $E'\to S'$ is a closed immersion, we see that
$f^*$ preserves $\sfE^{\sep}(X/S)$. The fact that $f^*$
preserves $\sfE^{\aff}$ was recalled above.
Finally $f^*$ preserves $\sfE^{\surj}$ because $f$ is
a universal homeomorphism.

\smallskip

\noindent (2) Here the morphisms $X\to E$ in the
factorizations are automatically flat. Thus such a morphism is
schematically dominant if and only if its image contains the set of
associated points $\Ass(E)$. Since $\Ass(E)=\cup_{s\in\Ass(S)} E_s$
by \cite{EGA} IV.3.3.1, we see that $X\to E$ is schematically
dominant if and only if the image of $X\to E$ contains all fibres
$E_s$ with $s\in \Ass(S)$. But $f$ induces a bijection of the
non-embedded associated points since it is a homeomorphism, and a
bijection on embedded points by assumption. Hence
it is equivalent to say that the image of $X'\to E'$
contains all fibres $E'_{s'}$ with $s'\in \Ass(S')$.
\end{proof}

\begin{noth}{\bf Suprema and infima}
We say that $E_1$ and $E_2$ have a {\em supremum} if the category
of factorizations~$E$ mapping to $E_1$ and $E_2$ has a terminal
element.
We say that $E_1$ and $E_2$ have a {\em infimum} if the category
of factorizations $E$ receiving maps from~$E_1$ and~$E_2$
has an initial element. In pictures:
{\small
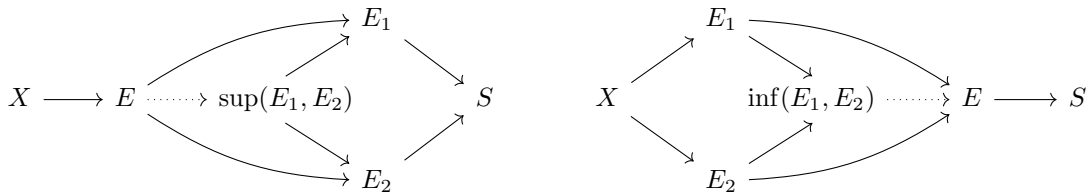
\[
\begin{tikzcd}[row sep=5mm,column sep=8mm]
& & &  E_1 \ar[rd] & \\
X \ar[r] & E \ar[rru,bend left=15] \ar[rrd,bend right=15] \ar[r,dotted] &
\sup(E_1,E_2)\hspace{-1cm} \ar[ru] \ar[rd] & & S \\
& & & E_2 \ar[ru] &
\end{tikzcd}
\quad\quad\quad
\begin{tikzcd}[row sep=5mm]
& E_1 \ar[rd] \ar[rrd,bend left=15] & & & \\
X \ar[ru] \ar[rd] & & \hspace{-1cm}\inf(E_1,E_2)
\ar[r,dotted] & E \ar[r] & S \\
& E_2 \ar[ru] \ar[rru,bend right=15] & & &
\end{tikzcd}
\]
}

\noindent Note that in the three categories
$\sfE^{\surj}(X/S)$, $\sfE^{\sep,\dom}(X/S)$ and
$\sfE^{\aff,\dom}(X/S)$, if there is a morphism
between $E_1$ and $E_2$ then it is unique.
In other words, these categories really are posets.
\end{noth}

\begin{coro} \label{coro:universal_homeo_descent}
Let $\sfE^\sharp(X/S)\subset \sfE(X/S)$ be any subcategory with
\[\sharp\in \{\varnothing,\sep,\aff,\surj,\dom\}.\]
Let $f:S'\to S$ be a morphism of spaces which is integral, radicial
and surjective. In case $\sharp=\dom$ assume moreover that $f$
and $X$ satisfy the assumptions
of~\ref{lemma:universal_homeo_descent}(2).
Then the following hold.
\begin{trivlist}
\itemn{1} $\sfE^\sharp(X/S)$ has
an initial element if and only if $\sfE^\sharp(X'/S')$ has one.
\itemn{2}
Let $E_1,E_2$ be factorizations in $\sfE^\sharp(X/S)$ and $E'_1,E'_2$
their images in $\sfE^\sharp(X'/S')$. Then $E_1,E_2$ have a supremum,
resp. an infimum, if and only if $E'_1,E'_2$ have a supremum,
resp. an infimum.
\end{trivlist}
\end{coro}

\begin{proof}
Suprema and infima are defined in terms of morphisms and are
therefore preserved by the equivalences
$f^*:\sfE^\sharp(X/S)\to \sfE^\sharp(X'/S')$.
\end{proof}

We arrive at the main existence result of this subsection.
We prepare the proof with two lemmas. The first is classical;
the proof given here was suggested to us by Daniel Ferrand.

\begin{lemma} \label{lemma:splitting_etale_separated}
Let $E\to S$ be an \'etale, quasi-compact, separated morphism of
schemes. Then after an \'etale surjective base change $S'\to S$,
the $S$-scheme $E$ is a disjoint union of a finite number of
open subschemes of $S$. If moreover $E\to S$ is surjective and
birational, it is an isomorphism.
\end{lemma}

\begin{proof}
Since $E\to S$ is of finite presentation, we can assume that $S$
is affine noetherian. Let $m(E/S)$ be the maximum of the number of
geometric connected components of the fibres of $E\to S$; this is
finite by \cite{EGA}, IV$_3$.9.7.8 and noetherian induction. The
base change $S'_1\defeq E\to S$ produces an open and closed section whose
complement has $m$-number strictly less. By induction on $m$, we
obtain a splitting of $E$ as a disjoint union of finitely many opens,
as asserted. The second claim follows because assuming birationality,
the number of opens has to be one.
\end{proof}

\begin{lemma} \label{lemma:smallest_open_affine}
Let $S$ be a separated noetherian scheme, and $U\subset S$
a nonempty dense open. Then the set of opens $V$ containing $U$
and such that $V\to S$ is affine is finite and has a smallest
element for inclusion.
\end{lemma}

\begin{proof}
If $V$ is such an open, the complement $S\setminus V$ is
included in $S\setminus U$ and has pure codimension~1 in~$S$
by \cite{EGA} IV.21.12.7. This proves that $S\setminus V$ is a
union of one-codimensional irreducible components of
$S\setminus U$. Since these are finite in number, we see the set
of interest is finite. Since $S$ is separated, the intersection of
all its elements is again $S$-affine and is the smallest element.
\end{proof}

\begin{theo} \label{theo_affine_etale_hull}
Let $f:X\to S$ be a faithfully flat, finitely presented
morphism of algebraic spaces. Assume that $S$ is noetherian,
geometrically unibranch, without embedded points.
Then the category $\sfE^{\aff,\dom}(X/S)$ is a lattice, that
is, any two objects have a supremum and an infimum. Moreover
$\sfE^{\aff,\dom}(X/S)$ has a largest element.
\end{theo}

A similar statement holds in the category $\sfE^{\surj,\sep}(X/S)$
where existence of suprema and maximum are due to
Ferrand~\cite{Fe19}.

\begin{proof}
Throughout the proof we write $\sfE=\sfE^{\aff,\dom}(X/S)$.
Note that for each factorization $X\to E\to S$, the morphism
$X\to E$ is flat and finitely presented.

We start with the proof that any two factorizations $E_1,E_2\in \sfE$
have a supremum. By topological invariance of the \'etale site,
we can assume that $S$ is reduced. Let $E$ be the schematic image
of the morphism $X\to E_1\times_S E_2$. As a closed subscheme of
$E_1\times_SE_2$, it is affine and unramified over $S$. By
the theorem on unramified morphisms over unibranch schemes
(\cite{EGA}, IV.18.10.1), it is enough to prove that for each
$e\in E$ with image $s\in S$, the map of local rings
$\cO_{S,s}\to\cO_{E,e}$ is injective. Let $\eta_1,\dots,\eta_n$
be the associated points of $S$ and let $\cO_{E,\eta_i}$ be the
semi-local rings of the fibres of $E\to S$ at $\eta_i$. Like in the
proof of~\ref{lemma:universal_homeo_descent}, we have
$\Ass(E)=E_{\eta_1}\cup \dots\cup E_{\eta_n}$. We have a commutative
diagram:
\[
\begin{tikzcd}
\cO_{S,s} \ar[r] \ar[d,hook] & \cO_{E,e} \ar[d,hook] \\
\Pi_{i=1}^n\cO_{S,\eta_i} \ar[r,hook] & \Pi_{i=1}^n\cO_{E,\eta_i}.
\end{tikzcd}
\]
The left and right maps are injective. The bottom map is injective
also because $E_{\eta_i}$ is in the image of $X\to E$ and $X\to S$
is faithfully flat. Therefore $\cO_{S,s}\to\cO_{E,e}$ is injective
and this concludes the argument.

Now we prove that there is a largest element. For each $E\in\sfE$,
the image of $X\to E$ is an open subscheme $U\subset E$, \'etale,
separated, quasi-compact over $S$, which we call the ``image''
of the factorization $E$. It is determined by the scheme
$R\defeq X\times_EX=X\times_UX$ which is the graph in $X\times_SX$
of an open and closed equivalence relation: indeed, we recover $U$
as the quotient algebraic space $X/R$. Because $S$ is noetherian,
there are finitely many open and closed equivalence relations
(\cite{Fe19}, 3.2.1, 3.2.2) hence finitely many ``images'' $U$. By the existence of suprema in $\sfE$, the poset of ``images'' forms a directed finite set, hence it has a largest element.

We fix $E\in \sfE$ whose ``image'' $U$
is largest. It is now enough to prove that the directed set of maps $u:F\to E$
in $\sfE$ has a largest element $u^{max}\colon E^{max}\to E$. Since $\sfE$ is a directed set, $E^{max}$ will automatically be a largest element for it, concluding the proof.

Given a map $u\colon E'\to E$, we observe
that there is an induced isomorphism $U'\simeq U$ between the
``images''. Moreover $U\subset E$ and $U'\subset E'$ are schematically
dense in $E$. It follows that the induced \'etale surjective separated morphism from $E'$ onto its image $u(E')\subset E$ is birational,
hence an isomorphism by Lemma~\ref{lemma:splitting_etale_separated}.
Since $E'$ is affine over $S$, then so is $u(E')$; hence
Lemma~\ref{lemma:smallest_open_affine} applied to the open
$U\subset E$ implies that the directed set of maps $F\to E$
stabilizes, so eventually an $E^{max}$ is achieved.

Finally, we construct an infimum for $E_1$ and $E_2$. Let $E_0$ be
the pushout of the diagram $E_1\leftarrow X\to E_2$, that is, the
quotient of $E_1\sqcup E_2$ by the \'etale equivalence relation
that identifies the image of $X\to E_1$ and the image of $X\to E_2$.
Let $E$ be the largest element of the category
$\sfE^{\aff,\dom}(E_0/S)$. This is the infimum of $E_1$ and $E_2$.
\end{proof}

\begin{defi} \label{defi:etale affine hull}
With the notations and assumptions of
Theorem~\ref{theo_affine_etale_hull}, the largest element of
the poset
$\sfE^{\aff,\dom}(X/S)$ is called the {\em \'etale affine hull
of $X/S$} and denoted $\pi^a(X/S)$. Its $\cO_S$-sheaf of
functions is called the {\em largest (quasi-coherent)
\'etale $\cO_S$-subalgebra of $f_*\cO_X$}.
\end{defi}

Giving an existence proof which is more constructive than the
one given above is not easy because of the limited formal properties
of the \'etale affine hull (compatibility with base change, with the
formation of products, etc). Such properties are of course
essential in most situations where the etale affine hull is useful.
A sample of base change results for the \'etale separated hull is
given in \cite{Fe19}, \S~7. Similar results can be proven for the
\'etale affine hull.

\begin{coro}\label{lemma:fonctoriality_for_pi^a}
Let $S$ be a noetherian geometrically unibranch scheme without
embedded points. Let $u:X\to Y$ be a morphism
between faithfully flat, finitely presented $S$-algebraic spaces.
\begin{trivlist}
\itemn{1} There is an induced morphism of \'etale affine hulls
$\pi^a(X/S)\to \pi^a(Y/S)$.
\itemn{2} The functor $\pi^a$ is left adjoint to the inclusion of the
category of \'etale, affine $S$-schemes into the category of
faithfully flat, finitely presented $S$-algebraic spaces.
\end{trivlist}
\end{coro}

\begin{proof}
(1) By topological invariance of the \'etale site
(Lemma~\ref{lemma:universal_homeo_descent}), we can
assume that $S$ is reduced. Let $E$ be the schematic image of
$X\to Y\to \pi^a(Y/S)$. It follows from the theorem on unramified
morphisms over unibranch schemes (\cite{EGA}, IV.18.10.1)
that $E\to S$ is \'etale. By the definition of $\pi^a(X/S)$
we obtain a morphism $\pi^a(X/S)\to \pi^a(Y/S)$.

\smallskip

\noindent (2) Let $u:X\to E$ be an $S$-morphism from a faithfully
flat, finitely presented space to an \'etale, affine scheme.
By (1) there is an induced morphism $\pi^a(X/S)\to \pi^a(E/S)$.
Since $E\to \pi^a(E/S)$ is an isomorphism, we obtain a morphism
$\pi^a(X/S)\to E$.
\end{proof}

\subsection{Affine hull of the space of components}
\label{subsection_global_section}

Let $S$ be a noetherian scheme and $X\to S$ a flat
separable morphism of finite type. A priori, there is no reason
to expect that $\pi_0(X/S)^{\aff}\to S$, the affine hull of
$\pi_0(X/S)\to S$, be \'etale. There are two reasons for this:
the first, is that a priori $\pi_0(X/S)^{\aff}$ may not be
of finite type. The second reason is that, even when it is of
finite type, it may well be ramified. This may happen already
over a dimension~$1$ base with a nodal singularity, as
Example~\ref{example_non_perfect} illustrates.

Here we describe a case where $\pi_0(X/S)^{\aff}$ is \'etale,
for some geometrically unibranch reduced base schemes $S$.
More precisely, in this situation the \'etale affine hull
$\pi^a(X/S)\to S$ exists, and there is a natural map
$\pi_0(X/S)^{\aff}\to \pi^a(X/S)$. We will prove that under some
local factoriality-type conditions, this is an isomorphism.

\begin{defi}
A noetherian local ring $R$ is called {\em geometrically set-theoretically
factorial} if its strict henselization is integral, and each
pure one-codimensional closed subscheme of $\Spec(R)$ has the
same support as a principal closed subscheme.
\end{defi}

Although a little ill-looking, this definition includes many
examples of interest such as regular rings, $\QQ$-factorial rings
like the quadratic cone singularity $xy=z^2$, and all reduced
unibranch curves. We note moreover that these examples are also $S_2$
and hence satisfy all the assumptions of the following statement.

\begin{prop} \label{prop:set_theoretically_factorial_etale}
Let $X\to S$ be a morphism of algebraic spaces which is flat,
separable, and finitely presented. Assume that $S$ is locally
noetherian, $S_2$, with geometrically set-theoretically factorial
local rings. Then the natural map $\pi_0(X/S)^{\aff}\to \pi^a(X/S)$
is an isomorphism.
\end{prop}

\begin{proof}
It is enough to prove that $\pi_0(X/S)^{\aff}\to S$ is \'etale.
We prove more generally that for all \'etale, quasi-compact algebraic
spaces $E\to S$ the map $E^{\aff}\to S$ is \'etale. For this, we
can work \'etale-locally on $S$. First let us see that we can reduce
to the case where $E\to S$ is separated. By Ferrand \cite{Fe19},
Th.~3.2.1 there is an \'etale separated hull $\pi^s(E/S)\to S$.
By \cite{Fe19}, Prop.~8.1.2 the map $E\to \pi^s(E/S)$ is initial
among maps to separated schemes; note that Ferrand assumes normality
of $S$ but really uses only the unibranch hypothesis
(in {\em loc. cit.}, this is said explicitly before Lemma~6.1.1 which
is the key to Lemma~8.1.1). Since $E^{\aff}\to S$ is separated, we obtain
a factorization $E\to \pi^s(E/S)\to E^{\aff}$. Taking global sections,
the map $\cO(E^{\aff})\to \cO(\pi^s(E/S))\to \cO(E)=\cO(E^{\aff})$
is the identity; since $E\to \pi^s(E/S)$ is dominant we see that $E$
has the same affine hull as $\pi^s(E/S)$. Hence replacing $E$ by
$\pi^s(E/S)$ if necessary, we can assume that it is separated. By
Lemma~\ref{lemma:splitting_etale_separated}, working \'etale-locally
around a fixed point $s\in S$ we can reduce to the case where $S$ is
affine and $E$ is an open of $S$. Let us write the closed complement as
$Z\defeq S\setminus E=Z_1\cup Z'$ where $Z_1$ has pure codimension 1 in $S$
and $Z'$ has codimension at least~2. By the assumption that the strictly
local ring of $s$ is geometrically set-theoretically factorial, the
1-cycle $Z_1$ is set-theoretically principal on a small enough \'etale
neighbourhood of $s$ in $S$. We replace $S$ by such a neighbourhood
and let $f$ be a local equation for $Z_1$. Then the morphism
$\cO(S)\to\cO(S\setminus Z_1)$ is the localization-by-$f$ map which
is \'etale. Since moreover $S$ has the $S_2$ property, the restriction
$\cO(S\setminus Z_1)\to\cO(S\setminus Z)=\cO(E)$
is an isomorphism. The result follows.
\end{proof}

In the one-dimensional case, removing the unibranch condition
in~\ref{prop:set_theoretically_factorial_etale}
yields a weaker result:

\begin{prop}\label{prop_curve_qf}
Let $S$ be a reduced noetherian excellent scheme of dimension
$\leq 1$. Let $X\to S$ be a flat separable morphism of finite
presentation. Then $\pi_0(X/S)^{\aff}$ is quasi-finite.
\end{prop}

\begin{proof}
Quasi-finiteness of $\pi_0(X/S)^{\aff}$ may be checked \'etale locally on $S$. So we let $s$ be a geometric point of $S$ and $(S',s')\to (S,s)$ an \'etale neighbourhood such that:
\begin{itemize}
\item[i)] the irreducible components $S_1,\ldots,S_n$ of $S'$ are geometrically unibranch;
\item[ii)] for every $i\neq j$, $S_i\cap S_j=\{s'\}$;
\item[iii)] the fibre $\pi_0(X/S)_{s'}$ is a disjoint union of copies of
$\Spec k(s')$;
\item[iv)] $S'=\Spec R'$ is affine.
\end{itemize}
The reason why an \'etale neighbourhood satisfying condition~i)
exists, is that the regular locus of~$S$ is open dense by
excellence, hence so is the geometrically unibranch locus.
So we may replace~$S$ by $S'$ and assume that $S=\Spec R$
satisfies the properties above.

Write $\pi=\pi_0(X/S)$. Then $\pi=\pi'\sqcup \pi^*$, where $\pi^*$
is the union of those connected components that do not meet the fibre
$\pi_{s'}$. Then $\pi^*$ lives over $S'\setminus\{s'\}$ which by
condition ii) is geometrically unibranch. By
Proposition~\ref{prop:set_theoretically_factorial_etale},
the map $\pi^{\aff}\to S$ is \'etale, and in
particular quasi-finite. It remains to check that $\pi'^{\aff}$ is quasi-finite.

Up to restricting $S$ by a further \'etale neighbourhood of $s$, we may assume that the isomorphism $\bigsqcup_{i=1}^n\Spec k(s)\to \pi'_s$ extends to an open immersion $\alpha\colon\bigsqcup_{i=1}^nS\to \pi'$. We claim that $\alpha$ has dense image. Indeed, let $Z$ be an irreducible component of $\pi'$. Then $Z$ maps to some irreducible component $S_i$ of $S$. By assumption, $S_i$ is geometrically unibranch, so by \cite{EGA}, th. 18.10.1, $Z\to S_i$ is \'etale. In particular $Z\to \pi'_{S_i}$ is an \'etale, closed immersion, that is, $Z$ is a connected component of $\pi'_{S_i}$. Thanks to condition ii), $Z$ is also a connected component of $\pi'$, and therefore meets the closed fibre. In particular it meets the image of $\alpha$. This proves the claim.

The morphism $\alpha$ is dominant and induces an injective $R$-algebra morphism $\cO(\pi')\into R^n.$ It follows that $\cO(\pi')$ is finite as an $R$-module. In particular $\pi'^{\aff}\to S$ is finite.
\end{proof}

\subsection{Computing the space of components}
\label{subsection:computing pi0}

In this subsection, we collect some ways to compute the space
of connected components $\pi_0(X/S)$ in various situations:
when we apply changes of the base, when we use an atlas of $X$,
and when we complete along a closed fibre of $X\to S$.

The first two results deal with arbitrary base change and
\'etale base restriction.
Both results hold whether~$\pi_0(X/S)$ is representable or not.

\begin{lemma} \label{lemma:base change for pi0}
Let $\mathscr{X}/S$ be an $S$-algebraic stack and let
$S'\to S$ be a base change. Then we have a canonical
isomorphism of $S'$-functors
$\pi_0(\mathscr{X}\times_SS'/S') \isomto \pi_0(\mathscr{X}/S)\times_S S'$.
\end{lemma}

\begin{proof}
This follows from the definition of $\pi_0$ because both sides
of the map in the statement parametrize relative connected
components of $\mathscr{X}\times_S T'$ for variable $S'$-schemes $T'$.
\end{proof}

The following statement is related to factorizations
in the sense of Definition~\ref{defi:cats_of_factorizations}.

\begin{lemma} \label{lemma:base restriction}
Let $\sX\stackrel{h}{\too}\sE\stackrel{f}{\too} S$
be morphisms of algebraic stacks.
\begin{trivlist}
    \itemn{1} If $\sE\to S$ is an \'etale algebraic space, there is a morphism of $S$-functors
    \[
    f_!\pi_0(\sX/\sE) \too \pi_0(f_!\sX/S).
    \]
    which is an isomorphism when $\sX\to \sE$ is universally open.
    \itemn{2} If $\sX,\sE$ are finitely presented over $S$ and $\sX\to \sE$ is a universal submersion with connected geometric fibres, there is an isomorphism
    \[
    \pi_0(\sX/S) \isomto \pi_0(\sE/S).
    \]
\end{trivlist}
\end{lemma}

\begin{proof}
Since (2) is easy to prove and not used in the paper, we only prove (1).
Note that if $\sX\to \sE$ is flat, finitely presented and separable,
this follows from Lemma~\ref{lemma:coperf and base restr}. However, here
we assume much less.
The morphism in the statement is constructed as follows.
For each $S$-scheme $T$, a point of $f_!\pi_0(\sX/\sE)$ with
values in $T$ is a pair composed of an $S$-morphism
$u:T\to \sE$ and a $T$-relative connected component
$\sC'\subset \sX\times_{\sE} T$. Since $\sE\to S$ is \'etale,
the map $\sX\times_{\sE} T\to \sX\times_S T$ is an open immersion
globally and a closed immersion in the fibres, showing
that~$\sC\defeq \sC'$ is a $T$-relative connected component of $\sX\times_ST$
i.e. a $T$-valued point of $\pi_0(f_!\sX/S)$. Let us describe
the inverse morphism, assuming $\sX\to \sE$ universally open. Let
$\sC\subset \sX\times_S T$ be a $T$-relative connected component.
By the assumption on $\sX\to \sE$, the image $\sD$ of $\sC$ in
$\sE\times_S T$ is open, hence \'etale over $T$ with nonempty
geometrically connected $T$-fibres. It follows that $\sD\to T$
is an isomorphism. Using its inverse, we obtain a morphism
$T\to \sD\to \sE$ and the
pair $(T\to \sE,\sC)$ is a $T$-point of $f_!\pi_0(\sX/\sE)$.
These constructions are inverse to each other.
\end{proof}

We continue with a description of $\pi_0(X/S)$ in terms of
an atlas. This takes the form of a pushout property which is a
consequence of the right exactness of the functor $\pi_0$,
and will have an important refinement in the context of
stacks in the later Lemma~\ref{lemma:stacky-pushout}.

\begin{lemma} \label{prop:pushout_with_pi_0}
Let $X\to S$ be a flat, finitely presented morphism of
separable algebraic spaces and
let $U\to X$ be an fppf, separable, surjective morphism.
\begin{trivlist}
\itemn{1} Let $R\subset U\times U$ be the fppf equivalence
relation defined by $U\to X$, so that $X$ is identified
with the coequalizer $\coeq(R\toto U)$.
Then we have $\pi_0(X/S)=\coeq(\pi_0(R/S)\toto \pi_0(U/S))$.
\itemn{2} The diagram
\[
\begin{tikzcd}
U \ar[r] \ar[d] & X \ar[d] \\
\pi_0(U/S) \ar[r] & \pi_0(X/S)
\end{tikzcd}
\]
is a pushout in the category of sheaves.
\end{trivlist}
\end{lemma}

We warn the reader that
$\pi_0(R/S)\to \pi_0(U/S)\times_S\pi_0(U/S)$
may fail to be injective; e.g. $R$ may be disconnected in
a connected $U$.

\begin{proof}
Throughout, we write $\pi_0(X)$ instead of $\pi_0(X/S)$
and we omit $S$ from fibred products.

\smallskip

\noindent (1) Let $\tilde\pi_0(R)$ denote the
equivalence relation generated
by the image of $\pi_0(R)\to \pi_0(U)\times\pi_0(U)$.
Let us prove that the formation of $\tilde\pi_0(R)$
commutes with fppf surjective refinements $f:U'\to U$.
That is, if $f^*R\subset U'\times U'$ is the preimage of
the equivalence relation $R$ under $U'\times U'\to U\times U$
and $f^*(\tilde\pi_0(R))$ is the preimage of the relation
$\tilde\pi_0(R)$ under
$\pi_0(U')\times \pi_0(U')\to \pi_0(U)\times \pi_0(U)$
then we want to prove that the natural map
\[
\tilde\pi_0(f^*R) \too f^*\tilde\pi_0(R)
\]
is an isomorphism. For this it is enough to prove that
$\pi_0(f^*R)\to f^*\pi_0(R)$ is surjective. Since the spaces
are \'etale, we may assume that $S$ is
the spectrum of an algebraically closed field, and we can
represent each connected component by a point lying on it.
Since $f:U'\to U$ is surjective, a point of
\[
f^*\pi_0(R)=\pi_0(R)\times_{\pi_0(U)\times \pi_0(U)}
\pi_0(U')\times \pi_0(U')
\]
can be represented by a triple
$(r,u_1,u_2)\in R(k)\times U'(k)\times U'(k)$, which is
what we wanted to prove. Since any two atlases for~$X$
have a common refinement, it follows that
the quotient space $\pi_0(U)/\tilde\pi_0(R)$ does not depend
on the choice of $U$ up to a canonical isomorphism, and
taking $U=X$ and $U'=U$ we see that
\[
\pi_0(U)/\tilde\pi_0(R) \simeq \pi_0(X).
\]

\noindent (2) Let $Y$ be a sheaf and let $a:X\to Y$,
$b:\pi_0(U)\to Y$ be maps that coincide on $U$. Denote by
$u:U\to X$ the chosen atlas and $s,t:R\to U$ the
projections. Let $\sigma,\tau$ be the maps
$\pi_0(s),\pi_0(t):\pi_0(U)\to \pi_0(X)$. Using that
$R\to \pi_0(R)$ is an epimorphism of sheaves, from
$aus=aut$ we deduce $b\sigma=b\tau$. Then~(1)
implies that $b$ factors through a map $\pi_0(X)\to Y$.
\end{proof}


\begin{noth}{Completion} \label{noth:completion}
We finish this subsection with a description of $\pi_0(X/S)$
over a complete local base which will be crucial for the proof
of Theorem~\ref{theo:preperfection_algebra}.
Let $S$ be the spectrum of a complete noetherian local ring $R$
with maximal ideal $\mathfrak m$.
For each $n\ge 0$ let $S_n=\Spec R/\mathfrak m^{n+1}$.
By \cite{EGA}~IV.18.5.15, restriction to $S_0$ yields
an equivalence $\FEt/S\simeq\FEt/S_0$ between the categories of
finite \'etale algebras.
In particular, given $X\to S$ flat of finite type and separable,
there exists a unique finite \'etale scheme $\hat{\pi}/S$ restricting
to $\pi_0(X\times_S{S_n}/S_n)$ over each $S_n$. Alternatively, one can
see $\hat{\pi}$ as the algebraization of the formal completion of
$\pi_0(X/S)$, which explains the choice of notation $\hat{\pi}$. As
$\hat\pi$ is finite over $S$, it is a product of complete local rings.
By
\cite{SP19}, Tag~\href{http://stacks.math.columbia.edu/tag/0AQH}{0AQH}
there is a natural morphism of $S$-algebraic spaces
\begin{equation}\label{map_pi_hat}
\psi\colon\hat \pi\to \pi_0(X/S),
\end{equation}
which restricts to an isomorphism over each $S_n$.
\end{noth}

\begin{prop}\label{prop:pushout}
Let $R$ be a complete noetherian ring, $A$ a flat separable $R$-algebra of finite type. Write $X=\Spec A$, $S=\Spec R$, $s$ for the closed point of $S$, and let $V=S\setminus\{s\}$. The commutative diagram of $S$-algebraic spaces
\begin{center}
\begin{tikzcd}
\hat\pi_V \ar[r, hookrightarrow]\ar[d, "\psi_V"] & \hat\pi \ar[d, "\psi"]\\
\pi_V \ar[r, hookrightarrow] & \pi_0(X/S) \\
\end{tikzcd}
\end{center}
is a pushout in the category of fppf sheaves over $S$.
\end{prop}

\begin{proof}
In the proof we write $\pi\defeq \pi_0(X/S)$.
In order to prove the claim, it suffices to show that any diagram of solid arrows
\begin{center}
\begin{tikzcd}
\hat\pi_V \ar[r]\ar[d, "\psi_V"] & \hat\pi \ar[d, "\psi"]
\ar[ddr, bend left=25,"a"]\\
\pi_V \ar[r] \ar[drr, bend right=25, "b"']& \pi \ar[dr, dashed]\\
&&Z
\end{tikzcd}
\end{center}
where $Z$ is an $S$-sheaf, admits a unique dashed arrow making the diagram commute.

First of all, notice that $\psi\colon\hat\pi\to\pi$ is \'etale; writing $U=\pi_V\sqcup \hat \pi$, it follows that $U\to \pi$ is faithfully flat of finite presentation, hence it is a coequalizer for $U\times_\pi U\to U$. Therefore, in order to obtain a unique dashed arrow, it suffices to check that $a\circ p_1=a\circ p_2$, where $p_1,p_2$ are the projections $\hat\pi\times_\pi\hat\pi\to\hat\pi$.

The $S$-scheme $\hat\pi$ is finite \'etale, hence the map $\psi\colon\hat\pi\to \pi$ is separated and quasi-finite, and so is also the base change $p_1\colon \hat\pi\times_\pi\hat\pi\to \hat\pi$. Moreover, we know that $\hat\pi$ is a finite disjoint union of spectra of completed local rings; by the classification of separated quasi-finite schemes over henselian local rings, $\hat\pi\times_\pi\hat\pi$ decomposes into a disjoint union $P^f\sqcup P'$ such that $p_1\colon P^f\to \hat \pi$ is finite (and \'etale), and $P'=P'_V$ has empty closed fibre. One obtains a similar decomposition for the map $p_2$, let us say $\hat\pi\times_\pi\hat\pi=Q^f\sqcup Q'$. However, the compositions $\hat\pi\times_\pi\hat\pi \xrightarrow{p_i} \hat\pi\to\pi\to S$ are the same map for $i=1,2$, and are both quasi-finite, separated; so both $P^f$ and $Q^f$ are equal to the finite part of the composition, and we find $P^f=Q^f$.

The restriction of $\psi$ to the closed fibre, $\psi_s\colon \hat\pi_s\to \pi_s$, is an isomorphism by construction of $\hat\pi$, and therefore so is $P^f_s=(\hat\pi\times_\pi\hat\pi)_s\xrightarrow{p_1} \hat\pi_s.$ The isomorphism extends uniquely to an isomorphism $P^f\to \hat \pi$.

Consider the diagram of solid arrows
\begin{center}
\begin{tikzcd}
\hat\pi\sqcup P' \ar[r, "p_2"]\ar[d, "p_1"] & \hat \pi \ar[d, "a"] \\
\hat\pi \ar[r, "a"] &  Z
\end{tikzcd}
\end{center}
where we have identified $P^f$ with $\hat \pi$. We want to show that it is commutative.

For $i=1,2$, the morphism $p_i$ is the identity on $\hat \pi$, so we really only need to show that $a\circ p_1$ agrees with $a\circ p_2$ on $P'$. As $P'$ is contained in $(\hat\pi\times_\pi\hat\pi)_V$, we have $a\circ (p_1)_V=b\circ \psi_V\circ (p_1)_V=b\circ \psi_V\circ (p_2)_V=a\circ (p_2)_V$ and the proof is complete.
\end{proof}

\section{Perfection of algebras}
\label{section:perf_of_algebras}

The commutative algebra developed in this section has independent
interest but is also fruitfully introduced with an eye towards
the geometric applications of the next section. Let
$X\to S$ be a flat, finitely presented morphism of algebraic
spaces of characteristic~$p$. In order to study the coperfection
of $X$ in the category of $S$-algebraic spaces, we will use the
\'etale algebraic spaces $\pi_0(X/S)$ and $\pi^a(X/S)$ (assuming
they exist). Since
\'etale implies relatively perfect, the morphism $X\to\pi_0(X/S)$
extends to the direct Frobenius system and we have a diagram:
\[
\left(X\stackrel{\Frob_0}{\too} X^{p/S}
\stackrel{\Frob_1}{\too} \dots\right) \tooo \pi_0(X/S)
\tooo \pi^a(X/S).
\]
The present section is devoted to the case where $S=\Spec(R)$
and $X=\Spec(A)$. The main question is whether there exists a
{\em perfection functor}, right adjoint to the inclusion of
perfect $R$-algebras into all $R$-algebras. In such generality
we do not know if such perfection exists. At least an obvious
approximation should be the {\em preperfection}:
\[
A^{p^{\infty}/R}\defeq \lim A^{p^i/R}=
\lim\,\big(\!\!
\begin{tikzcd}
\cdots \ A^{p^2/R} \ar[r, "\Frob_A"] &
A^{p/R} \ar[r, "\Frob_A"] & A
\end{tikzcd}
\!\!\big).
\]
The above diagram of spaces provides a diagram of algebras
\[
A^{\et/R} \too \cO(\pi_0(A/R)) \too A^{p^\infty/R}
\]
where $A^{\et/R}=\cO(\pi^a(A/R))$ is the largest \'etale
subalgebra of $A$, see Definition~\ref{defi:etale affine hull}.
Our goal is roughly to find as many situations as possible
where both maps above are isomorphisms.

We start in \S~\ref{subsection:base change} with preliminary
material on base change in the formation of the preperfection.
Then we prove that both maps above are indeed isomorphisms
when $R$ is artinian and $R\to A$ is of finite type, see
\S~\ref{subsection:artinian base}, or $R$ is regular and
$R\to A$ is of finite type and separable, see
\S~\ref{subsection:regular base}. Over a general noetherian ring,
only the map $\cO(\pi_0(A/R)) \to A^{p^\infty/R}$ is an
isomorphism, see \S~\ref{subsection:noetherian base}. This is
already remarkable, given the poor properties of both algebras:
in general $\cO(\pi_0(A/R))$ is not \'etale and $A^{p^\infty/R}$
is not perfect, even when $R\to A$ is flat, of finite type and
separable. One may expect that after iterating the preperfection
functor $(-)^{p^\infty/R}$ a finite (sufficiently high) number of
times, one reaches a perfect $R$-algebra. With the hope that this
might be true, we establish in \S~\ref{subsection:regular base} some
finiteness properties of $A^{p^\infty/R}$. We conclude the section
with counterexamples.

\subsection{Base change in preperfection}\label{subsection:base change}

For each morphism of $\FF_p$-algebras $R\to A$ and each
base change morphism $R\to R'$ we have a natural base change map
for preperfection:
\[
\phi=\phi_{R,R',A}: A^{p^\infty/R}\otimes_R R' \too (A\otimes_RR')^{p^\infty/R'}.
\]
It is important to understand this map for at least two reasons.
The first is that the study of $A^{p^\infty/R}$ with the usual tools
(localization, completion on $R$...) involves many base changes.
The second is that the base change map along Frobenius $\Frob:R\to R$
controls the success or failure of $A^{p^\infty/R}$ to be perfect;
we elaborate on this in Remark~\ref{rema:Frobenius of preperfection}.
Before stating the first lemma devoted to properties of $\phi$,
we recall a result of T. Dumitrescu.

\begin{theo}
\label{theo:Frobenius of reduced is injective}
Let $R\to A$ be a morphism of noetherian commutative rings.
Let $\Frob_{A/R}:A^{p/R}\to A$ be the relative Frobenius
morphism. Then the following are equivalent:
\begin{itemize}
\item[\rm (i)] $R\to A$ is flat and separable,
\item[\rm (ii)] $\Frob_{A/R}$ is injective and its cokernel
is a flat $R$-module.
\end{itemize}
\end{theo}

\begin{proof}
See \cite{Du95}, Theorem~3.
\end{proof}

\begin{rema}
If we do not assume that $R$ and $A$ are noetherian but
$R\to A$ is of finite presentation, then (i) $\To$ (ii) is
true. Indeed $R\to A$ is the base change of a map $R_0\to A_0$
along a morphism $R_0\to R$ with $R_0$ noetherian and we may
choose $R_0\to A_0$ flat and
separable, see \cite{EGA}~IV$_3$, 11.2.7 and 12.1.1(vii).
Then by the noetherian case, it follows that $\Frob_{A_0/R_0}$
is injective with $R_0$-flat cokernel. By base
change $\Frob_{A/R}$ is injective with $R$-flat cokernel.
\end{rema}

\begin{lemma}\label{lemma:properties of base change map}
The base change map
$\phi_{R,R',A}:A^{p^\infty/R}\otimes_R R' \too (A\otimes_RR')^{p^\infty/R'}$
is:
\begin{trivlist}
\itemn{1} an isomorphism if $R\to R'$ is finite locally
free.
\itemn{2} injective in each of the following cases:
\begin{itemize}
\item[\rm (i)] $R\to R'$ is projective.
\item[\rm (ii)] $R\to R'$ is flat and $R\to A$ is flat, finitely
presented, with reduced geometric fibres.
\item[\rm (iii)] $R'=\colim R$ is the absolute coperfection
of a ring $R$ such that $\Frob:R\to R$ is projective.
\end{itemize}
\end{trivlist}
\end{lemma}

\begin{proof}
Note that since
$(A\otimes_RR')\otimes_{R',\Frob^i}R'=A^{p^i/R}\otimes_RR'$,
the map $\phi_{R,R',A}$ is just a special case for the $R$-module
$M\defeq R'$ of the map $\phi_{R,M,A}$ that appears as the upper
horizontal row in the following commutative diagram:
\begin{center}
\begin{tikzcd}[column sep=large]
(\lim A^{p^i/R})\otimes_R M \ar[r,"\phi_{R,M,A}"] \ar[d] &
\lim (A^{p^i/R}\otimes_R M) \ar[d] \\
\left(\prod_{i\ge 0} A^{p^i/R}\right)\otimes_R M \ar[r,"\psi_{R,M,A}"] &
\prod_{i\ge 0} (A^{p^i/R}\otimes_R M).
\end{tikzcd}
\end{center}
In the sequel we assume that $M$ is flat, so the left-hand vertical
map is injective. If $M$ is free, resp. free of finite rank, then
$\psi_{R,M,A}$ is injective, resp. an isomorphism. It follows that
also $\phi_{R,M,A}$ is injective, resp.
an isomorphism. If~$M$ is projective, one reaches the same conclusions
by embedding it in a free module, resp. a free module of finite rank,
and using the facts that $\phi_{R,M,A}$ and $\psi_{R,M,A}$ are
additive in~$M$. This settles cases (1) and (2.i).

In case (2.ii), by Dumitrescu's
theorem~\ref{theo:Frobenius of reduced is injective} all the maps
$A^{p^i/R}\to A^{p^{i+1}/R}$ are injective; it follows that
$\lim A^{p^i/R}\to A^{p^j/R}$ is injective for each fixed $j$.
By flatness of $R\to R'$ the tensored map
$(\lim A^{p^i/R})\otimes_R R'\to A^{p^j/R}\otimes_R R'$ is injective.
Therefore $\phi_{R,R',A}$ is also injective.

In case (2.iii) we can write the coperfection as
$R'=\colim R^{p^{-j}}$. Since the absolute Frobenius of $R$ is
projective, it is in fact faithfully flat. It follows that
the maps $R^{p^{-j}}\to R^{p^{-(j+1)}}$ are faithfully flat,
hence universally injective. Thus for each $i,j$ the map
\[
A^{p^i/R}\otimes R^{p^{-j}}\too A^{p^i/R}\otimes R^{p^{-(j+1)}}
\]
is injective. Then for each $i$
\[
A^{p^i/R}\otimes R^{p^{-j}}\too
\underset{j}{\colim}\, A^{p^i/R}\otimes R^{p^{-(j+1)}}
\]
is injective. Taking limits
\[
\lim_i\,(A^{p^i/R}\otimes R^{p^{-j}})\too
\lim_i \underset{j}{\colim}\,A^{p^i/R}\otimes R^{p^{-(j+1)}}
\]
is injective, which implies that
\[
\underset{j}\colim\lim_i\,(A^{p^i/R}\otimes R^{p^{-j}})\too
\lim_i \underset{j}{\colim}\,A^{p^i/R}\otimes R^{p^{-(j+1)}}
=\lim_iA^{p^i/R}\otimes R'
\]
is injective. Since also by (2.i) the map
\[
(\lim_iA^{p^i/R})\otimes R'=
\underset{j}\colim\,(\lim_iA^{p^i/R})\otimes R^{p^{-j}}\too
\underset{j}\colim\lim_i\,(A^{p^i/R}\otimes R^{p^{-j}})
\]
is injective, by composition we obtain the result.
\end{proof}

\begin{remas} \label{rema:Frobenius of preperfection}
(1) Let $R\to A$ be a map of rings of characteristic $p>0$.
When inquiring whether the preperfection $A^{p^{\infty}/R}$ is
perfect, we are led to ask if the Frobenius of the preperfection
(``Frobenius of the limit'')
is an isomorphism. In general it is not; an example is given in
in~\ref{example_non_perfect}. In contrast, the morphism
obtained as the limit of the Frobenius maps of the individual rings
of the system (``the limit of Frobenius'')
{\em is} an isomorphism: it is essentially a shift
by one in the indices, which is invisible in the infinite system.
In fact, ``Frobenius of the limit'' and
``the limit of Frobenius'' are the two edges of a
commutative triangle whose third edge, the base change map in
preperfection, serves to compare them:
\[
\begin{tikzcd}[column sep=70,row sep=10]
A^{p^{\infty}/R}\otimes_{R,\Frob}R \ar[dd,"\phi_{R,R,A}"']
\ar[rd,bend left=10,"\Frob_{A^{p^{\infty}}/R}",pos=0.4] & \\
& A^{p^{\infty}/R} \\
(A\otimes_{R,\Frob}R)^{p^{\infty}/R}
\ar[ru,bend right=10,"\lim \Frob"',pos=0.4] &
\end{tikzcd}
\]
Since $\lim \Frob$ is an isomorphism, we see that $A^{p^{\infty}/R}$
is a perfect $R$-algebra if and only if the base change
map $\phi_{R,R,A}$ is an isomorphism. According to
Lemma~\ref{lemma:properties of base change map}(1), this happens
when $R$ is regular and $\Frob$-finite, for then absolute Frobenius
is finite locally free (see \cite{Ku69}).

\smallskip

\noindent (2)
In case (2.ii), it will be a consequence of
Theorem~\ref{theo:coperfection_spaces} that the base change map
is in fact an isomorphism.

\smallskip

\noindent (3)
Here is an example where the base change map is not surjective.
Let $k$ be a field of characteristic $p$ and $k'$ an
infinite-dimensional field extension. Let
\[
A=k[\eps_0,\eps_1,\dots]/(\eps_0^p,\eps_{i+1}^p-\eps_i)
\]
and $A'=A\otimes_kk'$. Let $t_0,t_1,\dots$ be an infinite family
of elements of $k'$ that is $k$-linearly independent.
Let $x'_i=\eps_0t_i+\eps_1t_{i-1}+\dots+\eps_it_0 \in (A')^{p^i/k'}$.
Then $\Frob_{A'/k'}(x'_{i+1})=x'_i$ so $x'=(x'_i)$ is an element of
$(A')^{p^\infty/k'}$ which obviously does not come from
$A^{p^\infty/k}\otimes k'$.
\end{remas}

We end this section with a case where preperfection commutes with
base change; since it is not used in the paper, we omit the proof.

\begin{lemma}\label{lemma_base_Zariski}
Let $A$ be an $R$-algebra, flat of finite presentation, such that the induced morphism $\Spec A\to\Spec R$ has geometrically reduced fibres. Let $f\in R$ be a non-zero divisor, with $R/fR$ reduced. The natural map $\phi\colon A^{p^\infty/R}\otimes_R R_f \to (A\otimes_RR_f)^{p^\infty/R_f}$ is an isomorphism.
\hfill $\square$
\end{lemma}

\subsection{Perfection over artinian rings}
\label{subsection:artinian base}

In this subsection we consider the case where $R$ is an artinian
ring. For such a ring, Theorem~\ref{theo_affine_etale_hull}
implies that any flat, finitely generated algebra $R\to A$
has a largest \'etale subalgebra $A^{\et}$. Below we prove that
the natural map $A^{\et}\to A^{p^\infty}$ to the preperfection
is an isomorphism. In particular, the preperfection
is perfect, hence a perfection. We point out that in this special
situation the separability of $R\to A$ is not needed.

\begin{theo}\label{theorem:artinian base case}
Let $R$ be an artinian local ring of characteristic $p$,
and let $A$ be a flat $R$-algebra of finite type. Then the
maps $A^{\et}\to \cO(\pi_0(A)) \to A^{p^\infty}$ are isomorphisms.
\end{theo}

\begin{proof}
It follows from \cite{Rom11}, 2.1.3 that $\pi_0(A)$ is an \'etale
quasi-compact $R$-algebraic space. Since $R$ is artinian, this
space is finite. In particular it is affine and the map
$A^{\et}\to \cO(\pi_0(A))$ is an isomorphism. It remains to prove
that $A^{\et}\to A^{p^\infty}$ is an isomorphism. The proof of this
is in five steps.

\medskip

\noindent Step 1. We reduce to the case where $R=k$ is a field.
Let $\mathfrak m$ resp. $k$ be the maximal ideal resp. residue field.
Let $\Frob:R\to R$ be the absolute Frobenius and $e$ an integer
such that $\mathfrak m=\ker\Frob^e$. Then $\Frob^e$ induces a ring
map $\alpha:k\to R$ which we use to view $R$ as a $k$-algebra.
We compute the perfection of $A$ using the cofinal system
of indices $e\NN\subset\NN$. For each $i\ge 0$ the morphism
$\Frob^{ei}:R\to R$ has a factorization:
\[
\begin{tikzcd}[column sep=8mm]
R \arrow[r, twoheadrightarrow]
& k \arrow[rr, "\Frob^{e(i-1)}"]
& & k \arrow[r, "\alpha"] & R.
\end{tikzcd}
\]
Writing $A_0=A\otimes_Rk$, it follows that
$A^{p^{ei}/R}=A_0^{p^{e(i-1)}/k}\otimes_k R$. Passing to
the limit and
using~\ref{lemma:properties of base change map}~(1),
we deduce an isomorphism:
\[
\begin{tikzcd}[column sep=scriptsize]
\lambda:A_0^{p^{\infty}/k}\otimes_k R \arrow[r, "\sim"]
& A^{p^{\infty}/R}.
\end{tikzcd}
\]
On the other hand, the $e$-fold absolute Frobenius
$\Frob^e_A:A_0^{\et\!/k} \to A^{\et\!/R}$ extends the map
$\alpha:k\to R$, providing an isomorphism:
\[
\begin{tikzcd}[column sep=scriptsize]
\mu:A_0^{\et\!/k}\otimes_{k,\alpha} R \arrow[r, "\sim"]
& A^{\et\!/R}.
\end{tikzcd}
\]
Since $\lambda$ and $\mu$ fit together in a commutative square,
the reduction step follows.

\medskip

\noindent Step 2. We reduce to the case where $k$ is algebraically closed.
Let $k'$ be an algebraic closure of~$k$, and $A'\defeq A\otimes_kk'$.
We have injections
\[
A^{\et/k}\otimes_kk' \intoo A^{p^\infty/k}\otimes_k k'
\intoo (A')^{p^\infty/k'}
\]
where the first is deduced from
$A^{\et/k} \into A^{p^\infty/k}$ and the second comes from
case (2.i) of Lemma~\ref{lemma:properties of base change map}.
It is classical that $A^{\et/k}\otimes_k k'=(A')^{\et/k'}$,
see \cite{Wa79}, Th.~6.5.
It follows that if
$(A')^{\et/k'}\to (A')^{p^\infty/k'}$ is an isomorphism, then
$A^{\et/k}\otimes_kk' \into A^{p^\infty/k}\otimes_k k'$ is
an isomorphism and hence $A^{\et/k}\to A^{p^\infty/k}$ is an
isomorphism.

\medskip

\noindent Step 3. We reduce to the case where $A$ is reduced. Let
$A_{\red}$ be the reduced quotient.
On the separable closure side, since $A^{\et/k}$ does not meet
the nilradical $\Nil(A)$ and all separable elements of $A_{\red}$
lift to $A$, we have an isomorphism
$A^{\et/k}\isomto (A_{\red})^{\et/k}$. On the preperfection side,
we use the isomorphisms
$A^{p^i/k}\isomto A$, $a\otimes \lambda\mapsto a\lambda^{p^{-i}}$
to obtain an isomorphism of rings
$A^{p^\infty/k}\isomto A^{p^\infty/\FF_p}$, and similarly for
$A_{\red}$. Since $\Nil(A)$ is finitely generated, there is $e\ge 0$
such that $\Nil(A)=\ker \Frob^e$ where $\Frob:A\to A$ is the
absolute Frobenius. Then the computation of the
perfection can be carried out along the cofinal system of indices
$e\NN\subset\NN$, showing that the projection
$A^{p^\infty/\FF_p}\to (A_{\red})^{p^\infty/\FF_p}$
is an isomorphism. Contemplating the commutative diagram below,
we see that if $(A_{\red})^{\et/k}\to (A_{\red})^{p^\infty/k}$ is
an isomorphism then $A^{\et/k}\to A^{p^\infty/k}$ also.
\[
\begin{tikzcd}[row sep=10mm]
A^{\et/k} \ar[d,"{\simeq}"] \ar[r] & A^{p^\infty/k} \ar[r,"{\simeq}"] &
A^{p^\infty/\FF_p} \ar[d,"{\simeq}"] \\
(A_{\red})^{\et/k} \ar[r] & (A_{\red})^{p^\infty/k} \ar[r,"{\simeq}"]
& (A_{\red})^{p^\infty/\FF_p}
\end{tikzcd}
\]

\noindent Step 4. We reduce to the case where $A$ has connected
spectrum. This is straightforward, because if
$A=A_1\times\dots\times A_d$ is the decomposition of
$A$ as a product of rings with connected spectrum, we have
$(\prod A_i)^{\et/k}\simeq \prod A_i^{\et/k}$ and
$(\prod A_i)^{p^\infty/k}\simeq \prod A_i^{p^\infty/k}$.

\medskip

\noindent Step 5. We conclude that $A^{\et/k}\to A^{p^\infty/k}$
is surjective. Let $x$ be an element of the ring
\[
A^{p^\infty/k}\simeq A^{p^\infty/\FF_p}=\cap_{n\ge 0}A^{p^n},
\]
with $x=x_n^{p^n}$ and $x_n\in A$, for each $n$. By noetherianity,
the increasing sequence of ideals $(x_i)$ stabilizes at some $N$.
It follows that $y\defeq x_N$ satisfies $(y)=(y^p)$, in particular
$(y)=(y^2)$. Since $X=\Spec(A)$ is connected, we deduce that $y=0$
or $y$ is a unit; therefore $x=0$ or $x$ is a unit. Let $A_i$ be the
quotients of~$A$ by the minimal primes. Again by connectedness,
the injection $A\into A_1\times\dots\times A_n$ induces a morphism
of groups of units modulo constants
$A^\times/k^\times\into
(A_1^\times/k^\times)\dots\times (A_n^\times/k^\times)$ which is
{\em injective}. It is a classical result of Rosenlicht
(\cite{Ros57}, lemma to Prop.~3) that each $A_i^\times/k^\times$
is a finitely generated free abelian group; hence the
same is true for $A^\times/k^\times$. In particular the class of $x$
in this group cannot be infinitely $p$-divisible, so $x\in k^\times$
and this proves the claim.
\end{proof}

\subsection{Preperfection over noetherian rings}
\label{subsection:noetherian base}

The aim of this section is to generalize the statement that
$\cO(\pi_0(A)) \to A^{p^\infty}$ is an isomorphism to the case
of a general noetherian base ring $R$, in the case of
{\em separable} algebras. The proof proceeds by thickening from
an artinian base to a complete local base, then a Zariski-local
base and then to a general base by induction on the dimension.

\begin{lemma}\label{lemma:complete_local_case}
Let $R$ be a complete noetherian local ring and $A$ a flat
separable $R$-algebra of finite type. Write $\hat{A}$ for the
completion of $A$ with respect to the maximal ideal of $R$,
and write $\hat{\pi}$ for the finite \'etale $R$-scheme built
from $\pi_0(A/R)$ as in the situation of
\S~\ref{noth:completion}.
Then the natural map $\cO(\hat{\pi})\to(\hat{A})^{p^{\infty}/R}$
is an isomorphism.
\end{lemma}

\begin{proof}
Let $\mathfrak m$ be the maximal ideal of $R$.
Write $B=\cO(\hat{\pi})$. For every $n\geq 0$, let
$R_n=R/\mathfrak m^{n+1}$, $A_n=A\otimes_RR_n$,
$B_n=B\otimes_RR_n$. As $B_n=\cO(\pi_0(A_n/R_n))$, for every $n$
we have an inclusion $B_n\into A_n$. Taking the limit over $n$,
and noticing that $B$ is finite over $R$ hence complete, we
obtain an inclusion $B\into \hat{A}$. As $B$ is also \'etale
over $R$, it is in fact contained in $(\hat{A}/R)^{p^\infty}$.

On the other hand, a section to the inclusion
$B\into \hat{A}^{p^{\infty}}$ is given by the map
\[
{\widehat A}^{p^\infty}=\lim_i {\widehat A}^{p^i}
=\lim_i\,(\lim_n A_n)^{p^i}\to \lim_i\lim_n(A_n^{p^i})=\lim_n\lim_i(A_n^{p^i})=\lim_nA_n^{p^\infty}=\lim_nB_n=B.
\]
Here, the second-to-last equality comes from
Theorem~\ref{theorem:artinian base case}.
To complete the proof it suffices to show that
$\hat A^{p^\infty}\to B$ is injective, or that
$(\lim_n A_n)^{p^i}\to \lim_n(A_n^{p^i})$ is injective. The latter
is the completion morphism:
\[
(\hat A)^{p^i}\to \hat{(\hat A)^{p^i}}.
\]
As $R$ and $\hat A$ are both noetherian, the hypotheses of
Theorem~\ref{theo:Frobenius of reduced is injective} are satisfied,
and we deduce that $(\hat A)^{p^i}$ is a subalgebra of $\hat A$.
As the latter is $\mathfrak m$-adically separated (that is,
$\bigcap_{i=1}^n\mathfrak m^i\hat A=0$), so is its subalgebra
$(\hat A)^{p^i}$. Hence the completion morphism above
is injective and we conclude.
\end{proof}

\begin{theo}\label{theo:preperfection_algebra}
Let $R$ be a noetherian ring and $A$ a flat, separable $R$-algebra
of finite type. Then the natural map
\[
\phi\colon \cO(\pi_0(A/R))\too A^{p^\infty/R}
\]
is an isomorphism.
\end{theo}

\begin{proof}
As a first step, we claim that we may reduce to the case of $R$ complete local. Indeed, let $R\to R'$ be the completion of the local ring at some prime $\mathfrak p\subset R$. The morphism $R\to R'$ is flat. We have a map
$$\cO(\pi_0(A\otimes_RR'/R'))=\cO(\pi_0(A/R))\otimes_RR'\to A^{p^\infty}\otimes_RR'\into (A\otimes_RR'/R')^{p^\infty}.$$
The first equality is compatibility of global sections and flat
base change,
the second arrow is $\phi\otimes_RR'$, while the last arrow is injective by \ref{lemma:properties of base change map}. We see that if the composition is an isomorphism, then also the central arrow $\phi\otimes_RR'$ is an isomorphism. As $R_{\mathfrak p}\to R'$ is faithfully flat, the map $\phi\otimes_RR_{\mathfrak p}$ is also an isomorphism. Repeating the argument for all $\mathfrak p\subset R$, we find that $\phi$ is an isomorphism. This proves the claim.

We argue by induction on the dimension of $R$. If $R$ is of
dimension zero, it is a product of finitely many artinian local
rings; we reduce to $R$ local and the result follows by
Theorem~\ref{theorem:artinian base case}.

Now let $d$ be the dimension of $R$, and assume the result true for base rings of dimension at most $d-1$. We may assume $R$ local and complete with respect to its maximal ideal. Let $s$ be the closed point of $\Spec R$, and $V=S\setminus\{s\}$. Notice that $V$ is of dimension $d-1$. Cover $V$ with open affines $U_i=\Spec R_i$. Consider the commutative diagram of solid arrows st
\begin{center}\label{diagram}
\begin{tikzcd}
& A \ar[r, dashed] \ar[d, dashed] \ar[ldd, bend right, dashed] & \widehat A\ar[d]\\
&\prod_i A\otimes_R R_i \ar[r]\ar[d] & \prod_i \widehat A\otimes_RR_i \ar[d]\\
0 \ar[r] & \prod_{i,j}A\otimes_R R_i\otimes_R R_j \ar[r] & \prod_{i,j}\widehat A\otimes_R R_i\otimes_R R_j
\end{tikzcd}
\end{center}
Clearly, $A$ admits natural compatible maps towards the diagram, represented by dashed arrows in the diagram.

Next, we take the preperfection of the diagram. By
Lemma~\ref{lemma:complete_local_case} we have
$\widehat A^{p^\infty}=\cO(\hat\pi)$. Moreover, for every $R$-algebra $R'$, there is a natural map $\cO(\hat\pi\otimes_RR')=\hat A^{p^\infty}\otimes_RR'\to (\hat A\otimes_RR')^{p^\infty}$. Finally, by the induction hypothesis $(A\otimes R_i)^{p^\infty}=\cO(\pi(X_{U_i}/U_i))$. We get a commutative diagram
\begin{center}
\begin{tikzcd}
& A^{p^\infty} \ar[r, dashed]\ar[d, dashed]\ar[ldd, dashed, bend right] & \cO(\hat\pi)\ar[d]\\
&\prod_i \cO(\pi_0(X_{U_i}/U_i)) \ar[r]\ar[d] & \prod_i \cO(\hat\pi_{U_i}) \ar[d]\\
0 \ar[r] & \prod_{i,j}\cO(\pi_0(X_{U_{ij}}/U_{ij}))\ar[r] & \prod_{i,j} \cO(\hat\pi_{U_{ij}})
\end{tikzcd}
\end{center}
where the horizontal arrows are those induced by the natural
morphism $\psi\colon\hat\pi\to\pi_0(X/S)$
of \S~\ref{noth:completion}. The limit of the diagram of solid
arrows coincides with the limit of the subdiagram of solid
arrows in the commutative diagram
\begin{equation}\label{diagram_rings}
\begin{tikzcd}
A^{p^\infty} \ar[d, dashed] \ar[r, dashed]& \cO(\hat\pi)\ar[d]\\
\cO(\pi_0(X_V/V)) \ar[r] & \cO(\hat\pi_V). \\
\end{tikzcd}
\end{equation}
Taking global sections in the pushout diagram of
Lemma~\ref{prop:pushout}, we see that $\cO(\pi_0(X/S))$ is a fibre product for the subdiagram \eqref{diagram_rings} of solid arrows. Therefore we get a natural map
$\chi\colon A^{p^\infty}\to \cO(\pi_0(X/S))$. The maps
\[
A^{p^\infty}\stackrel{\chi}{\too} \cO(\pi_0(X/S))
\stackrel{\phi}{\too} A^{p^\infty}
\]
are compatible with the natural inclusions of $A^{p^\infty}$
and $\cO(\pi_0(X/S))$ into $A$. Hence $\phi$ is injective,
and because $\phi\circ\chi$ is the identity, it is also
surjective, as we wished to show.
\end{proof}

With the notation of \ref{theo:preperfection_algebra}, the
algebraic space $\pi_0(X/S)$ is \'etale; however, its $R$-algebra
of global sections $\cO(\pi_0(X/S))$ may fail to be unramified
(and therefore \'etale and perfect); see for instance
Example~\ref{example_non_perfect}. In particular, the
preperfection $A^{p^\infty/R}$ needs not be perfect.

\subsection{Perfection over regular or unibranch one-dimensional rings}
\label{subsection:regular base}

Recall from Remark~\ref{rema:Frobenius of preperfection} that
if $R$ is regular and $\Frob$-finite, then for all $R\to A$
the preperfection $A^{p^{\infty}}$ is perfect. For the separable
$R$-algebras that we have been studying in this section,
Theorem~\ref{theo:preperfection_algebra} provides an explicit description
of $A^{p^\infty}$ which allows to find more cases when preperfection is perfect.

\begin{coro} \label{coro:ideal_situation_regular_case}
Let $R$ be a noetherian $\FF_p$-algebra and $A$ a flat and separable
$R$-algebra of finite type.
\begin{trivlist}
\itemn{1} If $R$ is either
\begin{itemize}
\item geometrically $\QQ$-factorial (e.g. regular), or
\item integral, geometrically unibranch and one-dimensional,
\end{itemize}
then we have isomorphisms:
\[
A^{\et} \isomto \cO(\pi_0(A)) \isomto A^{p^\infty}.
\]
In particular $A^{p^\infty}$ is \'etale, hence perfect and of
finite type.
\itemn{2} If $R$ is reduced, excellent, of dimension $\leq 1$,
then $A^{p^\infty}$ is quasi-finite, and in particular of finite type.
\end{trivlist}
\end{coro}

\begin{proof}
(1) We proved that the map $A^{\et} \to \cO(\pi_0(A))$ is an
isomorphism in
Proposition~\ref{prop:set_theoretically_factorial_etale},
and that the map $\cO(\pi_0(A)) \to A^{p^\infty}$ is an
isomorphism in Theorem~\ref{theo:preperfection_algebra}.

\smallskip

\noindent (2) This follows immediately from
Proposition~\ref{prop_curve_qf}.
\end{proof}

\subsection{Examples} \label{subsection:examples}

We shall see that the coperfection of the
spectrum of an algebra is not the spectrum of its perfection.
In fact, in the flat and separable case the coperfection of an
affine scheme is $\pi_0$ and may be non-separated. Here is an example.

\begin{lemma}\label{example_1}
Let $R=\FF_p[[u]]$ and consider the $R$-algebra
\[A=\frac{R[x,y,(x-y)^{-1}]}{(xy-u)}.\]
Then $A^{p^\infty}=R$ while $\pi_0(A/R)$ is the non-separated
scheme obtained by glueing two copies of $\Spec(R)$ along the
generic fibre.
\end{lemma}

\begin{proof}
Let $X=\Spec A$, $S=\Spec R$. The fibre of $X\to S$ over the closed point
has two connected components, while the generic fibre is connected.
The two sections $s_1,s_2\colon S\to X$, $s_1=\{x=u, y=1\}$ and
$s_2=\{x=1,y=u\}$ meet all components of all fibres; it follows that
the composition
\[
S\sqcup S\xrightarrow{s_1,s_2} X\to \pi_0(X/S)
\]
is given by glueing the two copies of $S$ along the generic fibre.
Therefore $\pi_0(X/S)$ is non-separated. From
\ref{theo:preperfection_algebra} it follows
that $A^{p^\infty}=\cO(\pi_0(X/S))$, which is equal to $R$.
\end{proof}

The following is the most basic example of a non-perfect
preperfection, that is, an $R$-algebra~$A$ which is flat,
separable, of finite presentation, for which the
preperfection $A^{p^\infty/R}$ is not perfect. The ring
$R$ is one-dimensional; we remark that, in accordance with
Proposition~\ref{prop:set_theoretically_factorial_etale},
we need to choose $R$ with
multiple branches. Since the preperfection is not perfect,
it is natural to ask what happens if we take the
preperfection once more. Here is the answer.

\begin{lemma}\label{example_non_perfect}
Let $R=\FF_p[[u,v]]/(uv)$ and
$A=R[x,y,(x-y)^{-1}]/(xy-u)$.
If $p\ne 2$, we have:
\begin{itemize}
\item[\rm (1)] $A^{p^\infty}\simeq\frac{R[\alpha]}{(u\alpha, v^2-\alpha^2)}$
mapping to $A$ by $\alpha\mapsto v\frac{x+y}{x-y}$,
\item[\rm (2)] $(A^{p^\infty})^{p^\infty}\simeq R$.
\end{itemize}
\end{lemma}

Notice that the restriction of $R\to A^{p^\infty}$ to the
branch $\{u=0\}$ is $\FF_p[[v]]\to \FF_p[[v]][\alpha]/(v^2-\alpha^2)$
which is not formally \'etale. Therefore $\phi$ itself is not
formally \'etale and in particular not relatively perfect.
The restriction $p\ne 2$ allows a simpler presentation of
$A^{p^\infty}$ but is inessential.

\begin{proof}
Once for all we set $k=\FF_p$.

\smallskip

\noindent
(1) Let $S=\Spec R$, $X=\Spec A$. The open
complement $V=S\setminus\{s\}$ of the closed point of $S$ is
affine, with $\cO(V)=R_u\times R_v$. It is easy to see that
$\cO(\pi_0(X_V/V))=R_u\times R_v\times R_v$. The inclusion
$\cO(\pi_0(X_V/V)\into \cO(X_V)=A_{u}\times A_{v}$ maps the
elements $(1,0,0)$, $(0,1,0)$, $(0,0,1)$ to $(1,0)$,
$(0,\frac{x}{x-y})$, $(0,\frac{y}{y-x})$ respectively.

Applying the global sections functor to the pushout diagram of
Lemma~\ref{prop:pushout}, and noticing that
$\hat\pi=\Spec(R\times R)$,
we obtain a cartesian diagram
\begin{equation}
\begin{tikzcd}
\cO(\pi_0(X/S)) \ar[r]\ar[d] & R\times R \ar[d]\\
R_u\times R_v\times R_v\ar[r] & (R_u\times R_v)\times (R_u\times R_v)
\end{tikzcd}
\end{equation}
The lower horizontal map sends $(f(u),g(v),h(v))$ to $(f(u),g(v),f(u),h(v))$. The right vertical map sends $(\alpha(u,v),\beta(u,v))$ to $(\alpha(u,0),\alpha(0,v),\beta(u,0),\beta(0,v)).$ It follows that the fibre product $\cO(\pi_0(X/S))$ is the subring of $R_u\times R_v\times R_v$ generated as an $R$-submodule by $(1,1,1)$, $(u,0,0)$, $(0,v,0)$, $(0,0,v)$.

Since $p\neq 2$, we may choose instead $(1,1,1)$, $(u,0,0)$, $(0,v,v)$ and $(0,v,-v)$ as generators. We find that
$$A^{p^\infty/R}\cong \frac{k[[u,v]][\alpha]}{uv,u\alpha,\alpha^2-v^2}=\frac{R[\alpha]}{u\alpha, \alpha^2-v^2}$$
via the map $(u,0,0)\mapsto u$, $(0,v,v)\mapsto v$, $(0,v,-v)\mapsto \alpha$.

Finally, notice that the element $(0,v,-v)\in R_u\times R_v\times R_v$ is mapped to $v\frac{x+y}{x-y}$ in $A_u\times A_v$. This proves the claim.

\smallskip


\noindent (2) Let $B=A^{p^{\infty}}$.
Notice first that any element of $B$ can be written uniquely as $f+g\alpha$, with $f\in R$ and $g\in R/u$. Therefore, any element of $B^{(p^n)}=B\otimes_{R,F^n}R$ takes either the form $1\otimes f$ with $f\in R$ or $\alpha\otimes g$ with $g\in R/u^{p^n}$. In fact, the map of $R$-modules
\[
\begin{array}{rcl}
B^{(p^n)} & \too & R\oplus R/u^{p^n} \\
1\otimes f & \longmapsto & (f,0) \\
\alpha\otimes g & \longmapsto & (0,g)
\end{array}
\]
is an isomorphism, which we will use to rewrite the preperfection diagram of $B$. The $n$-th map in the diagram is $B^{(p^n)}\to B^{(p^{n-1})}$ sending $1\otimes f$ to $1\otimes f$ and $\alpha\otimes g$ to $\alpha^p\otimes g=v^{p-1}\alpha\otimes g=\alpha\otimes v^{p^n-p^{n-1}}g$. Using the isomorphism of $R$-modules above, this becomes the map of $R$-modules
$$G_n\colon R\oplus R/u^{p^n} \to R \oplus R/u^{p^{n-1}}$$
sending $(f,g)$ to $(f,gv^{p^n-p^{n-1}})$.
Consider now the preperfection diagram
\[
\begin{tikzcd}[column sep=12mm]
\ldots \ar[r,"G_{n+1}"] &
R\oplus R/u^{p^n} \ar[r,"G_n"] &
R \oplus R/u^{p^{n-1}} \ar[r,"G_{n-1}"] &
\ldots \ar[r,"G_1"] &
R\oplus R/u.
\end{tikzcd}
\]
Let $H_n=G_1\circ\ldots\circ G_n \colon R\oplus R/u^{p^n}\to R\oplus R/u$ and let $(\ldots,a_n,a_{n-1},\ldots,a_0)$ be an element of the limit of the diagram. We can of course consider the limit in the category of $R$-modules, as it will automatically have an $R$-algebra structure making it into the limit in the category of $R$-algebras. Now, the image of $(f,g)\in R\oplus R/u^{p^n}$ via $H_n$ is $(f,gv^{p^n-1})$. Hence $a_0=(f_0,g_0)$ is such that for every $n \geq 1$, $g_0$ is in the ideal of $R/u$ generated by $v^{p^n-1}$. Therefore $g_0=0$. One can use the same argument to show that for every $a_n=(f_n,g_n)$, $g_n$ vanishes. Therefore the limit is simply the limit of the diagram:
\[
\ldots \stackrel{\id}{\too} R \stackrel{\id}{\too}
R \stackrel{\id}{\too} \ldots \stackrel{\id}{\too} R.
\]
This shows that $B^{p^{\infty}}=R$.
\end{proof}

\section{Unramified $\Frob$-divided objects and
the \'etale fundamental pro-groupoid}
\label{section:coperfection_spaces_stacks}

In this section, we define the \'etale fundamental pro-groupoid
$\sX\to\Pi_1(\sX/S)$ of a flat finitely presented algebraic stack
and we prove Theorem~A, namely that if moreover $\sX/S$ is
separable and $\sM/S$ is a Deligne-Mumford stack,
there is an isomorphism
$\sHom(\Pi_1(\sX/S),\sM) \to \sHom(\sX,\Fdiv(\sM))$.
As a first step, in \ref{ss:coperfection as alg space} we build
on Theorem~\ref{theo:preperfection_algebra} to prove this when
$\sX$ and $\sM$ are algebraic spaces; in this case only the coarse
moduli space $\pi_0(\sX/S)$ appears in the source of the isomorphism.
Then in \ref{subsection:etale_fundamental_stack} we introduce the
\'etale fundamental pro-groupoid  and its basic
properties. Finally in \ref{ss:coperfection as alg stack} we upgrade
the result from algebraic spaces to algebraic stacks in the correct
generality. In order to spare the reader unpleasant technicalities,
some material on groupoid closures needed to handle $\Pi_1(\sX/S)$
is relegated to~\ref{sec:groupoid_closure}.

Note that as we observed in Remark~\ref{rema:coperf determines perf},
the canonical isomorphism
\[
\sHom(\sX,\Fdiv(\sM)) = \sHom(\sX^{\copf},\sM)
\]
allows an equivalent interpretation of the result in terms of
the coperfection of $\sX$. The interplay between the two viewpoints
pervades the section, and the proofs.

\subsection{The case of algebraic spaces}
\label{ss:coperfection as alg space}
Let $S$ be an algebraic space of characteristic $p$
and $X$ a flat, finitely presented, separable $S$-algebraic space.
The algebraic space $\pi_0(X/S)$ is relatively perfect
over~$S$. Therefore the natural morphism
$\Fdiv(\pi_0(X/S)) \too \pi_0(X/S)$
is an isomorphism, and we obtain a natural morphism:
\[
\rho\colon \sX \too \pi_0(\sX/S) \isomto \Fdiv(\pi_0(\sX/S)).
\]
\begin{theo}\label{theo:coperfection_spaces}
Let $S$ be a noetherian algebraic space of characteristic~$p$.
Let $X\to S$ be a flat, finitely presented, separable algebraic
space. Let $M\to S$ be an arbitrary quasi-separated
algebraic space. Then the natural morphism 
given by $\alpha\mapsto \Fdiv(\alpha)\circ \rho$ 
\[
\underline\Hom(\pi_0(X/S),M) \isomto \underline\Hom(X,\Fdiv(M)).
\]
is a bifunctorial isomorphism of sheaves over $S$.
\end{theo}

We make two remarks before giving the proof.

\begin{remas} \label{rem:on coperfection 1}
(1) In terms of coperfection, this theorem says that if
$X\to S$ is a flat, finite type, separable morphism of
noetherian algebraic $\FF_p$-spaces then the inductive
system of relative Frobenii
\begin{center}
\begin{tikzcd}[column sep=12mm]
X \ar[r, "\Frob_{X/S}"] & X^{p/S} \ar[r, "\Frob_{X^p/S}"] &
X^{p^2/S} \ar[r] & \dots
\end{tikzcd}
\end{center}
admits a colimit in the category of quasi-separated algebraic
spaces over $S$; the colimit is the algebraic space
$\pi_0(X/S)$, and is also a coperfection of $X\to S$.

\smallskip

\noindent (2) Point (1) is remarkable if we consider that for
a noetherian ring $R$ and a flat, finite type separable
algebra $R\to A$, taking the preperfection of $A$, i.e., the limit
of relative Frobenius morphisms
\begin{center}
\begin{tikzcd}[column sep=12mm]
\dots \ar[r] & A^{p^2/R} \ar[r, "\Frob_{A^p/R}"] &
A^{p/R} \ar[r, "\Frob_{A/R}"] &
A
\end{tikzcd}
\end{center}
does not guarantee to produce a perfect object, as illustrated in \ref{example_non_perfect}.
\end{remas}

\begin{proof}
Throughout, we write $\pi_0(X)$ instead of $\pi_0(X/S)$.
Let $\rho_0:X\to\pi_0(X)$ be the natural map.
Since $\pi_0(X)\to S$ is perfect, we have a canonical
isomorphism $\underline\Hom(\pi_0(X/S),M)=
\underline\Hom(\pi_0(X/S),\Fdiv(M))$ so the statement to be
proven is that
\[
\Phi\defeq \rho_0^*\colon\underline\Hom(\pi_0(X/S),\Fdiv(M))
\isomto \underline\Hom(X,\Fdiv(M))
\]
is a bifunctorial isomorphism of sheaves over $S$.

We start with easy observations. Obviously we can assume that
$S$ is affine. Since the formation of the $\underline\Hom$
sheaves is compatible with base changes, it is enough to consider
the sections over $S$ and prove that we have a bijection of
$\Hom$ sets. Also, the injectivity part is clear because
$\rho_0$ is an epimorphism of sheaves.

First we reduce to the case $M$ affine. We are free to fix
a morphism $u\colon X\to M$ and prove that~$\Phi$ induces a
bijection between the subsets $\Hom_u(\pi_0(X/S),\Fdiv(M))$
and $\Hom_u(X,\Fdiv(M))$ of maps that induce the same $u$.
Since $X\to S$ is quasicompact, the map~$u$ factors through
a quasicompact open subspace $M'\subset M$, and all maps
in the above $\Hom_u$ subsets factor through $\Fdiv(M')$.
Therefore, replacing $M$ by $M'$ if necessary, we can assume
that $M$ is quasicompact. Let $V\to M$ be an \'etale
surjection with $V$ an affine scheme. Since~$M$ is
quasi-separated, then $g$ is finitely presented; hence the
space $X_V\defeq X\times_MV$ is finitely presented, flat and
separable over $S$. Now start from a map $f:X\to\Fdiv(M)$.
Taking into account that $\Fdiv(V)\isomto\Fdiv(M)\times_MV$,
see~\ref{noth:perfection}, point (v), by pullback along
$V\to M$ we obtain a map $f_V:X_V\to\Fdiv(V)$. By assumption,
since $V$ is affine the map
\[
\Phi_V:\Hom(\pi_0(X_V),\Fdiv(V)) \too \Hom(X_V,\Fdiv(V))
\]
is an isomorphism; hence $f_V$ factors uniquely
via $\pi_0(X_V)$. By the pushout property of
Lemma~\ref{prop:pushout_with_pi_0}, the diagram
\[
\begin{tikzcd}
X_V \ar[r] \ar[d] & X \ar[d] \ar[rdd,bend left] & \\
\pi_0(X_V) \ar[r] \ar[rd] & \pi_0(X) \ar[rd,dashed] & \\
& \Fdiv(V) \ar[r] & \Fdiv(M)
\end{tikzcd}
\]
can be completed by a dashed arrow and the claim is proven.

Now we reduce to the case $X$ affine. Let $U\to X$ be an
\'etale atlas with $U$ affine. Starting from a map $X\to\Fdiv(M)$,
by assumption the composition $U\to X\to\Fdiv(M)$ factors
through $\pi_0(U)$. Using once more the pushout of
Lemma~\ref{prop:pushout_with_pi_0}, the diagram
\[
\begin{tikzcd}
U \ar[r] \ar[d] & X \ar[d] \ar[rdd,bend left] & \\
\pi_0(U) \ar[r] \ar[rrd,bend right=20]
& \pi_0(X) \ar[rd,dashed] & \\
& & \Fdiv(M)
\end{tikzcd}
\]
can be completed by a dashed arrow and this completes
the proof.

To conclude when $S,M,X$ are affine, let $X^{\copf}$
be the coperfection in the sense of sheaves as in~\ref{noth:coperfection}, and compute:
\begin{align*}
\Hom(X,\Fdiv(M))
& =\Hom(X^{\copf},M)
\mbox{ by Remark~\ref{rema:coperf determines perf},} \\
& =\lim\Hom(X^{p^i},M) \\
& =\lim\Hom(\cO_M,\cO_{X^{p^i}}) 
\mbox{ because $M$ is affine,} \\
& =\Hom(\cO_M,\lim \cO_{X^{p^i}}) \\
& = \Hom(\cO_M,\cO(\pi_0(X))) 
\mbox{ by Theorem~\ref{theo:preperfection_algebra},} \\
& = \Hom(\pi_0(X),M)
\mbox{ because $M$ is affine,} \\
& = \Hom(\pi_0(X),\Fdiv(M)).
\end{align*}
\end{proof}

\subsection{The \'etale fundamental pro-groupoid}
\label{subsection:etale_fundamental_stack}

In this subsection, the \'etale fundamental pro-groupoid
$\Pi_1(\mathscr{X}/S)$ of a flat finitely presented algebraic
stacks $\sX/S$ is defined as a \textit{2-pro-object} of
the 2-category of algebraic stacks. Let us recall the
definition of this concept.
For more details, we refer to the paper \cite{DD14}.

\begin{defi}
A nonempty 2-category $\mathcal{I}$ is {\it 2-cofiltered}
if it satifies the following conditions:
\begin{enumerate}
    \item[(1)] Given two objects $i,j\in\mathcal{I}$,
there is an object $k\in\mathcal{I}$ and arrows $k\to i$, $k\to j$;
    \item[(2)] Given two arrows $f,g:j\to i$, there is an
arrow $h:k\to j$ and a 2-isomorphism $\alpha:fh\to gh$;
    \item[(3)] Given two 2-arrows $\alpha,\beta:f\to g$,
where $f,g\in\sHom_{\mathcal{I}}(j,i)$, there is an arrow
$h:k\to j$ such that $\alpha h = \beta h$.
\end{enumerate}
\end{defi}

Clearly, a nonempty $1$-category is cofiltered if and only if it is 2-cofiltered when seen as a $2$-category.

\begin{defi}
A {\it 2-pro-object} of a 2-category $\mathcal{C}$ is a
2-functor $F:\mathcal{I}\to\mathcal{C}$ from a small
2-cofiltered 2-category~$\mathcal{I}$. The 2-category of
2-pro-objects of $\mathcal{C}$ is denoted by
$2{\text -}{\rm Pro}(\mathcal{C})$. The category of morphisms
between two 2-pro-objects $F:\mathcal{I}\to\mathcal{C}$ and
$G:\mathcal{J}\to\mathcal{C}$ is
\[
\sHom_{2{\text -}{\rm Pro}(\mathcal{C})}(F, G) \defeq
\lim_{j\in \mathcal{J}}\underset{i\in \mathcal{I}}{\colim}\ \sHom_{\mathcal{C}}(F(i), G(j))
\]
where $\lim$ (resp. $\colim$) is the pseudolimit (resp. pseudocolimit)
for strict 2-categories, cf. \cite{DD14}, Prop.~2.1.5.
In particular, by a {\it pro-algebraic stack} we mean a
2-pro-object of the 2-category ${\rm\bf AlgStack}$ of
algebraic stacks.
\end{defi}

The index 2-category for defining $\Pi_1$ will be a 2-category
of factorizations similar to that of
Definition~\ref{defi:cats_of_factorizations}, with the difference
that the \'etale part $\sE\to S$ is allowed to be an algebraic
stack rather than an algebraic space. For simplicity, we use again
the notation $\sfE^{\surj}(\mathscr{X}/S)$ although to be fully
consistent, the category defined in~\ref{defi:cats_of_factorizations}
should be denoted $\sfE^{\rm surj,rep}(\mathscr{X}/S)$ to indicate
that $\sE\to S$ is representable by algebraic spaces. No confusion
is likely to occur since the former definition is not used anymore
in the present section of the article.

\begin{defi}
Let $\mathscr{X}/S$ be a flat finitely presented algebraic stack.
We define ${\sf E}^{\rm surj}(\mathscr{X}/S)$ to be the
following 2-category:
\begin{itemize}
\item objects are factorizations $\sX\xrightarrow{h}\sE\to S$
where $\sE/S$ is an \'etale, finitely presented algebraic stack
and $h$ is surjective;
\item 1-arrows
$(\sX\xrightarrow{h}\sE\to S)\to(\sX\xrightarrow{h'}\sE'\to S)$
are pairs $(f,\alpha)$, with $f\colon \sE \to \sE'$ and
$\alpha\colon fh\to h'$ giving a 2-commutative diagram:
\[
\begin{tikzcd}[row sep=tiny]
& \sE \ar[rd] \ar[dd, "f"] & \\
\sX \ar[ru, "h"] \ar[rd, "h'"'] & & S ; \\
& \sE' \ar[ru] &
\end{tikzcd}
\]
\item 2-arrows $(f,\alpha)\to(g,\beta)$ are $2$-morphisms
$u\colon f\to g$ giving a commutative diagram:
\[
\begin{tikzcd}
fh \ar[r, "uh"] \ar[rd, "\alpha"']& gh \ar[d, "\beta"]\\
& h'.
\end{tikzcd}
\]
\end{itemize}
\end{defi}

We emphasize that for a factorization $\sX\to\sE\to S$ in
${\sf E}^{\rm surj}(\mathscr{X}/S)$, the requirement that
$\sE\to S$ be quasi-separated will be crucial in the sequel,
cf Remark \ref{rmk:qs_importance}(1). On the contrary, the
condition of quasi-compactness of $\sE\to S$
is automatic from the same property for $\sX\to S$.

\begin{lemma}\label{lemma:Esurj_cofiltered}
Let $\mathscr{X}/S$ be a flat finitely presented algebraic stack.
The 2-category ${\sf E}^{\rm surj}(\mathscr{X}/S)$ is small
and 2-cofiltered. Moreover, it is equivalent to a 1-category.
\end{lemma}

\begin{proof}
Since $\sX$ and $\sE$ in ${\sf E}^{\rm surj}(\sX/S)$ are all finitely presented, it is standard to deduce that ${\sf E}^{\rm surj}(\sX/S)$ is a small 2-category. Moreover, ${\sf E}^{\rm surj}(\sX/S)$ is nonempty, because it contains the image of $\sX$ in $S$, which is open in $S$ hence \'etale over $S$. Next, we check the three conditions for 2-cofilteredness. \\

\noindent (1) Given two factorizations $h:\sX\to\sE$ and
$h':\sX\to\sE'$, there is the common refinement
$\sX\to\sE\times_S\sE'$ and 2-commutative diagram
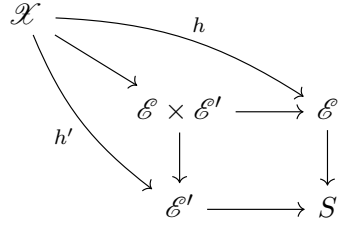
\[
\begin{tikzcd}
\sX \arrow[rrd, bend left=15, "h"] \arrow[rd] \arrow[rdd, bend right=15, "h'" swap] & & \\
& \sE\times\sE' \arrow[r] \arrow[d] & \sE \arrow[d] \\
& \sE' \arrow[r] & S
\end{tikzcd}
\]
Take the image $\sE''$ of $\sX\to\sE\times\sE'$. Then
$\sE''$ is again an \'etale finitely presented $S$-stack and
$h'':\sX\to\sE''$ is a common refinement of $h$ and $h'$ in
${\sf E}^{\rm surj}(\mathscr{X}/S)$. \\

\noindent (2) Given two morphisms $(f,\alpha)$ and $(g,\beta)$
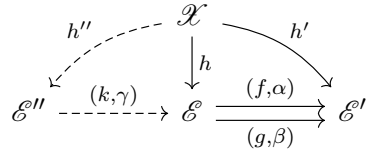
\[
\begin{tikzcd}[column sep=large]
& \sX \arrow[d,"h"] \arrow[ld,dashed,bend right=20,"h''"']
\arrow[rd,bend left=20,"h'"] & \\
\sE'' \arrow[r, dashed, "{(k,\gamma)}"] &
\sE \arrow[r, shift left, "{(f,\alpha)}"]
\arrow[r, shift right, "{(g,\beta)}" swap] & \sE'
\end{tikzcd}
\]
we want to find a third morphism $(k,\gamma):\sE''\to\sE$
and a 2-isomorphism $u:fk\to gk$. For this we consider the
2-fibred product:
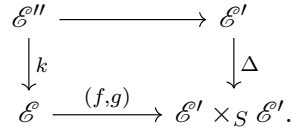
\[
\begin{tikzcd}[column sep=large]
\sE'' \arrow[r] \arrow[d,"k"] & \sE' \arrow[d,"\Delta"] \\
\sE \arrow[r,"{(f,g)}"] & \sE'\times_S\sE'.
\end{tikzcd}
\]
Then $u$ is given by definition. Moreover, the morphisms
$h:\sX\to \sE$ and $h':\sX\to\sE'$ and the 2-commutativity
isomorphisms
\[
fh \stackrel{\alpha}{\tooo} h' \stackrel{\beta^{-1}}{\tooo} gh
\]
provide a morphism $(h,h'):\sX\to \sE''$. \\

\noindent (3) Given two morphisms $(f,\alpha)$, $(g,\beta)$
and two 2-morphisms $u,v:(f,\alpha)\to (g,\beta)$:
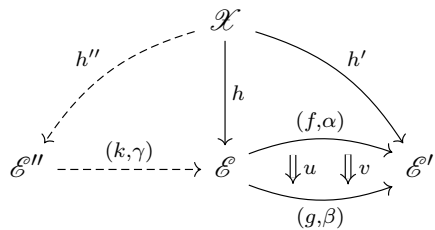
\[
\begin{tikzcd}[column sep=huge,row sep=huge]
& \sX \arrow[d,"h"] \arrow[ld,dashed,bend right=20,"h''"']
\arrow[rd,bend left=20,"h'"] & \\
\sE'' \arrow[r, dashed, "{(k,\gamma)}"] &
\sE \arrow[r, shift left, "{(f,\alpha)}",bend left=20, ""{name=U1,below,pos=0.3}, ""{name=U2,below,pos=0.7}]
\arrow[r, shift right, "{(g,\beta)}" swap,bend right=20, ""{name=D1,above,pos=0.3}, ""{name=D2,above,pos=0.7}] & \sE'
\arrow[Rightarrow, from=U1, to=D1,"\,u"]
\arrow[Rightarrow, from=U2, to=D2,"\,v"]
\end{tikzcd}
\]
we want to find a third morphism $(k,\gamma):\sE''\to\sE$
such that $uk=vk$. For this we view $f$ and $g$ as $\sE$-valued
points of the stack $\sE'$ and $u,v$ as sections of the Isom
functor $I\defeq \underline\Isom_{\sE}(f,g)\to \sE$, that is
$u,v:\sE\to\underline\Isom_{\sE}(f,g)$. Since the diagonal of $\sE'$
is an \'etale morphism, the map $I\to\sE$ is representable and
\'etale, so its diagonal is an open immersion.
We consider the fibred product:
\[
\begin{tikzcd}[column sep=large]
\sE'' \arrow[r] \arrow[d,hook] & I \arrow[d,hook,"\ \Delta"] \\
\sE \arrow[r,"{(u,v)}"] & I\times_{\sE} I.
\end{tikzcd}
\]
The 2-commutativity isomorphisms
\[
fh \stackrel{\alpha}{\tooo} h' \stackrel{\beta^{-1}}{\tooo} gh
\]
provide a morphism $\sX\to I$. Moreover, the conditions
$\beta\circ uh=\beta\circ vh=\alpha$ ensure that
$(uh,vh)=(\beta^{-1}\alpha,\beta^{-1}\alpha)$, that is, we
have a commutative square:
\[
\begin{tikzcd}[column sep=large]
\sX \arrow[r,"\beta^{-1}\alpha"] \arrow[d,"h"'] &
I \arrow[d,hook,"\ \Delta"] \\
\sE \arrow[r,"{(u,v)}"] & I\times_{\sE} I.
\end{tikzcd}
\]
We deduce a morphism $h'':\sX\to\sE''$. Moreover, since we
have the diagram
\[
\begin{tikzcd}
\sX \arrow[r, "h''"] \arrow[rd, two heads, "h" swap] &
\sE'' \arrow[d, hook, "k"] \\
& \sE
\end{tikzcd}
\]
where the map $h$ is surjective, the vertical inclusion is
in fact an isomorphism. Hence the two 2-morphisms $u,v$ are
equalized by an isomorphism $k:\sE''\to\sE$. In particular, it
means that for any such two morphisms $(f,\alpha)$ and $(g,\beta)$,
there is at most one 2-isomorphism between them, thus
${\sf E}^{\rm surj}(\sX/S)$ is equivalent to a 1-category.
\end{proof}

\begin{defi} \label{defi fund pro groupoid}
Let $\mathscr{X}/S$ be a flat finitely presented algebraic
stack. We define the {\em \'etale fundamental pro-groupoid}
$\Pi_1(\mathscr{X}/S)$ of $\sX$ to be the pro-algebraic stack
\setlength{\arraycolsep}{1mm}
\begin{eqnarray*}
\Pi_1(\mathscr{X}/S): {\sf E}^{\rm surj}(\mathscr{X}/S) &
\longrightarrow & {\rm\bf AlgStack}_S \\
\{\sX\to\sE\} & \longmapsto & \sE.
\end{eqnarray*}
\end{defi}

The pro-algebraic stack $\Pi_1(\mathscr{X}/S)$ is pro-\'etale
by definition, and it comes with a canonical morphism
$\sX\to\Pi_1(\sX/S)$ which is unique up to a unique
2-isomorphism. This object defines a 2-functor
\[
\Pi_1: {\rm\bf FlStack}_S \longrightarrow
2{\text -}{\rm Pro}({\rm\bf EtStack}_S)
\]
from the 2-category of flat finitely presented algebraic stacks
over $S$ to the 2-category of pro-\'etale stacks over $S$. It is
tautological from its definition that the 2-functor $\Pi_1(-/S)$
is pro-left adjoint to the inclusion
${\rm\bf EtStack}_S \hookrightarrow {\rm\bf FlStack}_S$.
Finally, if $\sX/S$ is moreover separable, the space of connected
components $\pi_0(\sX/S)$ is a member of the category
${\sf E}^{\rm surj}(\mathscr{X}/S)$. It follows that there is a
morphism $\Pi_1(\mathscr{X}/S)\to \pi_0(\sX/S)$ with target the
constant 2-pro-object. This morphism is easily seen to be universal
for morphisms from $\Pi_1(\mathscr{X}/S)$ to an \'etale algebraic
space; we call it the {\em coarse moduli space}.

\begin{noth}{$\Pi_1$ via smooth atlases}
\label{Pi1 via smooth atlases}
Now let us assume that $\sX$ is separable. Let $U\to\sX$ be a smooth atlas with $U$ finitely presented, and $R=U\times_{\sX}U$. Note that, because of quasi-compactness and quasi-separation of $\sX$, we can always choose $U\to\sX$ to be quasi-compact and quasi-separated. Indeed, we can find a quasi-compact algebraic space $U_0$ as a smooth atlas of $\sX$, then by \cite{SP19} Tag~\href{https://stacks.math.columbia.edu/tag/050Y}{050Y}, $U_0\to\sX$ is quasi-compact. Taking an affine Zariski covering $U\to U_0$ provides an atlas which is quasi-compact and quasi-separated over $\sX$. Now since $\sX$ is finitely presented, a quasi-compact quasi-separated $U\to\sX$ is also finitely presented, hence we can take $\pi_0$ of $U$ and $R$.
In the sequel, for simplicity let us write $\pi_0(U)$ for
$\pi_0(U/S)$. Therefore the
groupoid presentation
$R\rightrightarrows U$ of $\sX$ induces a 2-commutative diagram
\[
\begin{tikzcd}
R \arrow[r, shift left] \arrow[r, shift right] \arrow[d] & U \arrow[r] \arrow[d] & \mathscr{X} \arrow[d] \\
\pi_0(R) \arrow[r, shift left] \arrow[r, shift right] & \pi_0(U) \arrow[r] & \lbrack\pi_0(U)/\pi_0(R)\rbrack
\end{tikzcd}
\]
where $[\pi_0(U)/\pi_0(R)]$ is the quotient stack of the groupoid
closure of the pregroupoid $(\pi_0(R)\rightrightarrows\pi_0(U))$.
For details on pregroupoids and groupoid closures, see
Section~\ref{sec:groupoid_closure}. The construction of groupoid
closures works well for pregroupoids in objects of the category of
\'etale $S$-algebraic spaces, cf. Remark \ref{rmk:groupoid_closure}.
In particular, the groupoid closure
$(\pi_0(R)^{\gpd}\rightrightarrows\pi_0(U))$ is an \'etale groupoid,
and the quotient $[\pi_0(U)/\pi_0(R)]$ is an \'etale stack over $S$,
see Corollary~\ref{coro:coequalizer}.
Since moreover we have a surjection $R\to \pi_0(R)$, the
quasi-compactness of $R$ is inherited by $\pi_0(R)$ and this implies
that $[\pi_0(U)/\pi_0(R)]$ is finitely presented. Hence the
factorization $\sX\to[\pi_0(U)/\pi_0(R)]$
is an object of ${\sf E}^{\rm surj}(\sX/S)$.
\end{noth}

\begin{defi}
Let $\mathscr{X}/S$ be a flat, finitely presented, separable algebraic stack. We define ${\sf E}^{\rm cov}(\mathscr{X}/S)$ to be the full
subcategory of ${\sf E}^{\rm surj}(\mathscr{X}/S)$, which consists
of objects of the form
\[
\sX\to \lbrack\pi_0(U/S)/\pi_0(R/S)\rbrack,
\]
where $U\to\sX$ is a smooth atlas with $U$ finitely presented
and $R\defeq U\times_{\sX}U$.
\end{defi}

\begin{lemma}
The inclusion functor
$i:{\sf E}^{\rm cov}(\mathscr{X}/S)\hookrightarrow
{\sf E}^{\rm surj}(\mathscr{X}/S)$ is initial. In particular,
the full subcategory ${\sf E}^{\rm cov}(\mathscr{X}/S)$ is
cofiltered.
\end{lemma}

\begin{proof}
For the definition of initial functor, see \cite{SP19},
Tag~\href{https://stacks.math.columbia.edu/tag/09WN}{09WN}.
Since $i$ is fully faithful, by the dual version of
Prop.~8.1.3 (c) in \cite{SGA4.1} Expos\'e I, we only need to
verify that any object of ${\sf E}^{\rm surj}$ can be
dominated by an object of ${\sf E}^{\rm cov}$, according to
condition F 1) in {\em loc. cit}.

Let $\{\sX\to\sE\}\in {\sf E}^{\rm surj}(\sX/S)$. Choose an \'etale finitely presented atlas $E\to\sE$, and a smooth finitely presented
atlas $U\to\sX\times_{\sE}E$. Let $R=U\times_{\sX}U$ and
$F=E\times_{\sE}E$. Since $E,F$ are \'etale $S$-spaces, the two
morphisms $U\to E$ and $R\to F$ factor through their $\pi_0$.
Taking groupoid closures and using functoriality of stack
quotients (\cite{SP19},
Tag~\href{http://stacks.math.columbia.edu/tag/04Y3}{04Y3}),
we obtain a 2-commutative diagram:
\[
\begin{tikzcd}
R \arrow[r, shift left] \arrow[r, shift right] \arrow[d] & U \arrow[r] \arrow[d] & \sX \arrow[d] \\
\pi_0(R) \arrow[r, shift left] \arrow[r, shift right] \arrow[d]
& \pi_0(U) \arrow[r] \arrow[d]
& \lbrack\pi_0(U)/\pi_0(R)\rbrack \arrow[d] \\
F \arrow[r, shift left] \arrow[r, shift right] & E \arrow[r] & \sE
\end{tikzcd}
\]
The right column is a morphism in ${\sf E}^{\rm surj}(\sX/S)$,
hence
$\Hom_{{\sf E}^{\rm surj}}
\big(i\big([\pi_0(U)/\pi_0(R)]\big),\sE\big)\neq\emptyset$
and $i$ is an initial functor.
\end{proof}

Therefore the cofiltered category ${\sf E}^{\rm cov}(\sX/S)$,
seen as a 2-cofiltered 2-category, defines the same object
$\Pi_1(\sX/S)$ inside the 2-category
$2{\text -}{\rm Pro}({\rm\bf EtStack}_S)$:
\[
\Pi_1(\sX/S)\ \defeq \ \lim_{{\sf E}^{\rm surj}(\sX/S)}~\sE
\ =\ \lim_{{\sf E}^{\rm cov}(\sX/S)}~[\pi_0(U)/\pi_0(R)].
\]
Note that the stacks $[\pi_0(U)/\pi_0(R)]$ are \'etale
gerbes over the algebraic space $\pi_0(U)/\pi_0(R)=\pi_0(\sX/S)$.
The expression as a limit over ${\sf E}^{\rm cov}(\sX/S)$
is sometimes useful for computing $\Pi_1$.


\begin{prop} \label{prop:Pi_BG}
Let $G$ be a smooth group scheme over $S$. Then we have a
canonical isomorphism $\Pi_1(BG/S)\simeq B(\pi_0(G)/S)$.
In particular, the formation of $\Pi_1$ commutes with base
change in the special case of classifying stacks.
\end{prop}

\begin{proof}
Let $U\to BG$ be a finitely presented smooth atlas, this determines a $G$-torsor $P\to U$. Consider the refinement $P\to U$ of atlases
\[
\begin{tikzcd}
P\times_U P \arrow[r] \arrow[d] & P \arrow[r] \arrow[d] & S \arrow[d] \\
P \arrow[r] & U \arrow[r] & BG
\end{tikzcd}
\]
since $P\times_U P\simeq G\times_S P$, the left vertical arrow is a trivial $G$-torsor. Hence any smooth atlas of $BG$ is refined by an atlas corresponding to a trivial torsor, we may therefore assume that $U\to BG$ corresponds to a trivial $G$-torsor. Equivalently, it means that there is a factorization $U\to S\to BG$. From the following cartesian squares
\[
\begin{tikzcd}
U\times_S U\times_S G \rar\dar\drar[phantom, "\square"] \arrow[d] & U\times_S G \rar\dar\drar[phantom, "\square"] \arrow[d] & U \arrow[d] \\
U\times_S G \rar\dar\drar[phantom, "\square"] \arrow[d] & G \rar\dar\drar[phantom, "\square"] \arrow[d] & S \arrow[d] \\
U \arrow[r] & S \arrow[r] & BG
\end{tikzcd}
\]
we have $U\times_{BG}U\simeq U\times_S U\times_S G$. Hence the groupoid presentation of $BG$
\[
\begin{tikzcd}
U\times U\times G \arrow[r, shift left] \arrow[r, shift right] & U \arrow[r] & BG
\end{tikzcd}
\]
gives rise to the quotient stack
\[
\lbrack\pi_0(U)/\pi_0(U\times U\times G)\rbrack \simeq \lbrack\pi_0(U)/\pi_0(U)\times \pi_0(U)\times \pi_0(G)\rbrack \simeq B(\pi_0(G)/S)
\]
Since these atlases of trivial torsors are initial among all smooth atlases of $BG$, and the corresponding \'etale quotient stacks are initial in ${\sf E}^{\rm cov}(BG/S)$, we deduce the canonical isomorphism $\Pi_1(BG/S)\simeq B(\pi_0(G)/S)$.
\end{proof}

In the final part of this subsection, we explain the relation between $\Pi_1(\mathscr{X}/S)$ and the \'etale fundamental
gerbe of Borne and Vistoli \cite{BV15}, when the base $S=k$
is a field.
In {\it loc. cit.}, the authors introduced the notion of
{\em inflexible stack} over a field $k$. This notion extends
immediately to the case when the base is a finite product
of fields, e.g. a finite reduced $k$-scheme. In particular,
a separable geometrically connected stack of finite type
over a reduced $k$-scheme is inflexible and has an \'etale
fundamental gerbe (\cite{BV15}, Prop.~5.5, Th.~5.7).

\begin{prop}  \label{prop:fundamental gerbe}
{\rm (1)} Let $S$ be an artinian local scheme. Then in
the 2-category of stacks, the pro-algebraic stack
$\Pi_1(\mathscr{X}/S)$ is representable by a stack
which is an fpqc affine gerbe over $\pi_0(\sX/S)$.

\smallskip

\noindent {\rm (2)}
Let $k$ be a field, and $\sX$ a separable $k$-stack of finite
type. Let $\Pi^{\et}_{\mathscr{X}/k}\to \pi_0(\sX/k)$ denote
the \'etale fundamental gerbe of $\sX \to \pi_0(\sX/k)$
as defined in {\rm \cite{BV15}, \S~8}.
Then $\Pi^{\et}_{\mathscr{X}/k}$ is the fpqc
affine gerbe that represents the pro-algebraic stack
$\Pi_1(\mathscr{X}/k)$ in the 2-category of stacks.
\end{prop}

\begin{proof}
(1) For each smooth atlas $U\to\sX$ of finite presentation,
the \'etale stack $[\pi_0(U)/\pi_0(R)]$ has coarse moduli
space $\pi_0(\sX/S)$. If $S$ is local artinian, then each
quasi-finite $S$-space is in fact a finite $S$-scheme.
In particular, $\pi_0(\sX/S)$ is artinian and
$[\pi_0(U)/\pi_0(R)]$ is an affine flat gerbe over it.
It follows from \cite{BV15}, Prop.~3.7 that the stack
which represents the projective system
$\Pi_1(\mathscr{X}/S)$ is an fpqc affine gerbe
over $\pi_0(\sX/S)$.

\smallskip

\noindent (2)
Let $\Pi$ be the fpqc affine gerbe that represents the
$\Pi_1(\mathscr{X}/k)$. From the fact that
$\Pi_1(\mathscr{X}/k)$ has coarse moduli space $\pi_0(\sX/k)$,
the same follows for $\Pi$. Then we see that both
$\sX\to\Pi$ and $\sX\to \Pi^{\et}_{\mathscr{X}/k}$
are universal among morphisms from $\sX$ to an \'etale
$\pi_0(\sX/k)$.
\end{proof}

\subsection{Pushout along $\pi_0$ of an atlas}

The key fact allowing to upgrade our result to algebraic
stacks is an analogue of the pushout property from
Lemma~\ref{prop:pushout_with_pi_0}. We establish it in
Lemma~\ref{lemma:stacky-pushout} below. For this,
we will use a strengthening of the property that
$\sX\to\pi_0(\sX/S)$ is initial for morphisms from $\sX$
to \'etale $S$-algebraic spaces.

\begin{lemma}\label{pi0 initial to unramified}
Let $\mathscr{X}/S$ be a flat, finitely presented, separable algebraic stack. Then $\sX\to\pi_0(\sX/S)$ is initial for
morphisms from $\sX$ to unramified $S$-algebraic spaces.
\end{lemma}

\begin{proof}
Let $f:\sX\to I$ be a morphism to an unramified $S$-algebraic
space $I$. According to \cite{Rom11}, Th.~2.5.2 the algebraic
space $\pi_0(\sX/S)$
is the quotient of $\sX$ by the open equivalence relation
whose graph $\sR\subset\sX\times_S\sX$ is the open connected component
of the diagonal. Therefore, in order to obtain a factorization
$\pi_0(\sX/S)\to I$ it is enough to prove that $f\pr_1=f\pr_2$ where
$\pr_1,\pr_2:\sR\to\sX$ are the projections. Let $\sZ\to\sR$ be the
equalizer of $f\pr_1$ and $f\pr_2$. Since $I$ is unramified, $\sZ$
is an open substack of $\sR$. Moreover, in each fibre above a point
$s\in S$, we have $\sZ_s=\sR_s$ because $I_s$ is \'etale over the
residue field $k(s)$ and $\sX_s\to \pi_0(\sX_s/k(s))$ is initial for
maps to \'etale $k(s)$-spaces (note that the formation of $\pi_0$
commutes with arbitrary base change). Therefore $\sZ=\sR$, so
$f\pr_1=f\pr_2$ and we are done.
\end{proof}

\begin{lemma}\label{lemma:stacky-pushout}
Let $\mathscr{X}/S$ be a flat, finitely presented, separable algebraic
stack and $U\to\sX$ a faithfully flat, finitely presented, separable
atlas (e.g. a smooth surjective atlas of finite presentation). Let $R\rightrightarrows U$
be the corresponding groupoid presentation of $\mathscr{X}$.
Consider the $2$-commutative diagram
\[
\begin{tikzcd}
U \arrow[r] \arrow[d] \arrow[dr, phantom, "\ulcorner"] & \mathscr{X} \arrow[d] \\
\pi_0(U) \arrow[r] & \lbrack\pi_0(U)/\pi_0(R)\rbrack
\end{tikzcd}
\]
and let $\sM\to S$ be either
\begin{trivlist}
\itemm{i} a Deligne-Mumford stack, or
\itemm{ii} $\sM=\Fdiv_S(\sN)$ for some algebraic stack $\sN\to S$.
\end{trivlist}
Then the natural functor
\[
\sHom(\lbrack\pi_0(U)/\pi_0(R)\rbrack,\sM)
\too
\sHom(\sX,\sM)\underset{\sHom(U,\sM)}{\times}
\sHom(\pi_0(U),\sM)
\]
is an equivalence of categories.
\end{lemma}

\begin{proof}
Throughout the proof we write $\sQ=[\pi_0(U)/\pi_0(R)]$
the quotient stack of the pregroupoid $\pi_0(R)\toto\pi_0(U)$.
First we explain precisely what is the functor $F$ of
the statement. The target of $F$ is the category with objects
the triples
$(v:\sX\to\sM,f:\pi_0(U)\to\sM,\delta:v\pi\isomto fh)$, or
in other words the 2-commutative diagrams
\[
\begin{tikzcd}[column sep=30,row sep=30]
U \ar[r,"\pi"] \ar[d,"h"'] & \sX \ar[ld,"\delta"',Rightarrow]
 \ar[d,"v"] \\
\pi_0(U) \ar[r,swap,"f"] & \sM.
\end{tikzcd}
\]
For $\sM=\sQ$, we have a canonical particular object of
this category (see~\ref{Pi1 via smooth atlases}):
\[
\begin{tikzcd}[column sep=30,row sep=30]
U \ar[r, "\pi"] \ar[d,swap, "h"]
& \sX \ar[d, "v_0"] \ar[ld,Rightarrow,"\gamma"'] \\
\pi_0(U) \ar[r,swap,"f_0"] & \sQ.
\end{tikzcd}
\]
Here is how the functor $F$ is defined. For a morphism
$g:\sQ\to\sM$, we have:
\[
F(g)=\big(v=gv_0,f=gf_0,\delta=g\gamma:gv_0\pi \to gf_0h\big).
\]
To construct a quasi-inverse for $F$, we will construct
a functor $G$ such that $GF=\id$, and an isomorphism
$\epsilon:FG\isomto\id$. This means that given
$(v,f,\delta)$, we seek to construct functorially a
morphism $g\colon \sQ\to \sM$
and 2-isomorphisms $a\colon gf_0\to f$, $b\colon gv_0\to v$
filling in a 2-commutative diagram:
\[
\begin{tikzcd}[column sep=40,row sep=40]
U \ar[r, "\pi"] \ar[d,swap, "h"]
& \sX \ar[d, "v_0"] \ar[ld,Rightarrow,"\gamma"']
\ar[rdd, bend left=25, "v",
""{name=N1,below,pos=0.45},
""{name=P1,below,pos=0.77}] & \\
\pi_0(U) \ar[r,swap,"f_0"]
\ar[rrd, bend right=25, swap,"f",
""{name=N3,above,pos=0.47},""{name=P4,above,pos=0.81}] \ar[Rightarrow, from=P1, to=P4, swap, "\delta"{pos=0.2}]
& \sQ
\ar[Leftrightarrow,from=N3,to=N1,"a"{pos=0.25},"b"{pos=0.75},dashed]
\ar[rd,swap,"g"{pos=0.5},dashed,crossing over,
""{name=N4,below,pos=0.18},
""{name=N2,above,pos=0.25}]
& \\
& & \sM.
\end{tikzcd}
\]
We use the usual notations as in
Section~\ref{sec:groupoid_closure} for the groupoid $R\toto U$,
and we complete the picture by adding in the bottom row the
pregroupoid $\pi_0(R)\toto\pi_0(U)$.
\[
\begin{tikzcd}[column sep=40,row sep=30]
R\underset{s,U,t}{\times} R
\ar[d,"l"'] \ar[r, shift left=2,"\pr_1"]
\ar[r, "c" description] \ar[r, shift right=2,"\pr_2"']& R \ar[r, shift left=1,"s"] \ar[r, shift right=1,"t"'] \ar[d,"k"'] & U \ar[r,"\pi"] \ar[d,"h"'] &
\sX \ar[d,"v_0"] \ar[rd, bend left=15,"v"] & \\
\pi_0\left(R\underset{s,U,t}{\times} R\right) \ar[r, shift left=2,"p_1"]
\ar[r, "d" description] \ar[r, shift right=2,"p_2"']& \pi_0(R) \ar[r,shift left=1,"\sigma"] \arrow[r, shift right=1,"\tau"'] &
\pi_0(U) \ar[r,"f_0"] \ar[rr, bend right=30,"f"']
& \sQ & \!\sM
\end{tikzcd}
\]
First we construct the pair $(g,a)$ using
Corollary~\ref{coro:coequalizer} on the coequalizer property of
the stack quotient $\pi_0(U)\to [\pi_0(U)/\pi_0(R)]$ on objects.
Consider $x=f\sigma$ and $y=f\tau$ viewed as $\pi_0(R)$-points
of $\sM$, and $I\defeq \underline\Isom(x,y)$.
Let $\alpha:\pi s\to\pi t$ and $\alpha_0:f_0\sigma\isomto f_0\tau$
be the canonical 2-isomorphisms. The composition
\[
\begin{tikzcd}
f\sigma k=fhs \ar[r,"\delta^{-1}s"] & v\pi s \ar[r,"v\alpha"]
& v\pi t \ar[r,"\delta t"] & fht=f\tau k
\end{tikzcd}
\]
is an isomorphism
$\widetilde\beta\colon k^*x\isomto k^*y$, that is,
a point $\widetilde\beta\colon R\to I$.

We claim that $\widetilde \beta$ factors uniquely via $\pi_0(R/S)$.
We perform separately the two cases of the statement, starting
by case (i), where $\sM\to S$ is Deligne-Mumford. Then
$I\to \pi_0(R)$ is unramified, hence
so is $I\to S$. Lemma~\ref{pi0 initial to unramified}
implies that $\widetilde\beta$ factors uniquely as
\[
R\stackrel{k}{\too} \pi_0(R) \stackrel{\beta}{\too} I.
\]
Next case is (ii), suppose $\sM=\Fdiv(\sN)$. We write
$x_0,y_0\colon \pi_0(R/S)\to \sN$ for the compositions
of $x,y$ with $\Fdiv(\sN)\to \sN$.
Let $I_0\defeq \underline\Isom(x_0,y_0)$.
As $\pi_0(R)\to S$ is perfect,
we may apply Lemma~\ref{lemma:isoms-Fdiv}, and deduce that
\[
I=\Fdiv(I_0).
\]
Then, by Theorem~\ref{theo:coperfection_spaces}:
\[
\Hom_S(R,I)=\Hom_S(R,\Fdiv(I_0))
=\Hom_S(\pi_0(R),I_0)= \Hom(\pi_0(R),\Fdiv(I_0))
=\Hom(\pi_0(R),I).
\]
Therefore $\widetilde \beta\colon R\to I$ factors uniquely
via $\pi_0(R)$.  This completes the proof of the claim.

We have obtained an isomorphism $\beta\colon x\isomto y$.
Now we check that $\beta d=\beta p_1\circ\beta p_2$ holds.
Consider the equality $\alpha c=\alpha \pr_1\circ \alpha \pr_2$:
\[
\begin{tikzcd}[column sep=180]
R\times_{s,U,t} R
\arrow[r, "{\pi s\pr_2=\pi sc}",pos=0.39,bend left=30, ""{name=U1,below,pos=0.505}, ""{name=U2,below,pos=0.755}]
\ar[r, "{\pi t\pr_2=\pi s\pr_1}",pos=0.22,
""{name=A,above,pos=0.47},
""{name=B,below,pos=0.47}]
\arrow[r, "{\pi t\pr_1=\pi tc}" swap,pos=0.35,bend right=30, ""{name=D1,above,pos=0.5}, ""{name=D2,above,pos=0.75}]
& \sX
\arrow[Rightarrow, from=U1, to=A,"\,\alpha \pr_2"]
\arrow[Rightarrow, from=B, to=D1,"\,\alpha \pr_1"]
\arrow[Rightarrow, from=U2, to=D2, crossing over, "\,\alpha c",pos=0.3]
\end{tikzcd}
\]
This gives $v\alpha c=(v\alpha \pr_1)\circ (v\alpha \pr_2)$
which, using the three relations $t \pr_1=tc$, $s \pr_1=t\pr_2$,
$s \pr_2=sc$, we can write:
\[
(\delta tc)\circ (v\alpha c)\circ (\delta^{-1} s c)=
(\delta t \pr_1)\circ (v \alpha \pr_1) \circ (\delta^{-1} s \pr_1)
\circ (\delta t \pr_2) \circ (v \alpha \pr_2) \circ (\delta^{-1} s \pr_2).
\]
Now, by definition
$\widetilde\beta=\delta t\circ v\alpha\circ \delta^{-1} s$
so the above equality becomes
$\widetilde \beta c=\widetilde\beta\pr_1\circ\widetilde\beta \pr_2$
which in turn can be rewritten as
$\beta dl=\beta p_1 l\circ\beta p_2 l$. Finally, because $l$
is faithfully flat hence an epimorphism of spaces, we obtain:
\[
\beta d=\beta p_1\circ\beta p_2.
\]
Then Corollary~\ref{coro:coequalizer} applies and provides
a pair $(g,a)$ and a 2-commutative diagram:
\[
\begin{tikzcd}[column sep=40,row sep=30]
\pi_0(R) \ar[r, "\sigma",""{name=M1,below,pos=0.4}] \ar[d,swap, "\tau",""{name=M2,right,pos=0.5}] & \pi_0(U) \ar[d, "f_0"]
\ar[rdd, bend left=25, "f",""{name=N1,below,pos=0.35},""{name=P1,below,pos=0.75}] & \\
\pi_0(U) \ar[r, "f_0"]\ar[drr, bend right=25, swap, "f",""{name=N4,above,pos=0.53},""{name=P4,above,pos=0.85}] & \sQ
\ar[Rightarrow, from=P1, to=P4, swap, "\delta"{pos=0.2}]
\ar[Leftrightarrow,from=N1,to=N4,"a"'{pos=0.25},
"a"'{pos=0.75},dashed]
\ar[rd,swap,"g"{pos=0.5},dashed,crossing over,""{name=N2,above,pos=0.305},""{name=N3,below,pos=0.265}]
& \\
& & \sM.
\end{tikzcd}
\]
Now we construct $b:gv_0\to v$ using
Corollary~\ref{coro:coequalizer} on the coequalizer property
of $U\to [U/R]$ on morphisms. Define
$c\defeq \delta^{-1}\circ (bh) \circ (g\gamma)$ and consider the
solid diagram:
\[
\begin{tikzcd}[column sep=40,row sep=30]
gv_0\pi s \ar[r,"{g\gamma s}"] \ar[d,"{gv_0\alpha}"']
\ar[rrr,"cs",bend left,dotted,out=60,in=120,distance=1cm] &
gf_0hs \ar[r,"{bhs}"] \ar[d,"{g\alpha_0k}"'] &
fhs \ar[r,"\delta^{-1}s"] \ar[d,"{\widetilde\beta}"'] &
v\pi s \ar[d,"v\alpha"'] \\
gv_0 \pi t \ar[r,"{g\gamma t}"']
\ar[rrr,"ct",bend right,dotted,out=-60,in=-120,distance=1cm]
& gf_0ht \ar[r,"{bht}"'] &
fht \ar[r,"\delta^{-1} t"'] & v\pi t.
\end{tikzcd}
\]
The first square is commutative because
$\alpha_0 k\circ \gamma s=\gamma t\circ v_0\alpha$
by the functoriality of quotient stacks for the morphism
of pregroupoids $(R\toto U)\too (\pi_0(R)\toto\pi_0(U))$.
The second square is commutative by the compatiblity
between $\alpha_0$ and $b$ that results from
Corollary~\ref{coro:coequalizer}. The third square is
commutative by definition of $\widetilde\beta$. Therefore
the outer rectangle is commutative. That is, with the
words of Corollary~\ref{coro:coequalizer}, the arrow
$c$ is a morphism from $(f_1,\beta_1)=(gv_0\pi ,gv_0\alpha)$
to $(f_2,\beta_2)=(v\pi ,v\alpha)$ in the equalizer
category
\[
\eq\,(\sHom(U,\sM) \toto \sHom(R,\sM)).
\]
The quoted corollary gives existence of a 2-isomorphism
$b:gv_0\to v$ such that $c=b\pi$. This concludes the proof
of the lemma.
\end{proof}

\begin{rema}
(1) Lemma~\ref{lemma:stacky-pushout} does not hold if
$\sM$ is an arbitrary Artin stack. In fact,
using Proposition~\ref{prop:Pi_BG} we have
$\Pi_1(B\GG_m/S)=S$ and this implies that the lemma fails
already with $\sX=\sM=B\GG_m$ and $U=S$. For a maybe more
geometric counterexample, let~$k$ be a field and consider
the 2-commutative diagram of $k$-algebraic stacks:
\[
\begin{tikzcd}
U \ar[r] \ar[d] & \mathbb P^1 \ar[d,"\mathcal O(1)"] \\
\pi_0(U) \ar[r] & B\mathbb G_m.
\end{tikzcd}
\]
Here $U=\mathbb A^1\sqcup \mathbb A^1$, $U\to \mathbb P^1$
is the usual affine cover and
$\alpha\colon \mathcal O_U\isomto \mathcal O(1)_{U}$ is some
isomorphism. In this case
$\pi_0(R)=\pi_0(U)\times_{\pi_0(\mathbb P^1)}\pi_0(U)$, and the
two maps towards $\pi_0(U)$ coincide with the projections.
Therefore $[\pi_0(U)/\pi_0(R)]=\pi_0(\mathbb P^1)=\Spec k$.
However, the morphism
$\mathcal O(1)\colon \mathbb P^1\to B\mathbb G_m$ does not
factor via $\Spec k$ since $\mathcal O(1)$ is not trivial.

\smallskip

\noindent (2) A pushout property for $\Pi_1$ (analogous to
Lemma~\ref{prop:pushout_with_pi_0}) can be easily deduced from
the previous lemma, namely, the square
\[
\begin{tikzcd}
U \arrow[r] \arrow[d] \arrow[dr, phantom, "\ulcorner"] & \sX \arrow[d] \\
\Pi_1(U/S) \arrow[r] & \Pi_1(\sX/S)
\end{tikzcd}
\]
satisfies the 2-pushout property for morphisms to Deligne-Mumford
stacks.
\end{rema}

\subsection{The case of algebraic stacks}
\label{ss:coperfection as alg stack}

Finally we prove our main result (Theorem A from the
introduction), building on the case of algebraic spaces
(Theorem~\ref{theo:coperfection_spaces}) and the pushout
along an atlas (Lemma~\ref{lemma:stacky-pushout}).

We will use a lemma about epimorphisms of algebraic stacks.
Since these may fail to be right cancellable as point~(1)
below shows, the claim in~(2) must be estimated at its true value.

\begin{lemma} \label{pullback by epi of spaces}
Let $f:\sS'\to\sS$ be a morphism of algebraic stacks which is
schematically dominant, submersive, and remains so after any
smooth base change.
Let $\sX$ be a stack whose diagonal is representable
by algebraic spaces. Let $u,v:\sS\to\sX$ be morphisms of
stacks.
\begin{trivlist}
\itemn{1} There exist $u,v:\sS\to\sX$ such that $uf=vf$ but
$u$ and $v$ are not isomorphic. Moreover $\sX$ can be chosen
algebraic and $f$ can be chosen
representable, finite, \'etale, surjective.
\itemn{2} Let $a,b:u\isomto v$ be two 2-isomorphisms.
If $f^*a=f^*b$, then $a=b$.
\end{trivlist}
\end{lemma}

\begin{proof}
(1) Let $f:S\to BG$ be the canonical atlas of the classifying
stack of a finite \'etale group scheme~$G$ over a
scheme $S$. Let $a:BG\to S$ be the structure morphism. Let
$u=\id_{BG}:BG\to BG$ and $v=fa:BG\to BG$. Then we have
$af=\id_S$ hence $vf=uf$, but $u$ anv $v$ are not isomorphic
provided~$G$ is chosen so that $H^1(S,G)$ is nontrivial.

\smallskip

\noindent (2) Replacing $\sS'$ by a smooth atlas $S'\to\sS'$,
we can assume that $\sS'=S'$ is a scheme. Consider the
$\sS$-stack of 2-isomorphisms
$I_{\sS}\defeq \underline\Isom(u,v)$.
Then $I_{\sS}$ defines a sheaf over the lisse-\'etale site
of~$\sS$, and we have $a=b$ if and only $a_T=b_T$ for all
objects $T\to\sS$ of that site. Fix such a $T$, and let
$I_T=I_{\sS}\times_{\sS}T$.
Let $T'\defeq T\times_{\sS}S'$. Because $T'$ dominates $S'$, the
assumption $f^*a=f^*b$ implies $a_{T'}=b_{T'}$, that is,
we have two equal compositions:
\[
\begin{tikzcd}[column sep=30]
T' \ar[r] & [-1em] T \ar[r,shift left,"a_T"] \ar[r,shift right,"b_T"'] &
I_T.
\end{tikzcd}
\]
But the assumption on $f$ implies that $T'\to T$ is an
epimorphism of algebraic spaces, see \cite{RRZ18},
Lemma~2.1.5. Hence $a_T=b_T$ as was to be shown.
\end{proof}

Now let $S$ be an algebraic space of characteristic $p$
and $\sX$ a flat, finitely presented, separable $S$-algebraic stack.
As a consequence of Lemma~\ref{lemma:perfect_stack}, the \'etale
fundamental pro-groupoid $\Pi_1(\sX/S)$ is relatively perfect
over~$S$. Therefore the natural morphism
$\Fdiv(\Pi_1(\sX/S)) \too \Pi_1(\sX/S)$
is an isomorphism, and we obtain a natural morphism:
\[
\rho\colon \sX \too \Pi_1(\sX/S) \isomto \Fdiv(\Pi_1(\sX/S)).
\]

\begin{theo} \label{theo:coperfection_stacks}
Let $S$ be a noetherian algebraic space of characteristic~$p$.
Let $\sX\to S$ be a flat, finitely presented, separable algebraic
stack. Let $\sM\to S$ be a quasi-separated Deligne-Mumford stack. Then the functor $\alpha \mapsto \Fdiv(\alpha)\circ \rho$ is
an isomorphism
\[
\sHom(\Pi_1(\sX/S),\sM) \isomto \sHom(\sX,\Fdiv(\sM))
\]
between the stacks of morphisms of pro-Deligne-Mumford stacks
(with $\sM$ seen as a constant 2-pro-object) on the source,
and morphisms of stacks on the target. This isomorphism is
functorial in $\sX$ and $\sM$.
\end{theo}

\begin{rema} \label{rem:on coperfection 2}
In terms of coperfection, this says that if $\mathscr{X}/S$ is
a flat finitely presented separable algebraic stack, the
inductive system of relative Frobenii
\[
\begin{tikzcd}
\mathscr{X} \arrow[r, "\Frob_{\mathscr{X}/S}"] & \mathscr{X}^{p/S} \arrow[r, "\Frob_{\mathscr{X}^p/S}"] & \mathscr{X}^{p^2/S} \arrow[r] & \dots
\end{tikzcd}
\]
admits a colimit in the 2-category of pro-quasi-separated
Deligne-Mumford stacks over $S$, which is the pro-\'etale
stack $\Pi_1(\mathscr{X}/S)$. In particular, $\Pi_1(\mathscr{X}/S)$ is a coperfection of $\mathscr{X}/S$ in the 2-category of
quasi-separated Deligne-Mumford stacks.
\end{rema}

\begin{proof}
As in the proof of Theorem~\ref{theo:coperfection_spaces},
we write $\Pi_1(\sX)\defeq\Pi_1(\sX/S)$, we let
$\rho_0:\sX\to\Pi_1(\sX)$ be the natural map and we want to
prove that
\[
\Phi=\rho_0^*\colon
\sHom(\Pi_1(\sX),\Fdiv(\sM)) \isomto \sHom(\sX,\Fdiv(\sM))
\]
is a bifunctorial isomorphism of stacks over $S$. We assume
that $S$ is affine and we reduce to proving that the map on
global sections over $S$ is an equivalence of categories.

We start with essential surjectivity. Consider an object $f$ of
$\sHom(\sX,\Fdiv(\sM))$. Just like we did in the proof of
Theorem~\ref{theo:coperfection_spaces}, we can fix $u:\sX\to\sM$,
assume $\sM$ quasi-compact, pick an atlas $V\to \sM$ with
$V$ an affine scheme, define $\sX_V\defeq\sX\times_{\sM}V$ and
choose a smooth atlas of finite presentation
$U\to \sX\times_{\sM}V$ such that $\sX$ and $U$ are flat,
finitely presented, separable over $S$. Moreover~$f$ induces
an object $f'\in\Hom(\sX_V,\Fdiv(V))$, and by
precomposition an element $g\in\Hom(U,\Fdiv(V))$.

By Theorem~\ref{theo:coperfection_spaces}, the map $g$ is
induced by a unique morphism $\pi_0(U/S)\to V$; or,
equivalently, a morphism $\pi_0(U/S)\to \Fdiv(V)$. Using
the pushout diagram of Lemma~\ref{lemma:stacky-pushout},
case ii),
\[
\begin{tikzcd}
U \arrow[r] \arrow[dd] \arrow[ddr, phantom, "\ulcorner",pos=0.3]
& \sX \arrow[d] \arrow[rddd, bend left] & \\
& \Pi_1(\sX/S) \arrow[rdd, dashed, bend left=15] \arrow[d] & \\
\pi_0(U) \arrow[r] \arrow[rd] & \lbrack\pi_0(U)/\pi_0(R)\rbrack \arrow[rd, dashed] & \\
& \Fdiv(V) \arrow[r] & \Fdiv(\sM)
\end{tikzcd}
\]
we obtain a map $\Pi_1(\sX/S)\to\Fdiv(\sM)$ and this shows
essential surjectivity.

We pass now to full faithfulness of $\Phi$. Let $f,g$ be objects
of $\sHom(\Pi_1(\sX/S),\Fdiv(\sM))$. We want to prove that the
map $\Hom(f,g) \too \Hom(\Phi(f),\Phi(g))$ is bijective.

We start with surjectivity. Assume given a diagram
\[
\begin{tikzcd}
\sX \ar[r,"\rho_0"] &\Pi_1(\sX/S) \ar[r, shift left=1,"f"]\ar[r,shift right=1,swap,"g"] & \Fdiv(\sM)
\end{tikzcd}
\]
and an isomorphism $\alpha\colon f\rho_0\isomto g\rho_0$.
By the definition of morphisms in the pro-category and
cofilteredness of ${\sf E}^{\rm surj}(\sX/S)$, the
morphisms~$f,g$ as well as $\alpha$ are defined on some
common \'etale finitely presentated stack $\sE$ corresponding
to a surjective factorization $h:\sX\to\sE$. Abusing
notation slightly, we therefore assume that we have
$f,g:\sE\to\Fdiv(\sM)$ and $\alpha\colon fh\isomto gh$.
Our aim is to show that there exists
a refinement $\sX\xrightarrow{h'} \sE' \xrightarrow{l} \sE$ in
${\sf E}^{\rm surj}(\sX/S)$ and a 2-isomorphism
$\beta\colon fl\isomto gl$ such that $\beta h'=\alpha$.
Since $\sE\to S$ is \'etale hence perfect, we have
$\sHom(\sE,\Fdiv(\sM))=\sHom(\sE,\sM)$ canonically, and
similarly for~$\sE'$. We deduce that it is enough to work with
the compositions $f_0,g_0:\sE\to\Fdiv(\sM)\to\sM$.
Indeed, if we find $(h'_0:\sX\to\sE',\beta_0\colon fl\isomto gl)$
for $(f_0,g_0)$, then applying $\Fdiv$ will provide $(h',\beta)$
suitable for $(f,g)$. In sum, changing again notation,
we can start from:
\[
\begin{tikzcd}
\sX \ar[r,"h"] &\sE \ar[r, shift left=1,"f"]\ar[r,shift right=1,swap,"g"] & \sM.
\end{tikzcd}
\]
Letting $\sI\defeq \underline\Isom_{\sE}(f,g)$, we consider
the 2-commutative diagram
\[
\begin{tikzcd}
& \sI \ar[r] \ar[d] & \sM \ar[d, "\Delta"]\\
\sX \ar[ur, "\alpha"] \ar[r] & \sE \ar[r,"{(f,g)}"] &
\sM\times_S\sM.
\end{tikzcd}
\]
The assumption that $\sM\to S$ is Deligne-Mumford guarantees
that the representable morphism $\sI\to \sE$ is unramified.
Now let us pick an \'etale finitely presented atlas $E\to \sE$
with $E\to S$ an \'etale quasi-compact scheme; let us also
choose an fppf separable atlas $V \to \sX\times_{\sE}E$ by
a scheme $V$. Let $I=E \times_{\sE}\sI$. We obtain a
2-commutative diagram
\[
\begin{tikzcd}[row sep=20]
V \ar[d] & & I\ar[d]\ar[ld] \\
\sX \times_{\sE}E \ar[rru, "\alpha'"] \ar[r]\ar[d]
& E\ar[d] & \sI \ar[ld]\\
\sX \ar[rru, "\alpha", crossing over,pos=0.25] \ar[r] & \sE
\end{tikzcd}
\]
with $\alpha'$ induced by $\alpha$ via pullback along
$E\to \sE$. The morphism $I\to E$ is representable and
unramified; therefore $I$ is an unramified algebraic space
over $S$. By Lemma~\ref{pi0 initial to unramified} the map
$V\to I$ factors uniquely via $\pi_0(V/S)$. Letting
$R=V\times_{\sX}V$, we obtain by
Lemma~\ref{lemma:stacky-pushout} a dashed arrow
\[
\begin{tikzcd}
V \ar[r]\ar[d] & \sX \ar[d]\ar[rdd, bend left] \\
\pi_0(V/S) \ar[r]\ar[rrd, bend right=20] & \left[\pi_0(V)/\pi_0(R)\right] \ar[rd,dashed]& \\
& & \sI
\end{tikzcd}
\]
making the diagram 2-commute. Then
$\sX\to \sE'\defeq [\pi_0(V)/\pi_0(R)]$ is the required $h'$,
and the dashed arrow $\sE'\to \sI$ is $\beta$.

We finish with injectivity. Let $a,b:f\to g$ be two morphisms
such that $\rho_0^*a=\rho_0^*b$. Then as before, there is a
factorization $h:\sX\to\sE$ such that $f,g$ and $a,b$ are defined
on $\sE$. Let $f_i,g_i:\sE\to \Frob^i_*\sM$ and $a_i,b_i:f_i\to g_i$
be the $i$-th components of $f,g$ and $a,b$. It is enough to prove
that $a_i=b_i$ for each $i$. Let $f'_i,g'_i:\Frob^{i*}\sX\to \sM$
and $a'_i,b'_i:f'_i\to g'_i$ be the maps deduced by the
$(\Frob^*,\Frob_*)$ adjunction; it is then enough to prove that
$a'_i=b'_i$. Let $k_i=\Frob^{i*}h:\Frob^{i*}\sX\to \Frob^{i*}\sE$,
we have $k_i^*a'_i=k_i^*b'_i$ by assumption. Since~$k_i$ is
faithfully flat and of finite presentation, it satisfies the
assumptions of Lemma~\ref{pullback by epi of spaces} and
we deduce that $a'_i=b'_i$. This concludes the proof.
\end{proof}

\begin{rema} \label{rmk:qs_importance}
(1) Here the assumption of quasi-separation is crucial. If one
considers possibly non-quasi-separated Deligne-Mumford stacks in the
statement of Theorem~\ref{theo:coperfection_stacks}, one may not
be able to find a finitely presented \'etale atlas $V\to\sM$,
and hence the atlas $U\to \sX$ in the theorem may not be chosen
with $U/S$ being finitely presented, consequently one cannot apply
$\pi_0$ to~$U/S$. Here is a counterexample: let $\sX=C$ be a nodal
irreducible curve over a field $k$. It is known that there exists
an infinite \'etale cover of $C$ which does not come from finite
\'etale covers, corresponding to a morphism $C\to B\mathbb{Z}$
to the perfect stack $B\mathbb{Z}$ which is not quasi-separated.
However, by Proposition~\ref{prop:fundamental gerbe} the stack
$\Pi_1(C/k)\simeq\Pi^{\et}_{C/k}$ is profinite, while the morphism
$C\to B\mathbb{Z}$ does not factor through any finite \'etale
stack, hence it does not factor through $\Pi_1(C/k)$.

\smallskip

\noindent (2) If one wants to make the statement above an actual
adjunction, some rather costly strengthenings of the assumptions
are needed. First,
one needs to extend the functors to the 2-pro-categories; this is
no big problem. Second and more seriously, we need $\Fdiv$ to take
values in (the 2-pro-category of) flat, separable algebraic stacks.
This is much more binding; the natural way to ensure this is to
assume that the Frobenius of $S$ is {\em finite locally free}
(e.g. $S$ regular $\Frob$-finite) and~$\sM$ is {\em smooth}.
To sum up, let ${\rm\bf SpbStack}_S$ be the 2-category of
faithfully flat, finitely presented, separable algebraic stacks
and ${\rm\bf SmDM}_S$ the 2-category of smooth Deligne-Mumford
stacks. If $\Frob_S$ is finite locally free, we obtain a pair
of 2-adjoint functors:
    \[
    \begin{tikzcd}[column sep=15mm]
    2{\text -}\Pro({\rm\bf SpbStack}_S) \ar[r,shift left=1.9,"\Pi_1"] \ar[r,phantom,"\mbox{\tiny$\perp$}"] & 2{\text -}\Pro({\rm\bf SmDM}_S). \ar[l,shift left=1.9,"\Fdiv"]
    \end{tikzcd}
    \]
\end{rema}

To give a concrete illustration, we take as example the moduli
stack $\sM=\bar\sM_{g,n}$ of stable curves of genus $g$
with $n$ marked points, with $2g-2+n>0$.

\begin{prop} \label{prop:F_divided_objects}
Let $k$ be a field and let $X/k$ be a geometrically connected,
geometrically reduced scheme of finite type admitting
a $k$-rational point $x\in X(k)$. Set $\sM=\bar\sM_{g,n}$. Let
\[
(\sC_i\to X^{p^i/k},\sigma_i)\in \Fdiv(\sM)(X)
\]
be a divided curve over $X$. Let $C\in\Fdiv(\sM)(k)$ be its pullback
via $x\colon \Spec k\to X$; note that $\Fdiv(\sM)(k)=\sM(k)$
by taking $\sX=\Spec k$ in Theorem~\ref{rem:on coperfection 2}.
Then there exist:
\begin{itemize}
\item a finite \'etale subgroup scheme $G\subset \Aut_k(C)$;
\item a $G$-torsor $f\colon P\to X$;
\end{itemize}
such that the $\Frob$-divided curve on $P$ obtained from pullback of $(\sC_i,\sigma_i)$ via $f\colon P\to X$ is isomorphic to the pullback of $C$ via $P\to \Spec k$.
\end{prop}

\begin{proof}
By Theorem~\ref{theo:coperfection_stacks}, the $\Frob$-divided curve
$(\sC_i,\sigma_i)$ corresponds to an object of
\[
\sHom(\Pi_1(X),\sM)=\underset{X\onto\sE}\colim\,\sHom(\sE,\sM)
\]
and therefore to a $g\in\sHom(\sE,\sM)$ for some factorization
$X\onto \sE\to \Spec k$ in ${\sf E}^{\surj}(X/k)$.

Let $\sE\to E$ be the coarse moduli space. Then $E/k$ is an \'etale algebraic space, and $X\to \sE\to E$ is surjective; we have therefore a factorization $X\to \pi_0(X/k)\onto E$. As $X/k$ is geometrically connected, $\pi_0(X/k)=\Spec k$, and so $E=\Spec k$ as well.

The gerbe $\sE\to E=\Spec k$ has a section induced by $x\in X(k)$:
hence $\sE$ is equivalent to $BG$ for some finite \'etale $k$-group
scheme $G$. The morphism $BG \to \sM$ induced by $g$ is the datum
of a curve $C/k$ in $\sM(k)$ and a left $G$-action on $C$. We may
therefore replace $G$ by its image in $\Aut(C)$.

Let now $P\to X$ be the $G$-torsor associated to $X\to BG$. The $2$-commutative diagram
\[
\begin{tikzcd}
P \ar[r]\ar[d] & \Spec k \ar[d] \\
X \ar[r] & BG
\end{tikzcd}
\]
induces a $2$-commutative diagram
\[
\begin{tikzcd}
\Fdiv(\sM)(P)  & \Fdiv(\sM)(k)= \sM(k) \ar[l] \\
\Fdiv(\sM)(X) \ar[u] & \Fdiv(\sM)(BG) = \sM(BG) \ar[l] \ar[u,shift left=7]
\end{tikzcd}
\]
where the equivalences on the right are due to
Theorem~\ref{theo:coperfection_stacks} and Proposition \ref{prop:Pi_BG}.

As we said, the $\Frob$-divided curve $(\sC_i,\sigma_i)$ is in the essential image of the lower horizontal arrow, and its image in $\Fdiv(\sM)(P)$ is therefore isomorphic to the pullback of a curve $C\in M(k)$.
\end{proof}

\subsection{Appendix: the groupoid closure of a pregroupoid} \label{sec:groupoid_closure}

In this appendix, we give the construction of the groupoid closure of a pregroupoid.

\begin{noth}{Groupoids}
A {\em groupoid} is a small category where every morphism is an
isomorphism. Alternatively, it is given by a set of objects $U$,
a set of arrows $R$, and morphisms source and target
$s,t:R\to U$, composition $c:R\times_{s,U,t} R\to R$, identity
$e:U\to R$, inverse $i:R\to R$, satisfying the following axioms:
\begin{itemize}
\item[(1)] Associativity: $c\circ (1,c)=c\circ (c,1)$,
\item[(2)] Identity: $s\circ e=t\circ e=1$ and
$c\circ (1,e\circ s)=c\circ (e\circ t,1)=1$,
\item[(3)] Inverse: $s\circ i=t$, $t\circ i=s$,
$c\circ (i,1)=e\circ s$ and $c\circ (1,i)=e\circ t$.
\end{itemize}
In a groupoid, the maps $e$ and $i$ are uniquely determined,
$i$ is an involution, and $i\circ e=e$. In particular the quintuple
$(U,R,s,t,c)$ suffices to describe the groupoid.
\end{noth}


\begin{noth}{Symmetry}
Inversion $i$ extends to $n$-tuples of composable arrows:
\[
(R/U)^n\defeq R\times_{s,U,t} R\times_{s,U,t}\dots\times_{s,U,t}R
\quad,\quad
(\alpha_1,\dots,\alpha_n)\mapsto (\alpha_n^{-1},\dots,\alpha_1^{-1}).
\]
This can be used to shrink the number of axioms. Indeed, we have:
\[
c=i\circ c\circ i,
\quad (1,c)=i\circ (c,1)\circ i,
\quad (e\circ t,1)=i\circ (1,e\circ s)\circ i,
\quad (1,i)=i\circ (i,1)\circ i.
\]
Using this, the axioms $t\circ e=1$, $c\circ (e\circ t,1)=1$
and $c\circ (1,i)=e\circ t$ follow from $s\circ e=1$,
$c\circ (1,e\circ s)=1$ and $c\circ (i,1)=e\circ s$,
respectively. Of course the reduced system of axioms has the
drawback that it is not symmetric. In the sequel, for legibility
we will prefer
to give full, symmetric lists of axioms, but we will use symmetry
to reduce the number of constructions. Namely, when we want to
construct a (pre)groupoid and maps $\lambda=(1,e\circ s)$ and
$\lambda^{\mbox{\tiny +}}=(e\circ t,1)$ have to be provided, then
we know that it is enough to construct $\lambda$ since then
$\lambda^{\mbox{\tiny +}}=i\circ \lambda\circ i$.
\end{noth}

\begin{noth}{Pregroupoids: motivation} \label{motivation for pgpds}
Put in a nutshell, a pregroupoid is a structure which resembles
that of a groupoid, but where composition is only partially defined
and associativity holds only partially. More on the technical side,
our working definition will be that a pregroupoid is what you obtain
when you apply a functor to a groupoid. Since this produces a lot
of data, we will describe the motivating example first
in order to make the ensuing Definition~\ref{defi:pregpd}
readable. We simplify notations
by allowing the omission of the ``$\circ$'' sign for compositions.
\end{noth}

\begin{exam} \label{exam:motivating}
Assume that $(U_0,R_0,s_0,t_0,c_0)$ is a groupoid in
objects of a category $\sC_0$. Let $F:\sC_0\to\sC$ be a functor.
If $F$ transforms fibred products into fibred products, then
\[
(F(U_0),F(R_0),F(s_0),F(t_0),F(c_0))
\]
is a groupoid in objects of $\sC$. In general however, all data
and axioms involving fibred products are altered. We will now
describe the result precisely.

\medskip

\noindent
We first look at the data that do not involve fibred products,
that is $U_0,R_0,s_0,t_0,e_0,i_0$. By taking their images under $F$,
we obtain:
\begin{trivlist}
\itemm{1} objects $U,R$ and maps $s,t:R\to U$, $e:U\to R$, $i:R\to R$
such that $se=te=1$, $i^2=1$, $si=t$, $ti=s$.
\end{trivlist}
Now we look at the data involving double fibred products and
composition.
By taking the images of $D_0\defeq (R_0/U_0)^2$, the involution
$i_0:D_0\to D_0$, the projections $\pr_1,\pr_2:D_0\to R_0$, we obtain:
\begin{trivlist}
\itemm{2} an object $D$ and maps $i:D\to D$, $p_1:D\to R$,
$p_2:D\to R$ such that $sp_1=tp_2$ and $p_1i=ip_2$.
\end{trivlist}
By taking the images of the composition $c_0:D_0\to R_0$
and the maps $\lambda_0\defeq (1,e_0s_0):R_0\to D_0$,
$\lambda^{\mbox{\tiny +}}_0\defeq (e_0t_0,1):R_0\to D_0$,
$\mu_0\defeq (i_0,1):R_0\to D_0$,
$\mu^{\mbox{\tiny +}}_0\defeq (1,i_0):R_0\to D_0$ we obtain:
\begin{trivlist}
\itemm{3} maps $c:D\to R$,
$\lambda,\lambda^{\mbox{\tiny +}}:R\to D$,
$\mu,\mu^{\mbox{\tiny +}}:R\to D$ such that
\begin{trivlist}
\itemm{3.a} $p_1\lambda=1$, $p_2\lambda=es$,
$p_1\lambda^{\mbox{\tiny +}}=et$, $p_2\lambda^{\mbox{\tiny +}}=1$,
$\lambda i=i\lambda^{\mbox{\tiny +}}$,
$c\lambda=c\lambda^{\mbox{\tiny +}}=1$,
\itemm{3.b} $p_1\mu=i$, $p_2\mu=1$, $p_1\mu^{\mbox{\tiny +}}=1$,
$p_2\mu^{\mbox{\tiny +}}=i$, $\mu i=\mu^{\mbox{\tiny +}}$,
$c\mu=es$, $c\mu^{\mbox{\tiny +}}=et$.
\end{trivlist}
\end{trivlist}
Finally we look at the data involving triple fibred products and
associativity. By taking the images of $E_0\defeq (R_0/U_0)^3$, the
involution $i_0:E_0\to E_0$, and the projections
$\pr_{12},\pr_{23}:E_0\to D_0$, we obtain:
\begin{trivlist}
\itemm{4} an object $E$ and maps $i:E\to E$, $q_{12},q_{23}:E\to D$
such that $p_2q_{12}=p_1q_{23}$ and $q_{12}i=iq_{23}$.
\end{trivlist}
We define $q_1\defeq p_1q_{12}$, $q_2\defeq p_1q_{23}$, $q_3\defeq p_2q_{23}$.
By taking the images of $\nu_0\defeq (1,c_0):E_0\to D_0$
and $\nu^{\mbox{\tiny +}}_0\defeq (c_0,1):E_0\to D_0$ we obtain:
\begin{trivlist}
\itemm{5} maps $\nu,\nu^{\mbox{\tiny +}}:E\to D$ such that
$p_1\nu=q_1$, $p_2\nu=cq_{23}$, $p_1\nu^{\mbox{\tiny +}}=cq_{12}$,
$p_2\nu^{\mbox{\tiny +}}=q_3$, $\nu i=i\nu^{\mbox{\tiny +}}$
and $c\nu=c\nu^{\mbox{\tiny +}}$.
\end{trivlist}
The axioms of a groupoid survive in modified guise: associativity
is in (5); identity is in (1) and (3.a); inverse is in (1) and (3.b).
Using
symmetry, this set of data is determined by the subcollection
$P\defeq (U,(R,i),(D,i),(E,i),s,c,e,p_1,\lambda,\mu,q_{12},\nu)$.
\end{exam}

\begin{defi} \label{defi:pregpd}
A {\em pregroupoid} (over $U$) is given by a collection of
objects and maps
\[
P=\big(U,(R,i),(D,i),(E,i),s,c,e,p_1,\lambda,\mu,q_{12},\nu\big)
\]
satisfying the conditions (1) to (5) in~\ref{exam:motivating}.
A {\em morphism of pregroupoids} $f:P\to P'$ is given by a
quadruple of maps $U\to U'$, $R\to R'$, $D\to D'$,
$E\to E'$ that are compatible with all the structure maps of
the pregroupoids $P$ and $P'$.
\end{defi}

\begin{remas} \label{rema:gpds and pgpds}
(1) Each groupoid $(U,R,s,t,c)$ defines a unique pregroupoid such
that $D=(R/U)^2$ and $E=(R/U)^3$. This gives rise to a faithful
embedding of categories:
\[
\iota:(\Groupoid/U) \intoo (\Pregroupoid/U).
\]
\noindent (2) A pregroupoid is a groupoid if and only if the following
two maps are isomorphisms:
\[
(p_1,p_2):D\to (R/U)^2 \quad\mbox{and}\quad
(q_{12},q_{23}):E\to D\times_{p_2,R,p_1} D.
\]
\noindent (3) A pregroupoid is a truncated simplicial set:
\[
\begin{tikzcd}[column sep=40]
E \ar[r, shift left=4.5,"q_{12}"]
\ar[r, shift left=1.5,"\nu" description]
\ar[r, shift right=1.5,"\nu^+" description]
\ar[r, shift right=4.5,"q_{23}"']
& D \ar[r, shift left=3,"p_1"]
\ar[r, "c" description] \ar[r, shift right=3,"p_2"']
& R \ar[r, shift left=1.5,"s"] \ar[r, shift right=1.5,"t"'] & U.
\end{tikzcd}
\]
With the usual notations $d_n^i$ and $s_n^i$ for faces
and degenerations, we have $d_3^0=q_{12}$, $d_3^1=\nu$,
$d_3^2=\nu^+$, $d_3^3=q_{23}$, $s_2^0=(\id,e)$,
$s_2^2=(e,\id)$, $d_2^0=p_1$, $d_2^1=c$, $d_2^2=p_2$,
$s_1^0=\lambda$, $s_1^1=\lambda^+$, $d_1^0=t$, $d_1^1=s$,
$s_0^0=e$.
\end{remas}

\begin{noth}{Groupoid closure}
We now construct a left adjoint to the inclusion $\iota$.
This will be called the {\em groupoid closure}, since it is analogous
to the transitive closure of an equivalence relation.
Let $P=(U,R,D,E,\dots)$ be a pregroupoid. We wish to enlarge
$R$ and~$D$ in a universal way so that the vertical maps in the
diagrams below become isomorphisms:
\[
\begin{array}{ccc}
\begin{tikzcd}[row sep=8mm]
D \ar[r,"c"] \ar[d,"{(p_1,p_2)}"'] & R \\
\underset{s,U,t}{R\times R} &
\end{tikzcd}
& \qquad &
\begin{tikzcd}[row sep=8mm]
E \ar[r,"\nu_1"] \ar[d,"{(q_{12},q_{23})}"'] & D \\
\underset{p_2,R,p_1}{D\times D} &
\end{tikzcd}
\end{array}
\]
With this idea in mind, we seek to define a new pregroupoid $P'$:
\begin{itemize}
\item $R'=(R\times_{s,U,t}R) \amalg R$ is the pushout:
\[
\begin{tikzcd}[row sep=10mm]
D \ar[r,"c"] \ar[d,"{(p_1,p_2)}"']
\arrow[rd, phantom, "\ulcorner"]
& R \ar[d,"\rho"] \\
\underset{s,U,t}{R\times R} \ar[r,"c'"] & R'.
\end{tikzcd}
\]
\item $i'=c'i\amalg\rho i:R'\to R'$ with
$c'i:R\times_{s,U,t}R \to R'$ and $\rho i:R\to R'$.
\item $D'=R\times_{s,U,t}R$ and
$i':D'\to D'$ is the inversion of $(R/U)^2$.
\item $E'=D\times_{p_2,R,p_1} D$.
\item $i'=\swap\circ (i,i):E'\to E'$ where $\swap$
swaps the two $D$ factors.
\item $s'=s\pr_2\amalg s:R'\to U$ with
$s\pr_2:R\times_{s,U,t}R \to U$ and $s:R\to U$.
\item $c':D'\to R'$ is the map in the pushout defining $R'$.
\item $e'=\rho e:U\to R'$.
\item $p'_1=\rho\pr_1:D'\to R'$.
\item $\lambda'=(p_1,p_2)\lambda$, $\mu'=(p_1,p_2)\mu$
as maps $R\to D\to D'$.
\item $q'_{12}=(p_1,p_2)\pr_1:E'\to D\to D'$.
\item $\nu'=(1,c):E'\to D'$.
\end{itemize}
Should $\lambda'$ and $\mu'$ be defined on $R'$ instead of
merely on $R$, the data $P'$ would be a pregroupoid, and the
four maps
\[
1:U\to U, \quad \rho:R\to R', \quad (p_1,p_2):D\to D',
\quad (q_{12},q_{23}):E\to E'
\]
would define a morphism of pregroupoids $P\to P'$. Nevertheless
we can define:
\begin{align*}
& \phi(P)=P' \\
& P_n=\phi^n(P)=(U,R_n,D_n,E_n) \\
& P^{\gpd}= \colim P_n.
\end{align*}
The underlying sets of $P^{\gpd}$ are $U^{\gpd}=U$,
$R^{\gpd}=\colim R_n$, $D^{\gpd}=\colim D_n$, $E^{\gpd}=\colim E_n$.
Passing to the limit, the maps $\lambda,\mu:R_n\to D_{n+1}$
yield maps $\lambda^{\gpd},\mu^{\gpd}:R^{\gpd}\to D^{\gpd}$
so that the problem concerning the domain of definition of
these maps disappears at infinity.
\end{noth}

\begin{rema}
It is possible to modify the definition of $P_n$ so as to have
$\lambda',\mu':R_n\to D_n$, making $P_n$ a pregroupoid.
For this, it is enough to replace $D'=R\times_{s,U,t}R$ by a
suitable subset of $R'\times_{s,U,t}R'$ where $c'$ can be
defined. The description of $D'$ is made a little cumbersome
by the fact that $R'\times_{s,U,t}R'$ is an amalgam of four sets.
Since this complication can be avoided by passing to the limit,
we preferred to do it this way.
\end{rema}

\begin{prop} \label{prop:UP of groupoid closure}
With notation as before, the collection $P^{\gpd}$ is a
groupoid. Moreover, the morphism $P\to P^{\gpd}$ is universal
for morphisms from $P$ to a groupoid. Thus the functor
$P\mapsto P^{\gpd}$ is left adjoint to the embedding
$(\Groupoid/U) \intoo (\Pregroupoid/U)$.
\end{prop}

\begin{proof}
The proof is straighforward; we merely give the idea. We start
from a pregroupoid
\[
P=\big(U,(R,i),(D,i),(E,i),s,c,e,p_1,\lambda,\mu,q_{12},\nu\big),
\]
a groupoid $\sP=(\sU,\sR,s,t,c)$ and a morphism of pregroupoids
$P\to \sP$. Let $\sD=\sR\times_{s,\sU,t}\sR$. We have a cube:
\[
\begin{tikzcd}[row sep=4mm,column sep=4mm]
& D \arrow[rr] \arrow[ld] \arrow[dd] & & R \arrow[dd] \arrow[ld] \\
\underset{s,U,t}{R\times R} \arrow[rr, crossing over] \arrow[dd] & &
R' & \\
& \sD \arrow[rr] \arrow[ld, equal] & & \sR \arrow[ld, equal] \\
\underset{s,\sU,t}{\sR\times \sR} \ar[rr] & & \sR \arrow[from=uu, crossing over, dashed] & \\
\end{tikzcd}
\]
By commutativity of the diagram in solid arrows, we can find a
dotted arrow completing the cube with an arrow from the pushout
$R'$. The contruction of a morphism $P'\to \sP$ proceeds along
the same lines. Iterating this construction gives morphisms
$P_n\to\sP$ for all $n$ and finally a morphism $P^{\gpd}\to \sP$.
\end{proof}

\begin{rema} \label{rmk:groupoid_closure}
The construction of the groupoid closure works similarly for
pregroupoids in objects of a category $\sC$ with the following
properties: $\sC$ has fibred products, pushouts, colimits indexed
by $\NN$, and the latter colimits commute with fibred products.
Examples of categories satisfying these properties are the category
of sets; the category of sheaves on a site; the category of
algebraic spaces \'etale over a fixed algebraic space $S$.
\end{rema}

We finish this Appendix with the 2-coequalizer property of
the quotient stack of a pregroupoid in algebraic spaces.
In \cite{SP19},
Tag~\href{http://stacks.math.columbia.edu/tag/044U}{044U}
such a property is stated but we need a statement which is
stronger in three respects: handling the category
$\sHom([U/R],\sX)$ and not just the set $\Hom([U/R],\sX)$;
proving equivalence and not just essential surjectivity;
including pregroupoids.

\begin{coro}\label{coro:coequalizer}
Let $S$ be an algebraic space. All notations being as
in~\ref{exam:motivating}, let
\[
P=(U,R,D,E,\dots)
\]
be a pregroupoid in algebraic spaces over $S$.
Assume that the groupoid closure
\[
P^{\gpd}=(U,R^{\gpd},D^{\gpd},E^{\gpd},\dots)
\]
exists and is an fppf groupoid (this holds for example if $P$
is a pregroupoid in \'etale algebraic spaces, or if $P$ is an
fppf groupoid).
Let $\pi:U\to [U/R]$ be the quotient stack of the groupoid
closure $P^{\gpd}$ and $\alpha\colon \pi s\to \pi t$ the
canonical 2-isomorphism, such that
$\alpha c=\alpha p_1\circ \alpha p_2$. For each $S$-stack
in groupoids $\sX$, let
\[
\eq \big(\!\!
\begin{tikzcd}[column sep=20]
\sHom(U,\sX) \ar[r,shift left,"s^*"]
\ar[r,shift right,"t^*"'] & \sHom(R,\sX)
\end{tikzcd}
\!\!\big)
\]
be the ``equalizer'' category described as follows:
\begin{itemize}
\item[\rm (i)]
objects are pairs $(f,\beta)$
composed of a $1$-morphism $f\colon U\to \sX$ and a
$2$-isomorphism $\beta\colon fs\to ft$ such that
$\beta c=\beta p_1\circ \beta p_2$.
\item[\rm (ii)] morphisms $(f_1,\beta_1)\to (f_2,\beta_2)$
are 2-isomorphisms $\varphi:f_1\to f_2$ such
that $\beta_2\circ \varphi s=\varphi t\circ\beta_1$.
\end{itemize}
Then the functor
\[
\begin{tikzcd}[column sep=20,row sep=-1]
\sHom([U/R],\sX) \ar[r] &
\eq \big(\sHom(U,\sX) \ar[r,shift left,"s^*"]
\ar[r,shift right,"t^*"'] & \sHom(R,\sX)\big) \\
\qquad\qquad\qquad\ g \ar[r,mapsto] & (f=g\pi,\beta=g\alpha) &
\end{tikzcd}
\]
is an equivalence of categories.
\end{coro}

Before we pass to the proof, here are pictures for the
2-morphisms $\beta$ and $\varphi$:
\[
\begin{tikzcd}[column sep=150]
D \arrow[r, shift left, "{fsp_2=fsc}",pos=0.39,bend left=35, ""{name=U1,below,pos=0.505}, ""{name=U2,below,pos=0.71}]
\ar[r, "{ftp_2=fsp_1}",pos=0.25,""{name=M1,above,pos=0.5},""{name=M2,below,pos=0.5}]
\arrow[r, shift right, "{ftp_1=ftc}" swap,pos=0.37,bend right=35, ""{name=D1,above,pos=0.495}, ""{name=D2,above,pos=0.7}]
& \sX
\arrow[Rightarrow, from=U1, to=M1,"\,\beta p_2"]
\arrow[Rightarrow, from=M2, to=D1,"\,\beta p_1"]
\arrow[Rightarrow, from=U2, to=D2, crossing over, "\,\beta c",pos=0.3]
\end{tikzcd}
\qquad\qquad\qquad
\begin{tikzcd}[column sep=30]
f_1s \ar[r,"\varphi s"] \ar[d,"\beta_1"'] & f_2s \ar[d,"\beta_2"] \\
f_1t \ar[r,"\varphi t"] & f_2t
\end{tikzcd}
\]

\begin{proof}
Set $\sH=\sH(P)=\sHom([U/R],\sX)$ and
$\sE=\sE(P)=\eq(\sHom(U,\sX) \toto \sHom(R,\sX))$. Let
$F:\sH\to\sE$ be the functor in the statement.

Suppose first that $P$ is a groupoid. For each pair $(f,\beta)$,
Lemma Tag~\href{http://stacks.math.columbia.edu/tag/044U}{044U}
in \cite{SP19} produces functorially a morphism $g:[U/R]\to\sX$
and a 2-isomorphism $\epsilon:g\pi\isomto f$. That is, we have
a functor $G:\sE\to\sH$ and an isomorphism $\epsilon:FG\isomto \id$.
Moreover the proof of {\em loc. cit.} shows that $GF$ is equal to
the identity; hence $F$ and $G$ are quasi-inverse equivalences.

Suppose now that $P$ is a pregroupoid with a groupoid closure
$P^{\gpd}$ which is an fppf groupoid. The morphism of
pregroupoids $P\to P^{\gpd}$ induces a functor $\sE(P^{\gpd})\to\sE(P)$
which we claim is an equivalence. To show this, note
that an object $(f,\beta)\in \sE(P)$ is the same thing as a
morphism of prestacks of pregroupoids from $P$ to $\sX$, namely:
\begin{itemize}
\item the map from $U$ to objects of $\sX$ is given by the
1-morphism $f$, namely each object $u\in U(T)$ is mapped to
$f(u)\in \sX(T)$,
\item the map from $R$ to arrows of $\sX$ is given by the
2-morphism $\beta$, namely each arrow $r\in R(T)$ is mapped
to the arrow $\beta(r):fs(r)\to ft(r)$ in $\sX(T)$,
\item the maps from $D$ to pairs of composable arrows of $\sX$,
and from $E$ to triples of composable arrows of $\sM$, are
determined by the previous ones because $\sX$ is a stack in
groupoids, see Remark~\ref{rema:gpds and pgpds}(2). Namely, the
former is $(\beta p_1,\beta p_2)$ and the latter is
$(\beta q_1,\beta q_2,\beta q_3)$,
\item the condition $\beta c=\beta p_1\circ \beta p_2$ ensures
that the map on arrows is compatible with composition; one sees
easily that is also implies compatibility with associativity.
\end{itemize}
Eventually the universal property of the groupoid closure
(Proposition~\ref{prop:UP of groupoid closure}) shows
that the functor $\sE(P^{\gpd})\to\sE(P)$ is an equivalence.
Since $\sH(P)\to \sE(P^{\gpd})$ is an equivalence, so is
$\sH(P)\to \sE(P)$.
\end{proof}

\addcontentsline{toc}{section}{References}

\bigskip

\noindent
Yuliang HUANG,
{\sc Univ Rennes, CNRS, IRMAR - UMR 6625, F-35000 Rennes, France} \\
Email address: {\tt yu-liang.huang@univ-rennes1.fr}

\bigskip

\noindent
Giulio ORECCHIA,
{\sc Univ Rennes, CNRS, IRMAR - UMR 6625, F-35000 Rennes, France} \\
Email address: {\tt giulio.orecchia@univ-rennes1.fr}

\bigskip

\noindent
Matthieu ROMAGNY,
{\sc Univ Rennes, CNRS, IRMAR - UMR 6625, F-35000 Rennes, France} \\
Email address: {\tt matthieu.romagny@univ-rennes1.fr}

\end{document}